\documentclass[a4paper,11pt]{article}

\usepackage{amsfonts} 
\usepackage{amsmath,hhline,latexsym,mathrsfs}
\usepackage{amssymb}
\usepackage{relsize}
\DeclareMathAlphabet\mathbfcal{OMS}{cmsy}{b}{n}

\usepackage[utf8]{inputenc}
\usepackage[T1]{fontenc}

\usepackage{textcomp}
\usepackage{mathtools,amssymb,amsthm}
\usepackage{lmodern}

\usepackage[a4paper, margin=2.5cm]{geometry}

\usepackage{graphicx}
\usepackage[dvipsnames]{xcolor}
\usepackage{array}
\usepackage{calc}
\usepackage{placeins}
\usepackage{titlesec}
\usepackage{titletoc}
\usepackage{fancyhdr}
\usepackage{titling}
\usepackage{enumitem,hhline,latexsym,epsfig}
\usepackage{eufrak}
\usepackage{accents}
\usepackage{bbm}
\usepackage{float}

\usepackage{hyperref}
\hypersetup{pdfstartview=XYZ}

\usepackage{caption}
\usepackage{tikz,tkz-tab}
\usetikzlibrary{math}
\usetikzlibrary{calc}

\theoremstyle{theorem}
\newtheorem{theorem}{Theorem}

\theoremstyle{theorem}
\newtheorem{proposition}{Proposition}

\theoremstyle{theorem}
\newtheorem{corollary}{Corollary}

\theoremstyle{theorem}
\newtheorem{lemma}{Lemma}

\theoremstyle{definition}
\newtheorem{definition}{Definition}

\theoremstyle{theorem}
\newtheorem{remark}{Remark}

\DeclareMathOperator{\Card}{Card}
\DeclareMathOperator{\Vect}{Vect}
\DeclareMathOperator{\Sp}{Sp}
\DeclareMathOperator{\Diag}{Diag}

\DeclareMathOperator{\diam}{diam}

\begin{document}

\title{Uniform observability of the one-dimensional wave equation for non-cylindrical domains. \\
Application to the control's support optimization.}

\author{\textsc{Arthur Bottois} \and \textsc{Nicolae Cîndea} \and \textsc{Arnaud Münch}\thanks{Université Clermont Auvergne, Laboratoire de Mathématiques Blaise Pascal CNRS-UMR 6620, Campus des Cézeaux, F-63178 Aubière cedex, France. \newline
\href{mailto:Arthur.Bottois@uca.fr}{Arthur.Bottois@uca.fr}, \href{mailto:Nicolae.Cindea@uca.fr}{Nicolae.Cindea@uca.fr} and \href{mailto:Arnaud.Munch@uca.fr}{Arnaud.Munch@uca.fr}}}

\maketitle

\begin{abstract}
This work is concerned with the distributed controllability of the one-dimensional wave equation over non-cylindrical domains. The controllability in that case has been obtained in \textit{[Castro-Cîndea-Münch, Controllability of the
  linear one-dimensional wave equation with inner moving forces, SIAM J. Control Optim 2014]} for domains satisfying the usual geometric optics condition. In the present work, we first show that the corresponding observability property holds true uniformly in a precise class of non-cylindrical domains. Within this class, we then consider, for a given initial datum, the problem of the optimization of the control support and prove its well-posedness. Numerical experiments are then discussed and highlight the influence of the initial condition on the optimal domain.
\end{abstract}

\section{Introduction}
\label{sect:intro}

This work is concerned with the distributed controllability of the one-dimensional wave equation. We define the space domain $\Omega = (0,1)$, the controllability time $T > 0$ and the space-time domain $Q_T = \Omega \times (0,T)$, with $\Sigma_T = \partial \Omega \times (0,T)$. Moreover, in the sequel we shall denote by $L = \partial_t^2 - \partial_x^2$ the one-dimensional wave operator.

The controllability problem for the one-dimensional wave equation reads as follows: for a given control domain $q \subset Q_T$, for every initial datum $(y_0,y_1) \in \mathbf{V} \coloneqq H^1_0(\Omega) \times L^2(\Omega)$, find a control $v \in L^2(q)$ such that the corresponding solution of the wave equation
\begin{equation}\label{eq:ctrl_wave}
    \left\{
    \begin{array}{ll}
        L y = v \mathbbm{1}_q & \text{in } Q_T, \\
        y = 0 & \text{on } \Sigma_T, \\
        (y,y_t)(\cdot,0) = (y_0,y_1) & \text{in } \Omega
    \end{array}
    \right.
\end{equation}
satisfies
\begin{equation}\label{eq:y_T}
    (y,y_t)(\cdot,T) = (0,0) \quad \text{in } \Omega.
\end{equation}
The application $\mathbbm{1}_q$ denotes the characteristic function of $q$. We recall that for every $(y_0,y_1) \in \mathbf{V}$ and $v \in L^2(q)$, there exists a unique solution $y$ to~\eqref{eq:ctrl_wave} with the regularity $y \in C([0,T]; H^1_0(\Omega)) \cap C^1([0,T]; L^2(\Omega))$ (see for instance~\cite{Lions88}).

In the cylindrical case, \emph{i.e.} when $q = \omega \times (0,T)$, with $\omega \subset \Omega$ an open non empty interval, the exact controllability of~\eqref{eq:ctrl_wave} holds for controllability time $T$ greater than a critical time $T^*$, related to the measure of the set $\Omega \setminus \overline{\omega}$. In the non-cylindrical case, the controllability of~\eqref{eq:ctrl_wave} has been established in~\cite{CastroCindeaMunch14} if the control domain $q$ satisfies the usual \emph{geometric optics condition} (we also refer to~\cite{Lebeau, Shao} for results in any dimension). We recall that a domain $q$ verifies such condition if every characteristic line of the wave equation, starting from a point of $\overline{\Omega} \times \{ 0 \}$ and following the laws of geometric optics when reflected on the boundary $\Sigma_T$, meets the domain $q$.

In both cylindrical and non-cylindrical cases, the controllability of~\eqref{eq:ctrl_wave} can be proven by the \emph{Hilbert uniqueness method} (HUM) introduced by J.-L. Lions~\cite{Lions88}. The main idea of this method is to obtain the controllability as a consequence of an observability inequality for the adjoint problem associated to~\eqref{eq:ctrl_wave}: there exists a constant $C_\text{obs}(q) > 0$ such that
\begin{equation}\label{eq:obs_W}
    \| (\varphi_0, \varphi_1) \|_\mathbf{W}^2 \leq C_\text{obs}(q) \| \varphi \|_{L^2(q)}^2, \quad \forall (\varphi_0, \varphi_1) \in \mathbf{W} \coloneqq L^2(\Omega) \times H^{-1}(\Omega),
\end{equation}
where $\varphi$ is the solution of the following homogeneous wave equation
\begin{equation}\label{eq:adj_wave}
    \left\{
    \begin{array}{ll}
        L \varphi = 0 & \text{in } Q_T, \\
        \varphi = 0 & \text{on } \Sigma_T, \\
        (\varphi, \varphi_t)(\cdot,0) = (\varphi_0, \varphi_1) & \text{in } \Omega.
    \end{array}
    \right.
\end{equation}
We recall that for every $(\varphi_0, \varphi_1) \in \mathbf{W}$, there exists a unique solution $\varphi$ to~\eqref{eq:adj_wave} (defined in the sense of transposition) with the regularity $\varphi \in C([0,T]; L^2(\Omega)) \cap C^1([0,T]; H^{-1}(\Omega))$ (see, for instance,~\cite{Lions88}). We emphasize that the observability constant $C_\text{obs}$ appearing in~\eqref{eq:obs_W} depends on the observation domain $q$.

According to the HUM method, the control of minimal $L^2(q)$-norm is obtained as the restriction to $q$ of the solution $\varphi$ of~\eqref{eq:adj_wave} corresponding to the initial datum $(\varphi_0, \varphi_1)$ which minimize the functional
\begin{equation}\label{eq:J*}
    \mathcal{J}^\star(\varphi_0, \varphi_1) = \frac{1}{2} \iint_q \varphi^2 - \langle \varphi_1, y_0 \rangle_{H^{-1}(\Omega), H^1_0(\Omega)} + \langle \varphi_0, y_1 \rangle_{L^2(\Omega)} \quad \forall (\varphi_0, \varphi_1) \in \mathbf{W}.
\end{equation}
The existence and uniqueness of the minimum of the functional $\mathcal{J}^\star$ over the space $\mathbf{W}$ are mainly consequences of the observability inequality~\eqref{eq:obs_W}.

In the first part of this work, we provide a class of observation domains, based on the geometric condition considered in~\cite[Proposition~2.1]{CastroCindeaMunch14}, for which the observability constant $C_\text{obs}$ in~\eqref{eq:obs_W} is uniformly bounded. More precisely, for every $\varepsilon > 0$ small enough, we define the $\varepsilon$-interior of $q$ by
\begin{equation}\label{eq:q_eps}
    q^\varepsilon = \big\{ (x,t) \in q; \quad d((x,t), \partial q) > \varepsilon \big\},
\end{equation}
and the admissible set of control domains by
\begin{equation}\label{eq:Qad}
    \mathcal{Q}_\text{ad}^\varepsilon = \big\{ q \subset Q_T; \quad q \text{ open and } q^\varepsilon \text{ verifies the geometric optics condition} \big\}.
\end{equation}
We prove the following uniform observability inequality: there exists a constant $C_\text{obs}^\varepsilon$ such that for every $q \in \mathcal{Q}_\text{ad}^\varepsilon$,
\begin{equation}\label{eq:unif_obs_W}
    \| (\varphi_0, \varphi_1) \|_\mathbf{W}^2 \leq C_\text{obs}^\varepsilon \| \varphi \|_{L^2(q)}^2, \quad \forall (\varphi_0, \varphi_1) \in \mathbf{W},
\end{equation}
where $\varphi$ is the solution of~\eqref{eq:adj_wave} associated to the initial datum $(\varphi_0, \varphi_1)$.

This uniform property then allows, in a second part, to analyze the problem of the optimal distribution of the control domain $q$. Precisely, we consider controls acting on a horizontal neighborhood of a regular curve: for a given half-weight $\delta_0 > 0$, we define the domain associated to the curve $\gamma : (0,T) \to \Omega$ by
\begin{equation}\label{eq:q_gmm}
    q_\gamma = \big\{ (x,t) \in Q_T; \quad |x - \gamma(t)| < \delta_0 \big\}.
\end{equation}
 The curves $\gamma$ are chosen in the following set
\begin{equation}\label{eq:Gad}
    \mathcal{G}_\text{ad} = \big\{ \gamma \in W^{1,\infty}(0,T); \quad \| \gamma' \|_{L^\infty(0,T)} \leq M, \quad \delta_0 \leq \gamma \leq 1 - \delta_0 \big\}
\end{equation}
consisting of uniformly Lipschitz functions of fixed constant $M > 0$. For $T \geq 2$ and $\varepsilon > 0$ small enough, the class $\{ q_\gamma;\ \gamma \in \mathcal{G}_\text{ad} \}$ is a subset of $\mathcal{Q}_\text{ad}^\varepsilon$. The optimization problem we shall consider reads as follows: for a given initial datum $(y_0,y_1) \in \mathbf{V}$, solve
\begin{equation}\label{eq:opt_pb}
    \inf_{\gamma \in \mathcal{G}_\text{ad}} \| v \|_{L^2(q_\gamma)}^2,
\end{equation}
where $v$ is the control of minimal $L^2(q_{\gamma})$-norm distributed over $q_\gamma \subset Q_T$.

Controllability of partial differential equations by means of moving controls, although less studied than the cylindrical case, becomes more and more popular in the literature. One of the first contributions for the wave equation is due to Khapalov~\cite{Khapalov95}. The author proved an observability inequality for the one-dimensional wave equation with moving pointwise observation. This time-dependent observation allows to avoid the issue of strategic observation points and to get uniform controllability. More recently, the works~\cite{Castro13, CuiLiuGao13, HaakHoang19} addressed the controllability for the one-dimensional case. The controllability of the one-dimensional wave equation is proved for controls acting on an interior curve and on a moving boundary, by using d'Alembert's formula and the multiplier method respectively. For the $N$-dimensional case, in~\cite{LiuYong99} the authors employ the multiplier method to prove that the wave equation is controllable using a control acting on a time-dependent domain $q$, under the hypothesis that this domain covers the whole space domain before the control time $T$. Under similar hypotheses, we also mention the work~\cite{MartinRosierRouchon13} where the control of the damped wave equation $y_{tt} - y_{xx} - \varepsilon y_{txx} = 0$ defined on the 1D torus is obtained in a non-cylindrical case. Because of the presence of an essential spectrum, such property does not hold true in the cylindrical case. Moreover, assuming the standard geometric optics condition, the observability inequality has been obtained in~\cite{CastroCindeaMunch14} in dimension one by way of d'Alembert's formula, and extended in~\cite{Lebeau} to the multi-dimensional case using microlocal analysis. It is also worth mentioning the obtention of Carleman type inequality for general hyperbolic equations in~\cite{Shao}.

On the other hand, in the cylindrical situation, the uniform observability property for the wave equation with respect to the observation domain is addressed in~\cite{Periago09}. For $T \geq 2$, the author proves, using Fourier series, a uniform observability inequality for domains of the form $q = \omega \times (0,T)$, with $\omega \subset \Omega$ an open set of fixed length. The uniform property is then employed to analyze the optimal position of the support of the corresponding null control. This problem of the optimal shape and position of the support is also numerically investigated in~\cite{Munch08, Munch09} for the one and two dimensional wave equation. In a similar context, we also mention~\cite{Humbert} and the references therein.

This paper is organized as follows. In Section~\ref{sect:unif_obs} we prove the uniform observability inequality~\eqref{eq:unif_obs_W} on $\mathcal{Q}_\text{ad}^\varepsilon$ and its variant on the subset $\mathcal{G}_\text{ad}$. This is achieved by defining an appropriate decomposition of the observation domains in $\mathcal{Q}_\text{ad}^\varepsilon$, and by using d'Alembert's formula. The proof also relies on arguments from graph theory. Then, in Section~\ref{sect:shp_opt}, following arguments from~\cite{HenrotPierre05, Periago09}, we analyze a variant of the extremal problem~\eqref{eq:opt_pb}. Introducing a $C^1$-regularization of the support $q_\gamma$, we prove that the underlying cost is continuous over $\mathcal{G}_\text{ad}$ for the $L^\infty(0,T)$-norm, and admits at least one local minimum. Section~\ref{sect:num_simu} is concerned with numerical experiments. Minimization sequence for the regularized cost are constructed using a gradient method: each iteration requires the computation of a null control, performed using the space-time formulation developed in \cite{CindeaMunch15} and used in \cite{CastroCindeaMunch14}, and well-suited to the description of the non-cylindrical domains, where the control acts.

\section{Uniform observability with respect to the domain of observation}
\label{sect:unif_obs}

We prove in this section the uniform observability inequality~\eqref{eq:unif_obs_W} with respect to the domain of observation. Precisely, we prove the following equivalent result for regular data in $\mathbf{V}$.

\begin{theorem}\label{thm:unif_obs}
Let $T > 0$ and let $\varepsilon > 0$ be a small enough fixed parameter such that the set $\mathcal{Q}_\text{ad}^\varepsilon$ defined by~\eqref{eq:Qad} is non-empty. There exists a constant $C_\text{obs}^\varepsilon > 0$ such that for every $q \in \mathcal{Q}_\text{ad}^\varepsilon$, the following inequality holds
\begin{equation}\label{eq:unif_obs_V}
    \| (\varphi_0, \varphi_1) \|_\mathbf{V}^2 \leq C_\text{obs}^\varepsilon \| \varphi_t \|_{L^2(q)}^2, \quad \forall (\varphi_0, \varphi_1) \in \mathbf{V},
\end{equation}
where $\varphi$ is the solution of the wave equation~\eqref{eq:adj_wave} associated to the initial datum $(\varphi_0, \varphi_1)$.
\end{theorem}

In the remaining part of this section, we assume that the hypotheses of Theorem~\ref{thm:unif_obs} are satisfied.

\subsection{Some notations and technical lemmas}
\label{ssect:unif_obs_prelim}

We first introduce some notations and state some preliminary lemmas.

Let $N > 0$ be an integer and $\kappa_N = 1/N$. We denote $S_N = (x_i^N)_{0 \leq i \leq N}$ a regular subdivision of $\overline{\Omega}$ in $N$ intervals, \emph{i.e.} for every $i \in \{ 0,\dots,N \}$, we set $x_i^N = i/N$. Associated to the functions $\varphi_0 \in H^1_0(\Omega)$ and $\varphi_1 \in L^2(\Omega)$, we define the continuous function $\varphi_0^N$ affine on intervals $[x_{i-1}^N, x_i^N]$ of the subdivision $S_N$ and the function $\varphi_1^N$ constant on intervals of $S_N$:
\begin{align}
\label{eq:phi0_N}
    \varphi_0^N(x) & = \sum_{i=1}^N \left( \varphi_0(x_i^N) \frac{x - x_{i-1}^N}{\kappa_N} + \varphi_0(x_{i-1}^N) \frac{x_i^N - x}{\kappa_N} \right) \mathbbm{1}_{[x_{i-1}^N, x_i^N]}(x), \\
\label{eq:phi1_N}
    \varphi_1^N(x) & = \sum_{i=1}^N \beta_i^N \mathbbm{1}_{[x_{i-1}^N, x_i^N]}(x), \quad \text{with } \beta_i^N = \frac{1}{\kappa_N} \int_{x_{i-1}^N}^{x_i^N} \varphi_1.
\end{align}
We also denote by $(\varphi_0^N)' \in L^2(\Omega)$ the ``derivative'' of $\varphi_0^N$:
\begin{equation}\label{eq:phi0'_N}
    (\varphi_0^N)'(x) = \sum_{i=1}^N \alpha_i^N \mathbbm{1}_{[x_{i-1}^N, x_i^N]}(x), \quad \text{with } \alpha_i^N = \frac{\varphi_0(x_i^N) - \varphi_0(x_{i-1}^N)}{\kappa_N}.
\end{equation}
Using that $\varphi_0 \in H^1_0(\Omega)$, we then easily check that
\begin{equation}\label{eq:phi01_N_nrm}
    \| \varphi_0^N \|_{H^1_0(\Omega)}^2 = \frac{1}{N} \sum_{i=1}^N (\alpha_i^N)^2, \quad \| \varphi_1^N \|_{L^2(\Omega)}^2 = \frac{1}{N} \sum_{i=1}^N (\beta_i^N)^2 \quad \text{and} \quad \sum_{i=1}^N \alpha_i^N = 0.
\end{equation}

In order to use d'Alembert's formula for the solution of the wave equation~\eqref{eq:adj_wave} associated to initial datum $(\varphi_0^N, \varphi_1^N)$ of the form~\eqref{eq:phi0_N}-\eqref{eq:phi1_N}, we need to extend these functions to odd functions to $[-1,1]$ and then by $2$-periodicity to $\mathbb{R}$. In this respect, we first extend the definition of $x_i^N$ to $i \in \mathbb{Z}$ by putting $x_i^N = i/N$ for every $i \in \mathbb{Z}$, and then denote by $I_i^N$ for every $i \in \mathbb{Z}^*$ the following interval:
\begin{equation}\label{eq:Ii}
    I_i^N =
    \left\{
    \begin{array}{ll}
        [x_{i-1}^N, x_i^N] & \text{if } i > 0, \\[2mm]
        [x_i^N, x_{i+1}^N] & \text{if } i < 0.
    \end{array}
    \right.
\end{equation}
Similarly, we extend $\alpha_i^N$ and $\beta_i^N$ for every $i \in \mathbb{Z}^*$ as follows: if $i \in \{ -N,\dots,-1 \}$, we set $\alpha_i^N = \alpha_{-i}^N$ and $\beta_i^N = -\beta_{-i}^N$; if $|i| > N$, the definitions of $\alpha_i^N$ and $\beta_i^N$ are a little more complex:
\begin{equation*}
    \alpha_i^N = \alpha_{\mathfrak{j}_N(i)}^N, \qquad \beta_i^N = \beta_{\mathfrak{j}_N(i)}^N,
\end{equation*}
with $\mathfrak{j}_N(i)$ defined for every integer $i \geq 1$ by
\begin{equation}\label{eq:j}
    \mathfrak{j}_N(i) =
    \left\{
    \begin{array}{ll}
        (i-1) \mod(2N) + 1 & \text{if } (i-1) \mod(2N) < N, \\
        (i-1) \mod(2N) - 2N & \text{if } (i-1) \mod(2N) \geq N,
    \end{array}
    \right.
\end{equation}
and $\mathfrak{j}_N(i) = -\mathfrak{j}_N(-i)$ if $i \leq -1$. Remark that for every $i \in \mathbb{Z}^*$, we have $\mathfrak{j}_N(i) \in \mathbb{I}_N$ with the set $\mathbb{I}_N$ given by
\begin{equation}\label{eq:IN}
    \mathbb{I}_N = \{ -N,\dots,-1, 1,\dots,N \}.
\end{equation}
For $i,j \in \mathbb{I}_N$, we also define
\begin{equation}\label{eq:gmm_N}
    \gamma_i^N = \alpha_i^N + \beta_i^N.
\end{equation}

We extend the functions $\varphi_0^N$ and $\varphi_1^N$ to odd functions on $[-1,1]$ and by $2$-periodicity to $\mathbb{R}$. Then, using the notations above, we obtain
\begin{equation}\label{eq:phi01_N_ext}
    (\varphi_0^N)'(x) = \sum_{i \in \mathbb{Z}^*} \alpha_i^N \mathbbm{1}_{I_i^N}(x), \qquad \varphi_1^N(x) = \sum_{i \in \mathbb{Z}^*} \beta_i^N \mathbbm{1}_{I_i^N}(x), \qquad \forall x \in \mathbb{R}.
\end{equation}
Furthermore, from d'Alembert's formula, the solution $\varphi^N$ of~\eqref{eq:adj_wave} associated to the initial datum $(\varphi_0^N, \varphi_1^N)$ is given as follows:
\begin{equation}\label{eq:dAlembert}
    \varphi^N(x,t) = \frac{1}{2} \Big( \varphi_0^N(x+t) + \varphi_0^N(x-t) \Big) + \frac{1}{2} \int_{x-t}^{x+t} \varphi_1^N, \quad \forall (x,t) \in Q_T.
\end{equation}
Taking the derivative with respect to $t$ and replacing the expressions~\eqref{eq:phi01_N_ext} in the above equation, we deduce that for all $(x,t) \in Q_T$, we have
\begin{align}
\nonumber
    \varphi^N_t(x,t) & = \frac{1}{2} \Big( (\varphi_0^N)'(x+t) - (\varphi_0^N)'(x-t) + \varphi_1^N(x+t) + \varphi_1^N(x-t) \Big) \\
\nonumber
    & = \frac{1}{2} \sum_{i \in \mathbb{Z}^*} \Big( (\alpha_i^N + \beta_i^N) \mathbbm{1}_{I_i^N}(x+t) - (\alpha_i^N - \beta_i^N) \mathbbm{1}_{I_i^N}(x-t) \Big) \\
\label{eq:phit_N_1}
    & = \frac{1}{2} \sum_{i \in \mathbb{Z}^*} \sum_{j \in \mathbb{Z}^*} (\alpha_i^N + \beta_i^N - \alpha_j^N + \beta_j^N) \mathbbm{1}_{I_i^N}(x+t) \mathbbm{1}_{I_j^N}(x-t).
\end{align}
Using the properties of the function $\mathfrak{j}_N$ defined in~\eqref{eq:j}, we deduce that for $i,j \in \mathbb{Z}^*$,
\begin{align}
\nonumber
    \alpha_i^N + \beta_i^N - \alpha_j^N + \beta_j^N & = \alpha_{\mathfrak{j}_N(i)}^N + \beta_{\mathfrak{j}_N(i)}^N - \alpha_{\mathfrak{j}_N(j)}^N + \beta_{\mathfrak{j}_N(j)}^N \\
\nonumber
    & = (\alpha_{\mathfrak{j}_N(i)}^N + \beta_{\mathfrak{j}_N(i)}^N) - (\alpha_{-\mathfrak{j}_N(j)}^N + \beta_{-\mathfrak{j}_N(j)}^N) \\
\label{eq:phit_N_2}
    & = \gamma_{\mathfrak{j}_N(i)}^N - \gamma_{-\mathfrak{j}_N(j)}^N.
\end{align}

In view of the expression~\eqref{eq:phit_N_1}, we introduce the following definition.

\begin{definition}\label{def:Cij}
For every $i,j \in \mathbb{Z}^*$, the \emph{elementary square of indices $(i,j)$} associated to the subdivision $S_N$ is defined as the following closed set of $\mathbb{R}^2$:
\begin{equation}\label{eq:Cij}
    C_{(i,j)}^N = \big\{ (x,t) \in \mathbb{R}^2 \text{ such that } x + t \in I_i^N \text{ and } x - t \in I_j^N \big\},
\end{equation}
where for every $i \in \mathbb{Z}^*$, the interval $I_i^N$ is given by~\eqref{eq:Ii}. We denote by $\mathcal{C}_N = \{ C_{(i,j)}^N;\ i,j \in \mathbb{Z}^* \}$ the set of all the elementary squares associated to the subdivision $S_N$ of $\overline{\Omega}$. It is easy to see that $\mathbb{R}^2 = \bigcup_{i,j \in \mathbb{Z}^*} C_{(i,j)}^N$.
\end{definition}

Figure~\ref{fig:CN_Q} illustrates the way the elementary squares are indexed, using elementary squares associated to the subdivision $S_4$ of $\overline{\Omega}$.

\noindent
\begin{minipage}{0.48\textwidth}
\begin{figure}[H]
\centering
\begin{tikzpicture}[scale=2.7]
\tikzmath{\T = 2;}

\coordinate (i) at (1,0);
\coordinate (j) at (0,1);
\coordinate (ij) at (1,1);



\draw[red,line width=1] (0,0) -- (1,1);
\draw[red,line width=1] (0.25,0) -- (1,0.75);
\draw[red,line width=1] (0.5,0) -- (1,0.5);
\draw[red,line width=1] (0.75,0) -- (1,0.25);

\draw[red,line width=1] (0.25,0) -- (0,0.25);
\draw[red,line width=1] (0.5,0) -- (0,0.5);
\draw[red,line width=1] (0.75,0) -- (0,0.75);
\draw[red,line width=1] (1,0) -- (0,1);

\draw[red,line width=1] (0,0.25) -- (1,1.25);
\draw[red,line width=1] (0,0.5) -- (1,1.5);
\draw[red,line width=1] (0,0.75) -- (1,1.75);
\draw[red,line width=1] (0,1) -- (1,2);

\draw[red,line width=1] (1,0.25) -- (0,1.25);
\draw[red,line width=1] (1,0.5) -- (0,1.5);
\draw[red,line width=1] (1,0.75) -- (0,1.75);
\draw[red,line width=1] (1,1) -- (0,2);

\draw[red,line width=1] (0,1.25) -- (0.75,2);
\draw[red,line width=1] (0,1.5) -- (0.5,2);
\draw[red,line width=1] (0,1.75) -- (0.25,2);

\draw[red,line width=1] (1,1.25) -- (0.25,2);
\draw[red,line width=1] (1,1.5) -- (0.5,2);
\draw[red,line width=1] (1,1.75) -- (0.75,2);


\draw[line width=2] (0,0) rectangle (1,\T);
\path (0,0) -- (1,0) node[midway,below]{$x$};
\path (0,0) -- (0,\T) node[midway,left]{$t$};
\draw[line width=2] (0,0) -- ++($-0.025*(j)$) node[below]{$0$};
\draw[line width=2] (1,0) -- ++($-0.025*(j)$) node[below]{$1$};
\draw[line width=2] (0,0) -- ++($-0.025*(i)$) node[left]{$0$};
\draw[line width=2] (0,\T) -- ++($-0.025*(i)$) node[left]{$T = 2$};


\draw[->,>=latex,line width=1] (1.1,0.375) node[right]{$C_{(4,1)}^4$} -- (0.5,0.375);
\draw[->,>=latex,line width=1] (1.1,0.75) node[right]{$C_{(7,1)}^4$} -- (0.875,0.75);
\draw[->,>=latex,line width=1] (1.1,1.125) node[right]{$C_{(7,-3)}^4$} -- (0.5,1.125);
\end{tikzpicture}
\caption{Some elementary squares in $\mathcal{C}_4$.}
\label{fig:CN_Q}
\end{figure}
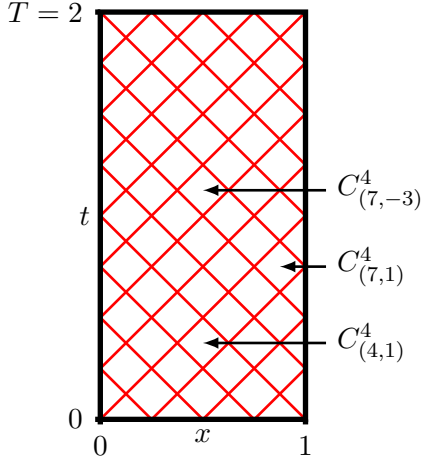
\end{minipage}
\hfill
\begin{minipage}{0.48\textwidth}
\begin{figure}[H]
\centering
\begin{tikzpicture}[scale=2.7]
\tikzmath{\T = 2;}

\coordinate (i) at (1,0);
\coordinate (j) at (0,1);
\coordinate (ij) at (1,1);



\path[fill=red!25] (0.375,0) -- ++(0.0625,0.0625) -- ++(0.0625,-0.0625) -- ++(0.0625,0.0625) -- ++(0.0625,-0.0625) -- ++(0.0625,0.0625) -- ++(-0.0625,0.0625) -- ++(0.0625,0.0625) -- ++(-0.0625,0.0625) -- ++(0.0625,0.0625) -- ++(-0.0625,0.0625) -- ++(0.0625,0.0625) -- ++(-0.0625,0.0625) -- ++(0.0625,0.0625) -- ++(-0.0625,0.0625) -- ++(0.0625,0.0625) -- ++(-0.0625,0.0625) -- ++(0.0625,0.0625) -- ++(-0.0625,0.0625) -- ++(0.0625,0.0625) -- ++(-0.0625,0.0625) -- ++(0.0625,0.0625) -- ++(-0.0625,0.0625) -- ++(0.0625,0.0625) -- ++(-0.0625,0.0625) -- ++(0.0625,0.0625) -- ++(-0.0625,0.0625) -- ++(0.0625,0.0625) -- ++(-0.0625,0.0625) -- ++(0.0625,0.0625) -- ++(-0.0625,0.0625) -- ++(0.0625,0.0625) -- ++(-0.0625,0.0625) -- ++(0.0625,0.0625) -- ++(-0.0625,0.0625) -- ++(0.0625,0.0625) -- ++(-0.0625,0.0625) -- ++(-0.0625,-0.0625) -- ++(-0.0625,0.0625) -- ++(-0.0625,-0.0625) -- ++(-0.0625,0.0625) -- ++(-0.0625,-0.0625) -- ++(0.0625,-0.0625) -- ++(-0.0625,-0.0625) -- ++(0.0625,-0.0625) -- ++(-0.0625,-0.0625) -- ++(0.0625,-0.0625) -- ++(-0.0625,-0.0625) -- ++(0.0625,-0.0625) -- ++(-0.0625,-0.0625) -- ++(0.0625,-0.0625) -- ++(-0.0625,-0.0625) -- ++(0.0625,-0.0625) -- ++(-0.0625,-0.0625) -- ++(0.0625,-0.0625) -- ++(-0.0625,-0.0625) -- ++(0.0625,-0.0625) -- ++(-0.0625,-0.0625) -- ++(0.0625,-0.0625) -- ++(-0.0625,-0.0625) -- ++(0.0625,-0.0625) -- ++(-0.0625,-0.0625) -- ++(0.0625,-0.0625) -- ++(-0.0625,-0.0625) -- ++(0.0625,-0.0625) -- ++(-0.0625,-0.0625) -- ++(0.0625,-0.0625) -- ++(-0.0625,-0.0625) -- ++(0.0625,-0.0625) -- ++(-0.0625,-0.0625) -- ++(0.0625,-0.0625) -- ++(-0.0625,-0.0625) -- ++(0.0625,-0.0625);


\draw[red,line width=1] (0,0) -- (1,1);
\draw[red,line width=1] (0.125,0) -- (1,0.875);
\draw[red,line width=1] (0.25,0) -- (1,0.75);
\draw[red,line width=1] (0.375,0) -- (1,0.625);
\draw[red,line width=1] (0.5,0) -- (1,0.5);
\draw[red,line width=1] (0.625,0) -- (1,0.375);
\draw[red,line width=1] (0.75,0) -- (1,0.25);
\draw[red,line width=1] (0.875,0) -- (1,0.125);

\draw[red,line width=1] (0.125,0) -- (0,0.125);
\draw[red,line width=1] (0.25,0) -- (0,0.25);
\draw[red,line width=1] (0.375,0) -- (0,0.375);
\draw[red,line width=1] (0.5,0) -- (0,0.5);
\draw[red,line width=1] (0.625,0) -- (0,0.625);
\draw[red,line width=1] (0.75,0) -- (0,0.75);
\draw[red,line width=1] (0.875,0) -- (0,0.875);
\draw[red,line width=1] (1,0) -- (0,1);

\draw[red,line width=1] (0,0.125) -- (1,1.125);
\draw[red,line width=1] (0,0.25) -- (1,1.25);
\draw[red,line width=1] (0,0.375) -- (1,1.375);
\draw[red,line width=1] (0,0.5) -- (1,1.5);
\draw[red,line width=1] (0,0.625) -- (1,1.625);
\draw[red,line width=1] (0,0.75) -- (1,1.75);
\draw[red,line width=1] (0,0.875) -- (1,1.875);
\draw[red,line width=1] (0,1) -- (1,2);

\draw[red,line width=1] (1,0.125) -- (0,1.125);
\draw[red,line width=1] (1,0.25) -- (0,1.25);
\draw[red,line width=1] (1,0.375) -- (0,1.375);
\draw[red,line width=1] (1,0.5) -- (0,1.5);
\draw[red,line width=1] (1,0.625) -- (0,1.625);
\draw[red,line width=1] (1,0.75) -- (0,1.75);
\draw[red,line width=1] (1,0.875) -- (0,1.875);
\draw[red,line width=1] (1,1) -- (0,2);

\draw[red,line width=1] (0,1.125) -- (0.875,2);
\draw[red,line width=1] (0,1.25) -- (0.75,2);
\draw[red,line width=1] (0,1.375) -- (0.625,2);
\draw[red,line width=1] (0,1.5) -- (0.5,2);
\draw[red,line width=1] (0,1.625) -- (0.375,2);
\draw[red,line width=1] (0,1.75) -- (0.25,2);
\draw[red,line width=1] (0,1.875) -- (0.125,2);

\draw[red,line width=1] (1,1.125) -- (0.125,2);
\draw[red,line width=1] (1,1.25) -- (0.25,2);
\draw[red,line width=1] (1,1.375) -- (0.375,2);
\draw[red,line width=1] (1,1.5) -- (0.5,2);
\draw[red,line width=1] (1,1.625) -- (0.625,2);
\draw[red,line width=1] (1,1.75) -- (0.75,2);
\draw[red,line width=1] (1,1.875) -- (0.875,2);


\draw[blue,line width=2] (0.3125,0) -- (0.3125,2);
\draw[blue,line width=2] (0.6875,0) -- (0.6875,2);

\draw[blue,dotted,line width=2] (0.4625,0.15) -- (0.5375,0.15) -- (0.5375,1.85) -- (0.4625,1.85) -- (0.4625,0.15);


\draw[line width=2] (0,0) rectangle (1,\T);
\path (0,0) -- (1,0) node[midway,below]{$x$};
\path (0,0) -- (0,\T) node[midway,left]{$t$};
\draw[line width=2] (0,0) -- ++($-0.025*(j)$) node[below]{$0$};
\draw[line width=2] (1,0) -- ++($-0.025*(j)$) node[below]{$1$};
\draw[line width=2] (0,0) -- ++($-0.025*(i)$) node[left]{$0$};
\draw[line width=2] (0,\T) -- ++($-0.025*(i)$) node[left]{$T = 2$};


\draw[->,>=latex,line width=1] (1.1,0.375) node[right]{$q^\varepsilon$} -- (0.5375,0.375);
\draw[->,>=latex,line width=1] (1.1,0.5625) node[right]{$R_8(q)$} -- (0.625,0.5625);
\draw[->,>=latex,line width=1] (1.1,0.75) node[right]{$q$} -- (0.6875,0.75);
\end{tikzpicture}
\caption{Cover $R_8(q)$ of $q^\varepsilon$, for $\varepsilon = 0.15$.}
\label{fig:RN_q}
\end{figure}
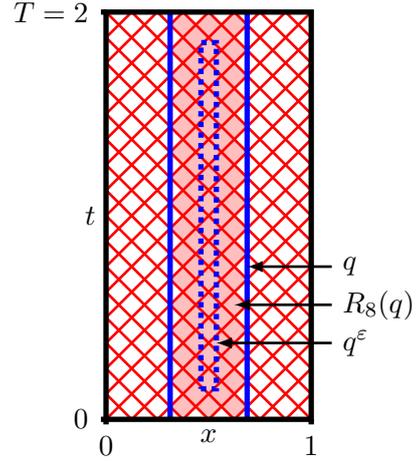
\end{minipage} \\

\begin{remark}\label{rmk:Cij}
For every $i,j \in \mathbb{Z}^*$, the coordinates of the center of the elementary square $C_{(i,j)}^N$ associated to the subdivision $S_N$ are given by
\begin{equation}\label{eq:Cij_ctr}
    \left\{
    \begin{array}{l}
        x_{(i,j)}^N = \frac{m_i^N + m_j^N}{2}, \\
        t_{(i,j)}^N = \frac{m_i^N - m_j^N}{2},
    \end{array}
    \right.
    \quad \text{with } m_i^N =
    \left\{
    \begin{array}{ll}
        \frac{x_{i-1}^N + x_i^N}{2} & \text{if } i > 0, \\
        \frac{x_i^N + x_{i+1}^N}{2} & \text{if } i < 0.
    \end{array}
    \right.
\end{equation}
The area of every elementary square $C_{(i,j)}^N \in \mathcal{C}_N$ is given by $|C_{(i,j)}^N| = \frac{1}{2 N^2}$. Notice that for every $i,j \in \mathbb{Z}^*$ with $|i|,|j| > 1$, the elementary squares having one side in common with the elementary square $C_{(i,j)}^N$ are $C_{(i \pm 1,j)}^N$ and $C_{(i,j \pm 1)}^N$.
\end{remark}

\begin{definition}\label{def:CN_q}
For every $q \in \mathcal{Q}_\text{ad}^\varepsilon$, we denote by $\mathcal{C}_N(q)$ and $\mathcal{C}_N(Q_T)$ the sets of the elementary squares in $\mathcal{C}_N$ with their interior included in $q$ and $Q_T$ respectively:
\begin{equation}\label{eq:CN_q}
    \mathcal{C}_N(q) = \big\{ C_{(i,j)}^N \in \mathcal{C}_N; \quad \accentset{\circ}{C}_{(i,j)}^N \subset q \big\}, \quad \mathcal{C}_N(Q_T) = \big\{ C_{(i,j)}^N \in \mathcal{C}_N; \quad \accentset{\circ}{C}_{(i,j)}^N \subset Q_T \big\}.
\end{equation}
If $N$ is large enough, the sets $\mathcal{C}_N(q)$ and $\mathcal{C}_N(Q_T)$ are non-empty. We also define $R_N(q)$ the union of the elementary squares in $\mathcal{C}_N(q)$:
\begin{equation}\label{eq:RN_q}
    R_N(q) = \overbrace{\bigcup_{C_{(i,j)}^N \in \mathcal{C}_N(q)} C_{(i,j)}^N}^\circ.
\end{equation}
\end{definition}

With these notations, we can now prove the following lemma.

\begin{lemma}\label{lem:RN_q}
Let $N > 1/\varepsilon$ be a fixed integer. For every $q \in \mathcal{Q}_\text{ad}^\varepsilon$, the set $\bigcup_{C_{(i,j)}^N \in \mathcal{C}_N(q)} C_{(i,j)}^N$ is a cover of $q^\varepsilon$ given by~\eqref{eq:q_eps}.
Moreover, the set $R_N(q)$ defined by~\eqref{eq:RN_q} satisfies $q^\varepsilon \subset R_N(q) \subset q$.
\end{lemma}

\begin{proof}
Let $X \in q^\varepsilon$. Using the definition of $q^\varepsilon$, we have $X \in q$ and $d(X, \partial q) > \varepsilon$. Since $\mathbb{R}^2$ is covered by squares in $\mathcal{C}_N$, there exists $C_{(i,j)}^N \in \mathcal{C}_N$ such that $X \in C_{(i,j)}^N$. Moreover, since $\diam(C_{(i,j)}^N) = \kappa_N$, we have $C_{(i,j)}^N \subset \overline{B}(X,\kappa_N)$. Let $Y \in \overline{B}(X,\kappa_N)$. Then, for every $Z \in \mathbb{R}^2\setminus q$, it holds that
\begin{equation*}
    d(Y,Z) \geq |d(Y,X) - d(X,Z)| > \varepsilon - \kappa_N > 0.
\end{equation*}
Consequently, $d(Y,\mathbb{R}^2 \setminus q) > 0$, which implies $Y \in q$. Therefore, $C_{(i,j)}^N \subset \overline{B}(X,\kappa_N) \subset q$ and, finally, $C_{(i,j)}^N \in \mathcal{C}_N(q)$.
\end{proof}

Figure~\ref{fig:RN_q} illustrates Lemma~\ref{lem:RN_q} in the case of the cylindrical observation domain $q = (\frac{5}{16}, \frac{11}{16}) \times (0,2)$, for $\varepsilon = 0.15$ and $N = 8$. In order to write several expressions in a simpler form, we use the following graph theory framework.

\begin{definition}\label{def:GN_q}
Let $q \in \mathcal{Q}_\text{ad}^\varepsilon$ an observation domain. We define the weighted graph $G_N(q)$ as follows:
\begin{itemize}
    \item[\textbullet] $\mathbb{I}_N$ given by~\eqref{eq:IN} is the set of vertices;
    \item[\textbullet] for every $i \in \mathbb{I}_N$, the degree of the vertex $i$ is given by:
    \begin{equation*}
        d_i^N = \Card \left( \big\{ C_{(k,-l)}^N \in \mathcal{C}_N(q); \quad i \in \{ \mathfrak{j}_N(k), \mathfrak{j}_N(l) \} \big\} \right);
    \end{equation*}
    \item[\textbullet] for every $i,j \in \mathbb{I}_N$, the weight of the edge linking the vertices $i$ and $j$ is
    \begin{equation*}
        w_{i,j}^N = w_{j,i}^N = \Card \left( \big\{ C_{(k,-l)}^N \in \mathcal{C}_N(q); \quad \{ i,j \} = \{ \mathfrak{j}_N(k), \mathfrak{j}_N(l) \} \big\} \right).
    \end{equation*}
\end{itemize}
\end{definition}

\begin{definition}\label{def:GN_q_path}
Let $q \in \mathcal{Q}_\text{ad}^\varepsilon$ and let $i,j \in \mathbb{I}_N$ be two vertices of the graph $G_N(q)$. We say that \emph{there is a path in $G_N(q)$ from $i$ to $j$} and we denote $i\ \accentset{N}{\sim}\ j$ if the vertices $i$ and $j$ are in the same connected component of $G_N(q)$. In particular, if $w_{i,j}^N \ne 0$, then $i\ \accentset{N}{\sim}\ j$.
\end{definition}

We then recall the definition of the Laplacian matrix associated to a graph.

\begin{definition}\label{def:AN_q}
Let $q \in \mathcal{Q}_\text{ad}^\varepsilon$. The \emph{Laplacian matrix associated to the graph $G_N(q)$} (see Definition~\ref{def:GN_q}) is the symmetric positive matrix $A_N(q) \in \mathcal{M}_{2N}(\mathbb{R})$ defined by
\begin{equation}\label{eq:AN_q}
    A_N(q) =
    \begin{pmatrix}
        d_{-N}^N & \cdots & -w_{-N,-1}^N & -w_{-N,1}^N & \cdots & -w_{-N,N}^N \\
          \vdots & \ddots &       \vdots &      \vdots &        &      \vdots \\
    -w_{-1,-N}^N & \cdots &     d_{-1}^N & -w_{-1,1}^N & \cdots & -w_{-1,N}^N \\
    -w_{ 1,-N}^N & \cdots & -w_{ 1,-1}^N &     d_{1}^N & \cdots & -w_{ 1,N}^N \\
          \vdots &        &       \vdots &      \vdots & \ddots &      \vdots \\
    -w_{ N,-N}^N & \cdots & -w_{ N,-1}^N & -w_{ N,1}^N & \cdots &     d_{N}^N
    \end{pmatrix}_{2N \times 2N}.
\end{equation}
\end{definition}

\begin{remark}\label{rmk:AN_q}
Remark that for every $q \in \mathcal{Q}_\text{ad}^\varepsilon$, the graph $G_N(q)$ has no loop, i.e. $w_{i,i}^N = 0$ for every $i \in \mathbb{I}_N$. Indeed, the elementary squares $C_{(k,-l)}^N$ such that $\mathfrak{j}_N(k) = \mathfrak{j}_N(l) = i$ have their centers $x_{(k,-l)}^N \in \mathbb{Z}$ and, consequently, cannot be in $Q_T$.

Remark also that the Laplacian matrix $A_N(q)$ of the graph $G_N(q)$ verifies the following property (see~\cite{BrouwerHaemers12, Chung97}) : for every $\eta = (\eta_{-N},\dots,\eta_{-1}, \eta_1,\dots,\eta_N) \in \mathbb{R}^{2N}$,
\begin{equation}\label{eq:AN_q_eta}
    \eta^T A_N(q) \eta = \sum_{i \in \mathbb{I}_N} d_i^N \eta_i^2 - \sum_{i,j \in \mathbb{I}_N} w_{i,j}^N \eta_i \eta_j = \sum_{C_{(i,j)}^N \in \mathcal{C}_N(q)} (\eta_{\mathfrak{j}_N(i)} - \eta_{-\mathfrak{j}_N(j)})^2.
\end{equation}
\end{remark}

From now on, we consider that the assumption of Lemma~\ref{lem:RN_q} holds true, \emph{i.e.} we take $N > 1/\varepsilon$. More precisely, we fix $N$ the smallest integer strictly greater than $\varepsilon^{-1}$.

\begin{lemma}\label{lem:GN_q_connect}
Let $q \in \mathcal{Q}_\text{ad}^\varepsilon$, so $q^\varepsilon$ verifies the usual \emph{geometric optics condition}. Then the associated graph $G_N(q)$ is connected.
\end{lemma}

\begin{proof}
Let $i \in \{ 1,\dots,N-1 \}$. We denote by $D_i^+$ the support of the characteristic line "$x + t = x_i^N$", starting from $x_i^N$ in the direction of decreasing $x$ and following the rules of geometric optics for its reflexion on $\Sigma_T$. Since $q^\varepsilon$ satisfies the geometric optics condition, there exists $(x^*,t^*) \in q^\varepsilon \cap D_i^+$. From Lemma~\ref{lem:RN_q}, we have $q^\varepsilon \subset R_N(q)$, so $(x^*,t^*)$ belongs to the common side of two elementary squares in $\mathcal{C}_N(q)$:
\begin{equation*}
    \left( C_{(k,l)}^N \text{ and } C_{(k+1,l)}^N \text{ with } \mathfrak{j}_N(k) = i \right) \quad \text{or} \quad \left( C_{(l,k)}^N \text{ and } C_{(l,k-1)}^N \text{ with } \mathfrak{j}_N(k) = -i \right).
\end{equation*}
Therefore $i\ \accentset{N}{\sim}\ i+1$ and, so, the vertices $\{ 1,\dots,N \}$ are in the same connected component.

We denote by $D_i^-$ the support of the characteristic line "$x - t = x_i^N$", starting from $x_i^N$ in the direction of increasing $x$ and following the rules of geometric optics for its reflexion on $\Sigma_T$. Since $q^\varepsilon$ satisfies the geometric optics condition, there exists $(x^*,t^*) \in q^\varepsilon \cap D_i^-$. From Lemma~\ref{lem:RN_q}, we have $q^\varepsilon \subset R_N(q)$, so $(x^*,t^*)$ belongs to the common side of two elementary squares in $\mathcal{C}_N(q)$:
\begin{equation*}
    \left( C_{(l,k)}^N \text{ and } C_{(l,k+1)}^N \text{ with } \mathfrak{j}_N(k) = i \right) \quad \text{or} \quad \left( C_{(k,l)}^N \text{ and } C_{(k-1,l)}^N \text{ with } \mathfrak{j}_N(k) = -i \right).
\end{equation*}
Therefore $-i\ \accentset{N}{\sim}\ -i-1$ and, so, the vertices $\{ -N,\dots,-1 \}$ are in the same connected component.

In order to finish the proof, it remains to show that the vertices $N$ and $-N$ belong to the same connected component. We denote by $D_N^+$ the support of the characteristic line "$x + t = x_N^N$", starting from $x_N^N$ in the direction of decreasing $x$ and following the rules of geometric optics for its reflexion on $\Sigma_T$. Since $q^\varepsilon$ satisfies the geometric optics condition, there exists $(x^*,t^*) \in q^\varepsilon \cap D_N^+$. From Lemma~\ref{lem:RN_q}, we have $q^\varepsilon \subset R_N(q)$, so $(x^*,t^*)$ belongs to the common side of two elementary squares in $\mathcal{C}_N(q)$:
\begin{equation*}
    \left( C_{(k,l)}^N \text{ and } C_{(k+1,l)}^N \text{ with } \mathfrak{j}_N(k) = N \right) \quad \text{or} \quad \left( C_{(l,k)}^N \text{ and } C_{(l,k-1)}^N \text{ with } \mathfrak{j}_N(k) = -N \right).
\end{equation*}
Hence, $N\ \accentset{N}{\sim}\ -N$.
\end{proof}

\begin{remark}\label{rmk:lmbd_N}
A well known graph theory result (see, for instance,~\cite[Proposition~1.3.7]{BrouwerHaemers12}) states that the graph $G_N(q)$ is connected if and only if $\dim(\ker(A_N(q))) = 1$. Moreover, if $G_N(q)$ is connected, then $\ker(A_N(q)) = \Vect(\mathbf{1}_{2N})$, where $\mathbf{1}_{2N}$ is the vector in $\mathbb{R}^{2N}$ with all its component equal to 1.

Let us denote $\lambda_N(q) > 0$ the smallest non-zero eigenvalue of the matrix $A_N(q)$. This eigenvalue is known in graph theory as the \emph{algebraic connectivity} of the graph. We also define $\lambda_N$ by
\begin{equation}\label{eq:lmbd_N}
    \lambda_N = \min_{q \in \mathcal{Q}_\text{ad}^\varepsilon} \lambda_N(q) > 0.
\end{equation}
Note that since the set $\{ G_N(q);\ q \in \mathcal{Q}_\text{ad}^\varepsilon \}$ has a finite number of elements, $\lambda_N$ is well defined.
\end{remark}

\begin{definition}\label{def:ApN_q}
For every $p \in \mathbb{N}^*$, we denote by $\mathcal{C}_N^p(q)$ the set formed by the elementary squares associated to the subdivision $S_{pN}$ having their interior in $R_N(q)$:
\begin{equation}\label{eq:CpN_q}
    \mathcal{C}_N^p(q) = \big\{ C_{(i,j)}^{pN} \in \mathcal{C}_{pN}; \quad \accentset{\circ}{C}_{(i,j)}^{pN} \subset R_N(q) \big\}.
\end{equation}
We then define the graph $G_N^p(q)$ following Definition~\ref{def:GN_q}, substituting $N$ by $p N$, and substituting $\mathcal{C}_N(q)$ by $\mathcal{C}_N^p(q)$ in the definitions of the vertex degrees and the edge weights. Finally, we denote $A_N^p(q) \in \mathcal{M}_{2pN}(\mathbb{R})$ the Laplacian matrix associated to the graph $G_N^p(q)$. This matrix has the following block form:
\begin{equation}\label{eq:ApN_q}
    A_N^p(q) =
    \begin{pmatrix}
      d_{-N}^N p I_p & \cdots & -w_{-N,-1}^N J_p & -w_{-N,1}^N J_p & \cdots & -w_{-N,N}^N J_p \\
              \vdots & \ddots &           \vdots &          \vdots &        &          \vdots \\
    -w_{-1,-N}^N J_p & \cdots &   d_{-1}^N p I_p & -w_{-1,1}^N J_p & \cdots & -w_{-1,N}^N J_p \\
    -w_{ 1,-N}^N J_p & \cdots & -w_{ 1,-1}^N J_p &   d_{1}^N p I_p & \cdots & -w_{ 1,N}^N J_p \\
              \vdots &        &           \vdots &          \vdots & \ddots &          \vdots \\
    -w_{ N,-N}^N J_p & \cdots & -w_{ N,-1}^N J_p & -w_{ N,1}^N J_p & \cdots &   d_{N}^N p I_p
    \end{pmatrix}_{2pN \times 2pN},
\end{equation}
where $I_p, J_p \in \mathcal{M}_p(\mathbb{R})$ are respectively the identity matrix and the matrix with all its elements equal to $1$.

Moreover, for every $\eta = (\eta_{-pN},\dots,\eta_{-1}, \eta_1,\dots,\eta_{pN}) \in \mathbb{R}^{2pN}$, we have
\begin{equation}\label{eq:ApN_q_eta}
    \eta^T A_N^p(q) \eta = \sum_{C_{(i,j)}^N \in \mathcal{C}_N(q)} \sum_{i' \in \mathbb{J}_i^p} \sum_{j' \in \mathbb{J}_j^p} (\eta_{\mathfrak{j}_{pN}(i')} - \eta_{-\mathfrak{j}_{pN}(j')})^2,
\end{equation}
with $\mathbb{J}_i^p$ defined in~\eqref{eq:Ji}.
\end{definition}

For any $p \in \mathbb{N}^*$, the following lemma makes the link between the spectrum of the Laplacian matrix $A_N^p(q)$ (see Definition~\ref{def:ApN_q}) and the spectrum of the Laplacian matrix $A_N(q)$ (see Definition~\ref{def:AN_q}).

\begin{lemma}\label{lem:Lap_spec}
Let $p \in \mathbb{N}^*$. The spectrum of the Laplacian matrix $\frac{1}{p} A_N^p(q)$ (see~\eqref{eq:ApN_q}) is composed of the spectrum of the Laplacian matrix $A_N(q)$ (see~\eqref{eq:AN_q}), and the diagonal elements of $A_N(q)$ repeated $p - 1$ times. Moreover, $\dim(\ker(A_N^p(q))) = 1$ and $\ker(A_N^p(q)) = \Vect(\mathbf{1}_{2pN})$.
\end{lemma}

\begin{proof}
Let $\eta = (\eta_{-pN},\dots,\eta_{-1}, \eta_1,\dots,\eta_{pN}) \in \mathbb{R}^{2pN}$. For every $i \in \mathbb{I}_N$, we denote by $\Gamma_i = (\eta_{i'})_{i' \in \mathbb{J}_i^p} \in \mathbb{R}^p$ and we group these vectors in the matrix
\begin{equation*}
    \Gamma = (\Gamma_{-N} | \cdots | \Gamma_{-1} | \Gamma_1 | \cdots | \Gamma_N) \in \mathcal{M}_{p,2N}(\mathbb{R}).
\end{equation*}
In view of~\eqref{eq:ApN_q}, it follows that
\begin{equation*}
    \eta^T A_N^p(q) \eta = p \sum_{i \in \mathbb{I}_N} d_i^N \Gamma_i^T \Gamma_i - \sum_{i,j \in \mathbb{I}_N} w_{i,j}^N \Gamma_i^T J_p \Gamma_j.
\end{equation*}
Since $J_p$ is a real symmetric matrix, there exists an orthonormal basis $(b_k)_{1 \leq k \leq p}$ of $\mathbb{R}^p$ diagonalizing $J_p$. Let us denote $b_1 = \frac{1}{\sqrt{p}} \mathbf{1}_p$. Then, there exists a diagonal matrix $D \in \mathcal{M}_p(\mathbb{R})$ and a unitary matrix $Q \in \mathcal{M}_p(\mathbb{R})$ such that $J_p = Q D Q^T$. These matrices have the following form:
\begin{equation*}
    D =
    \begin{pmatrix}
         p &      0 & \cdots &      0 \\
         0 &      0 & \cdots &      0 \\
    \vdots & \vdots & \ddots & \vdots \\
         0 &      0 & \cdots &      0
    \end{pmatrix}_{p \times p}
    \quad \textrm{and} \quad Q = (b_1 | \cdots | b_p)_{p \times p}.
\end{equation*}
We also define the matrix $U = Q^T \Gamma \in \mathcal{M}_{p,2N}(\mathbb{R})$, and denote respectively
\begin{equation*}
    U_{k,.} = (b_k^T \Gamma_i)_{i \in \mathbb{I}_N} \in \mathbb{R}^{2N} \quad \text{and} \quad U_{.,i} = (b_k^T \Gamma_i)_{1 \leq k \leq p} \in \mathbb{R}^p
\end{equation*}
the rows and the columns of $U$. Then, for $i,j \in \mathbb{I}_N$, we have
\begin{equation*}
    \Gamma_i^T \Gamma_i = U_{.,i}^T U_{.,i} = \sum_{k=1}^p U_{k,i}^2 \quad \text{and} \quad \Gamma_i^T J_p \Gamma_j = U_{.,i}^T D U_{.,j} = p U_{1,i} U_{1,j}.
\end{equation*}
The spectrum of the matrix $\frac{1}{p} A_N^p(q)$ can now be computed from
\begin{align*}
    \frac{1}{p} \eta^T A_N^p(q) \eta & = \sum_{i \in \mathbb{I}_N} d_i^N U_{1,i}^2 - \sum_{i,j \in \mathbb{I}_N} w_{i,j}^N U_{1,i} U_{1,j} + \sum_{k=2}^p \sum_{i \in \mathbb{I}_N} d_i^N U_{k,i}^2 \\
    & = U_{1,.}^T A_N(q) U_{1,.} + \sum_{k=2}^p U_{k,.}^T \Diag(A_N(q)) U_{k,.}.
\end{align*}
Indeed, the expression above shows that $A_N^p(q)$ is unitarily similar to the block diagonal matrix
\begin{equation*}
    \begin{pmatrix}
    A_N(q) &               &        & \\
           & \Diag(A_N(q)) &        & \\
           &               & \ddots & \\
           &               &        & \Diag(A_N(q))
    \end{pmatrix}_{2pN \times 2pN}.
\end{equation*}
\end{proof}

\subsection{Proof of Theorem~\ref{thm:unif_obs}}
\label{ssect:unif_obs_pf}

We are now in position, using all these notations and results, to prove the Theorem~\ref{thm:unif_obs}.

\begin{proof}[Proof of Theorem~\ref{thm:unif_obs}]
We prove the theorem in three steps.

\noindent \emph{Step 1.} Let $N$ be the smallest integer strictly greater than $\varepsilon^{-1}$. The first step is to prove an observability inequality for the function $\varphi^N$ given by~\eqref{eq:dAlembert} and associated to the initial datum~\eqref{eq:phi0_N}-\eqref{eq:phi1_N}.

In view of~\eqref{eq:phit_N_1} and~\eqref{eq:phit_N_2}, remark that the function $(\varphi^N_t)^2$ is constant on each elementary square $C_{(i,j)}^N$ in $\mathcal{C}_N$:
\begin{equation}\label{eq:phit_N_cst}
    (\varphi^N_t)^2 |_{C_{(i,j)}^N} = \frac{1}{4} (\gamma_{\mathfrak{j}_N(i)}^N - \gamma_{-\mathfrak{j}_N(j)}^N)^2,
\end{equation}
where $\gamma_{\mathfrak{j}_N(i)}^N$ is given by~\eqref{eq:gmm_N} and $\mathfrak{j}_N$ by~\eqref{eq:j}. Using the definition~\eqref{eq:RN_q} of the set $R_{N}(q)$, we can minorate the $L^2$-norm of $\varphi^N_t$ restricted to $q$ as follows:
\begin{align}
\nonumber
    \iint_q (\varphi^N_t)^2 \geq \iint_{R_N(q)} (\varphi^N_t)^2 & = \sum_{C_{(i,j)}^N \in \mathcal{C}_N(q)} \iint_{C_{(i,j)}^N} (\varphi^N_t)^2 \\
\label{eq:obs_N_1}
    & = \frac{1}{8 N^2} \sum_{C_{(i,j)}^N \in \mathcal{C}_N(q)} (\gamma_{\mathfrak{j}_N(i)}^N - \gamma_{-\mathfrak{j}_N(j)}^N)^2.
\end{align}
For the last equality, we used that the area of every elementary square in $\mathcal{C}_N$ is $\frac{1}{2 N^2}$. Combining~\eqref{eq:obs_N_1} with the relation~\eqref{eq:AN_q_eta} in Remark~\ref{rmk:AN_q}, we obtain
\begin{equation}\label{eq:obs_N_2}
    \iint_q (\varphi^N_t)^2 \geq \frac{1}{8 N^2} (\gamma^N)^T A_N(q) \gamma^N,
\end{equation}
with $\gamma^N = (\gamma_{-N}^N, \dots, \gamma_{-1}^N, \gamma_1^N, \dots, \gamma_N^N) \in \mathbb{R}^{2N}$ and $\gamma_i^N$ given by~\eqref{eq:gmm_N}.

It is easy to see that $\gamma^N \in \ker(A_N(q))^\perp$. Indeed, applying Lemma~\ref{lem:GN_q_connect}, the graph $G_N(q)$ is connected. Then, from Remark~\ref{rmk:lmbd_N}, we have $\ker(A_N(q)) = \Vect(\mathbf{1}_{2N})$ and, since $\alpha_i^N = \alpha_{-i}^N$ and $\beta_i^N = -\beta_{-i}^N$, the vector $\gamma^N$ verifies
\begin{equation*}
    (\gamma^N)^T \mathbf{1}_{2N} = \sum_{i \in \mathbb{I}_N} \gamma_i^N = 2 \sum_{i=1}^N \alpha_i^N = 0.
\end{equation*}
Then, using $\lambda_N$ defined in~\eqref{eq:lmbd_N}, it follows that
\begin{align}
\nonumber
    (\gamma^N)^T A_N(q) \gamma^N \geq \lambda_N \sum_{i \in \mathbb{I}_N} (\gamma_i^N)^2 & = 2 \lambda_N \sum_{i=1}^N \Big( (\alpha_i^N)^2 + (\beta_i^N)^2 \Big) \\
\label{eq:AN_q_phi_01}
    & = 2 N \lambda_N \left( \| \varphi_0^N \|_{H^1_0(\Omega)}^2 + \| \varphi_1^N \|_{L^2(\Omega)}^2 \right).
\end{align}
From~\eqref{eq:obs_N_2} and~\eqref{eq:AN_q_phi_01}, we deduce the following observability inequality :
\begin{equation}\label{eq:obs_N}
    \| (\varphi_0^N, \varphi_1^N) \|_\mathbf{V}^2 \leq \frac{4 N}{\lambda_N} \| \varphi^N_t \|_{L^2(q)}^2, 
\end{equation}
where the constant $\frac{4 N}{\lambda_N}$ is independent of the domain $q$ and the initial datum $(\varphi_0, \varphi_1)$.

\emph{Step 2.} For any $p \in \mathbb{N}^*$, the second step of the proof consists in obtaining a uniform observability inequality for an initial datum $(\varphi_0^{pN}, \varphi_1^{pN})$ of the form~\eqref{eq:phi0_N}-\eqref{eq:phi1_N}. More precisely, we aim to obtain a uniform inequality with respect to the domain $q \in \mathcal{Q}_\text{ad}^\varepsilon$ and the integer $p \in \mathbb{N}^*$.

Let $p \in \mathbb{N}^*$. As in the first step of the proof, we easily see that $(\varphi^{pN}_t)^2$ is constant on every elementary square $C_{(i',j')}^{pN} \in \mathcal{C}_{pN}$:
\begin{equation*}
    (\varphi^{pN}_t)^2 |_{C_{(i',j')}^{pN}} = \frac{1}{4} (\gamma_{\mathfrak{j}_{pN}(i')}^{pN} - \gamma_{-\mathfrak{j}_{pN}(j')}^{pN})^2,
\end{equation*}
where $\varphi^{pN}$ is the solution of~\eqref{eq:adj_wave} associated to the initial datum $(\varphi_0^{pN}, \varphi_1^{pN})$ given by~\eqref{eq:phi0_N}-\eqref{eq:phi1_N}, and $\gamma_{\mathfrak{j}_{pN}(i')}^{pN}$ is defined by~\eqref{eq:gmm_N}. For any $i \in \mathbb{Z}^*$, we define the set $\mathbb{J}_i^p$ by
\begin{equation}\label{eq:Ji}
    \mathbb{J}_i^p =
    \left\{
    \begin{array}{ll}
        \{ p(i-1)+1, \dots, pi \} & \text{if } i > 0, \\[2mm]
        \{ pi, \dots, p(i+1)-1 \} & \text{if } i < 0.
    \end{array}
    \right.
\end{equation}
Then, remark that every elementary square $C_{(i,j)}^N \in \mathcal{C}_N$ is the union of $p^2$ elementary squares in $\mathcal{C}_{pN}$, or more precisely that
\begin{equation*}
    C_{(i,j)}^N = \bigcup_{i' \in \mathbb{J}_i^p} \bigcup_{j' \in \mathbb{J}_j^p} C_{(i',j')}^{pN}, \quad \forall i,j \in \mathbb{Z}^*.
\end{equation*}
Using the above expression in the evaluation of the $L^2(q)$-norm of $\varphi^{pN}$, we have
\begin{align}
\nonumber
    \iint_q (\varphi^{pN}_t)^2 & \geq \iint_{R_N(q)} (\varphi^{pN}_t)^2 = \sum_{C_{(i,j)}^N \in \mathcal{C}_N(q)} \iint_{C_{(i,j)}^N} (\varphi^{pN}_t)^2 \\
\nonumber
    & = \sum_{C_{(i,j)}^N \in \mathcal{C}_N(q)} \sum_{i' \in \mathbb{J}_i^p} \sum_{j' \in \mathbb{J}_j^p} \iint_{C_{(i',j')}^{pN}} (\varphi^{p N}_t)^2 \\
\nonumber
    & = \frac{1}{8 p^2 N^2} \sum_{C_{(i,j)}^N \in \mathcal{C}_N(q)} \sum_{i' \in \mathbb{J}_i^p} \sum_{j' \in \mathbb{J}_j^p} (\gamma_{\mathfrak{j}_{pN}(i')}^{pN} - \gamma_{-\mathfrak{j}_{pN}(j')}^{pN})^2 \\
\label{eq:obs_pN_1}
    & = \frac{1}{8 p^2 N^2} (\gamma^{pN})^T A_N^p(q) \gamma^{pN}.
\end{align}

Since the graph $G_N(q)$ is a connected graph, the degree $d_i^N$ of every vertex $i \in \mathbb{I}_N$ verifies $d_i^N \geq 1$. Applying Lemma~\ref{lem:Lap_spec}, the smallest non-zero eigenvalue $\lambda_N^p(q)$ of $\frac{1}{p} A_N^p(q)$ verifies
\begin{equation}\label{eq:lmbd_pN}
    \lambda_N^p(q) = \min(\lambda_N(q), \min_{i \in \mathbb{I}_N} d_i^N) \geq \min(\lambda_N,1) > 0,
\end{equation}
so we set $\widehat{\lambda}_N = \min(\lambda_N,1)$. The vector $\gamma^{pN} = (\gamma_{i'}^{pN})_{i' \in \mathbb{I}_{pN}}$ belongs to $\ker(A_N^p(q))^\perp$. Indeed,
\begin{equation*}
    (\gamma^{pN})^T \mathbf{1}_{2pN} = \sum_{i' \in \mathbb{I}_{pN}} \gamma_{i'}^{pN} = 2 \sum_{i'=1}^{pN} \alpha_{i'}^{pN} = 0.
\end{equation*}
It follows that
\begin{align*}
    \frac{1}{p} (\gamma^{pN})^T A_N^p(q) \gamma^{pN} \geq \widehat{\lambda}_N \sum_{i' \in \mathbb{I}_{pN}} (\gamma_{i'}^{pN})^2 & = 2 \widehat{\lambda}_N \sum_{i'=1}^{pN} \Big( (\alpha_{i'}^{pN})^2 + (\beta_{i'}^{pN})^2 \Big) \\
    & = 2 p N \widehat{\lambda}_N \left( \| \varphi_0^{pN} \|_{H^1_0(\Omega)}^2 + \| \varphi_1^{pN} \|_{L^2(\Omega)}^2 \right).
\end{align*}
Consequently, combining the above relation with~\eqref{eq:obs_pN_1}, we obtain the following observability inequality
\begin{equation}\label{eq:obs_pN}
    \| (\varphi_0^{pN}, \varphi_1^{pN}) \|_\mathbf{V}^2 \leq \frac{4 N}{\widehat{\lambda}_N} \| \varphi^{pN}_t \|_{L^2(q)}^2, \quad \forall p\in \mathbb{N}
\end{equation}
with the observability constant $\frac{4 N}{\widehat{\lambda}_N}$ independent of the domain $q$, the initial datum $(\varphi_0, \varphi_1)$ and the integer $p$.

\emph{Step 3.} In order to finish the proof, we pass to the limit when $p \to \infty$ in the observability inequality~\eqref{eq:obs_pN}. It is easy to see that when $p \to \infty$, we have the convergences
\begin{equation*}
    \varphi_0^{pN} \to \varphi_0 \quad \text{in } H^1_0(\Omega) \quad \text{and} \quad \varphi_1^{pN} \to \varphi_1 \quad \text{in } L^2(\Omega).
\end{equation*}
Moreover, since the solution $\varphi$ of the wave equation~\eqref{eq:adj_wave} depends continuously on its initial condition $(\varphi_0, \varphi_1)\in \boldsymbol{V}$ , we can write
\begin{equation*}
    \varphi^{pN}_t \to \varphi_t \quad \text{in } L^2(0,T; L^2(\Omega)).
\end{equation*}
Finally, passing to the limit in~\eqref{eq:obs_pN}, we get
\begin{equation*}
    \| (\varphi_0, \varphi_1) \|_\mathbf{V}^2 \leq \max \left\{ 4N, \frac{4 N}{\lambda_N} \right\} \| \varphi_t \|_{L^2(q)}^2,\quad \forall (\varphi_0, \varphi_1)\in \mathbf{V}
\end{equation*}
which concludes the proof with $C_\text{obs}^\varepsilon = \max\left\{ 4N, \frac{4 N}{\lambda_N} \right\}$. We recall that $N$ depends on $\varepsilon$ by the condition $N > 1/\varepsilon$.
\end{proof}

\begin{remark}\label{rmk:Cobs_bnd}
Let $q \subset Q_T$ be a finite union of open sets. If $q$ verifies the usual geometric optics condition, there exists $\varepsilon > 0$ small enough such that $q^\varepsilon$ still verifies the geometric optics condition. We then set $N = \lfloor 1/\varepsilon \rfloor + 1$. The associated graph $G_N(q)$ being connected, there exists a relation (see, for instance,~\cite{Mohar91}) between the algebraic connectivity $\lambda_N(q)$, the number of vertices $N_V$ and the diameter $D_G$ of the graph. More exactly, this relation is $\lambda_N(q) \geq \frac{4}{N_V D_G}$. Since in our case $N_V = 2 N$ and $D_G \leq 2 N$, we deduce that $\lambda_N(q) \geq \frac{1}{N^2}$ and therefore that $C_\text{obs}(q) \leq 4 N^3$. In the worst situation, we can have an observability constant of order $1/\varepsilon^3$. Therefore, if we consider $\varepsilon$ as a measure of the "thickness" of the observation domain $q$, we find the estimation of the observability constant given in~\cite[Proposition~2.1]{Periago09}.
\end{remark}

\subsection{One explicit example}
\label{ssect:unif_obs_ex}

We illustrate in this section the proof of Theorem~\ref{thm:unif_obs} on a simple example for which the observation domain $q$ depicted in Figure~\ref{fig:RN_q_ex} (colored in red) is well adapted to the subdivision $S_4$. The study of this example is also the opportunity to develop a method for the computation of the observability constant for observation domains which are exactly the union of elementary squares associated to a given subdivision $S_N$, for a fixed integer $N > 0$.

We start by enumerating the elementary squares composing the observation domain $q$. In Table~\ref{tab:RN_q_ex}, we list, for $i,j \in \mathbbm{Z}^*$, the elementary squares $C_{(i,j)}^4$ included in $q$ and the values of the indices $\mathfrak{j}_4(i)$ and $-\mathfrak{j}_4(j)$, allowing to compute the Laplacian matrix $A_4(q)$ associated to the corresponding graph $G_4(q)$.

\begin{table}[ht]
\centering
\begin{tabular}[t]{|lcc|lcc|lcc|}
\hline
$C_{(i,j)}^4$ & $\mathfrak{j}_4(i)$ & $-\mathfrak{j}_4(j)$ & $C_{(i,j)}^4$ & $\mathfrak{j}_4(i)$ & $-\mathfrak{j}_4(j)$ & $C_{(i,j)}^4$ & $\mathfrak{j}_4(i)$ & $-\mathfrak{j}_4(j)$ \\
\hline
$C_{(2, 1)}^4$ & $ 2$ & $-1$ & $C_{(6,-1)}^4$ & $-3$ & $ 1$ & $C_{(9,-4)}^4$ & $ 1$ & $ 4$ \\
$C_{(2,-1)}^4$ & $ 2$ & $ 1$ & $C_{(7, 1)}^4$ & $-2$ & $-1$ & $C_{(8,-5)}^4$ & $-1$ & $-4$ \\
$C_{(3, 1)}^4$ & $ 3$ & $-1$ & $C_{(7,-1)}^4$ & $-2$ & $ 1$ & $C_{(9,-5)}^4$ & $ 1$ & $-4$ \\
$C_{(3,-1)}^4$ & $ 3$ & $ 1$ & $C_{(8,-1)}^4$ & $-1$ & $ 1$ & $C_{(8,-6)}^4$ & $-1$ & $-3$ \\
$C_{(4, 1)}^4$ & $ 4$ & $-1$ & $C_{(8,-2)}^4$ & $-1$ & $ 2$ & $C_{(9,-6)}^4$ & $ 1$ & $-3$ \\
$C_{(4,-1)}^4$ & $ 4$ & $ 1$ & $C_{(9,-2)}^4$ & $ 1$ & $ 2$ & $C_{(8,-7)}^4$ & $-1$ & $-2$ \\
$C_{(5, 1)}^4$ & $-4$ & $-1$ & $C_{(8,-3)}^4$ & $-1$ & $ 3$ & $C_{(9,-7)}^4$ & $ 1$ & $-2$ \\
$C_{(5,-1)}^4$ & $-4$ & $ 1$ & $C_{(9,-3)}^4$ & $ 1$ & $ 3$ &                &    & \\
$C_{(6, 1)}^4$ & $-3$ & $-1$ & $C_{(8,-4)}^4$ & $-1$ & $ 4$ &                &    & \\
\hline
\end{tabular}
\caption{Elementary squares associated to $S_4$ and belonging to $\mathcal{C}_4(q)$.}
\label{tab:RN_q_ex}
\end{table}

\noindent
\begin{minipage}{0.6\textwidth}
\hspace{15pt} The Laplacian matrix associated to the graph $G_4(q)$ is given by
\begin{equation*}
    A_4(q) =
    \begin{pmatrix}
     4 &  0 &  0 & -2 & -2 &  0 &  0 &  0 \\
     0 &  4 &  0 & -2 & -2 &  0 &  0 &  0 \\
     0 &  0 &  4 & -2 & -2 &  0 &  0 &  0 \\
    -2 & -2 & -2 & 13 & -1 & -2 & -2 & -2 \\
    -2 & -2 & -2 & -1 & 13 & -2 & -2 & -2 \\
     0 &  0 &  0 & -2 & -2 &  4 &  0 &  0 \\
     0 &  0 &  0 & -2 & -2 &  0 &  4 &  0 \\
     0 &  0 &  0 & -2 & -2 &  0 &  0 &  4 \\
    \end{pmatrix}_{8 \times 8}.
\end{equation*}
The spectrum of $A_4(q)$ can be explicitly computed:
\begin{equation*}
    \Sp(A_4(q)) = \{ 0, 4, 4, 4, 4, 4, 14, 16 \}.
\end{equation*}
\end{minipage}
\hfill
\begin{minipage}{0.35\textwidth}
\begin{figure}[H]
\centering
\begin{tikzpicture}[scale=2.5]
\tikzmath{\T = 2;}

\coordinate (i) at (1,0);
\coordinate (j) at (0,1);
\coordinate (ij) at (1,1);



\draw[blue,line width=1] (0,0) -- (1,1);
\draw[blue,line width=1] (0.25,0) -- (1,0.75);
\draw[blue,line width=1] (0.5,0) -- (1,0.5);
\draw[blue,line width=1] (0.75,0) -- (1,0.25);

\draw[blue,line width=1] (0.25,0) -- (0,0.25);
\draw[blue,line width=1] (0.5,0) -- (0,0.5);
\draw[blue,line width=1] (0.75,0) -- (0,0.75);
\draw[blue,line width=1] (1,0) -- (0,1);

\draw[blue,line width=1] (0,0.25) -- (1,1.25);
\draw[blue,line width=1] (0,0.5) -- (1,1.5);
\draw[blue,line width=1] (0,0.75) -- (1,1.75);
\draw[blue,line width=1] (0,1) -- (1,2);

\draw[blue,line width=1] (1,0.25) -- (0,1.25);
\draw[blue,line width=1] (1,0.5) -- (0,1.5);
\draw[blue,line width=1] (1,0.75) -- (0,1.75);
\draw[blue,line width=1] (1,1) -- (0,2);

\draw[blue,line width=1] (0,1.25) -- (0.75,2);
\draw[blue,line width=1] (0,1.5) -- (0.5,2);
\draw[blue,line width=1] (0,1.75) -- (0.25,2);

\draw[blue,line width=1] (1,1.25) -- (0.25,2);
\draw[blue,line width=1] (1,1.5) -- (0.5,2);
\draw[blue,line width=1] (1,1.75) -- (0.75,2);


\path[fill=red!25] (0.25,0) -- (1,0.75) -- (0.75,1) -- (0,0.25) -- (0.25,0);
\path[fill=red!25] (0.875,0.875) -- (1,1) -- (0.875,1.125) -- (0.75,1) -- (0.875,0.875);
\path[fill=red!25] (0.75,1) -- (1,1.25) -- (0.25,2) -- (0,1.75) -- (0.75,1);


\draw[red,line width=2] (0,0.25) -- (1,1.25);
\draw[red,line width=2] (0.125,0.125) -- (1,1);
\draw[red,line width=2] (0.25,0) -- (1,0.75);
\draw[red,line width=2] (1,0.75) -- (0,1.75);
\draw[red,line width=2] (1,1) -- (0.125,1.875);
\draw[red,line width=2] (1,1.25) -- (0.25,2);

\draw[red,line width=2] (0.25,0) -- (0,0.25);
\draw[red,line width=2] (0.375,0.125) -- (0.125,0.375);
\draw[red,line width=2] (0.5,0.25) -- (0.25,0.5);
\draw[red,line width=2] (0.625,0.375) -- (0.375,0.625);
\draw[red,line width=2] (0.75,0.5) -- (0.5,0.75);
\draw[red,line width=2] (0.875,0.625) -- (0.625,0.875);

\draw[red,line width=2] (0,1.75) -- (0.25,2);
\draw[red,line width=2] (0.125,1.625) -- (0.375,1.875);
\draw[red,line width=2] (0.25,1.5) -- (0.5,1.75);
\draw[red,line width=2] (0.375,1.375) -- (0.625,1.625);
\draw[red,line width=2] (0.5,1.25) -- (0.75,1.5);
\draw[red,line width=2] (0.625,1.125) -- (0.875,1.375);


\draw[line width=2] (0,0) rectangle (1,\T);
\path (0,0) -- (1,0) node[midway,below]{$x$};
\path (0,0) -- (0,\T) node[midway,left]{$t$};
\draw[line width=2] (0,0) -- ++($-0.025*(j)$) node[below]{$0$};
\draw[line width=2] (1,0) -- ++($-0.025*(j)$) node[below]{$1$};
\draw[line width=2] (0,0) -- ++($-0.025*(i)$) node[left]{$0$};
\draw[line width=2] (0,\T) -- ++($-0.025*(i)$) node[left]{$T = 2$};


\draw[->,>=latex,line width=1] (1.1,0.375) node[right]{$q$} -- (0.5,0.375);
\end{tikzpicture}
\caption{Observation domain $q$ adapted to $S_4$.}
\label{fig:RN_q_ex}
\end{figure}
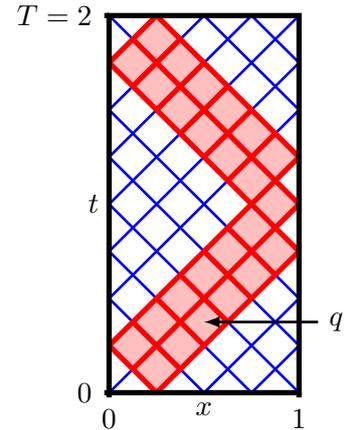
\end{minipage} \\

It confirms that the kernel of $A_4(q)$ is one-dimensional -- therefore $G_4(q)$ is connected -- and implies that the smallest non-zero eigenvalue of $A_4(q)$ is $\lambda_4(q) = 4$. If we replace the subdivision $S_4$ by the subdivision $S_{4p}$ for any $p \in \mathbb{N}^*$, then the Laplacian matrix associated to the graph $G_4^p(q)$ is the following one:
\begin{equation*}
    A_4^p(q) =
    \begin{pmatrix}
    4 p I_p &     0_p &     0_p &   -2 J_p &   -2 J_p &     0_p &     0_p &     0_p \\
        0_p & 4 p I_p &     0_p &   -2 J_p &   -2 J_p &     0_p &     0_p &     0_p \\
        0_p &     0_p & 4 p I_p &   -2 J_p &   -2 J_p &     0_p &     0_p &     0_p \\
     -2 J_p &  -2 J_p &  -2 J_p & 13 p I_p &     -J_p &  -2 J_p &  -2 J_p &  -2 J_p \\
     -2 J_p &  -2 J_p &  -2 J_p &     -J_p & 13 p I_p &  -2 J_p &  -2 J_p &  -2 J_p \\
        0_p &     0_p &     0_p &   -2 J_p &   -2 J_p & 4 p I_p &     0_p &     0_p \\
        0_p &     0_p &     0_p &   -2 J_p &   -2 J_p &     0_p & 4 p I_p &     0_p \\
        0_p &     0_p &     0_p &   -2 J_p &   -2 J_p &     0_p &     0_p & 4 p I_p \\
    \end{pmatrix}_{8p \times 8p}.
\end{equation*}
According to Lemma~\ref{lem:Lap_spec}, the smallest non-zero eigenvalue of $\frac{1}{p} A_4^p(q)$ is given by
\begin{equation*}
    \lambda_4^p(q) = \min(\lambda_4(q), \min_{i \in \mathbb{I}_4} d_i^4) = \min(4,4) = 4.
\end{equation*}
Consequently, the observability constant associated to the observation domain $q$ depicted in Figure~\ref{fig:RN_q_ex} is given by
\begin{equation*}
    C_\text{obs}(q) = \frac{4 \cdot 4}{\lambda_4^p(q)} = 4.
\end{equation*}


\subsection{A corollary}
\label{ssect:unif_obs_cor}

We show in this section a uniform observability inequality for the observation domains $q_\gamma$ defined in~\eqref{eq:q_gmm}, with $\gamma \in \mathcal{G}_\text{ad}$, which will be used in the next section. 

\begin{corollary}\label{cor:unif_obs_W_gmm}
Let $T \geq 2$. There exists a constant $C_\text{obs} > 0$ such that for every $\gamma \in \mathcal{G}_\text{ad}$,
\begin{equation}\label{eq:unif_obs_W_gmm}
    \| (\varphi_0, \varphi_1) \|_\mathbf{W}^2 \leq C_\text{obs} \| \varphi \|_{L^2(q_\gamma)}^2, \quad \forall (\varphi_0, \varphi_1) \in \mathbf{W},
\end{equation}
where $\varphi$ is the solution of the homogeneous wave equation~\eqref{eq:adj_wave} associated to the initial condition $(\varphi_0, \varphi_1)$.
\end{corollary}

\begin{proof}
We show that for any $\varepsilon > 0$ small enough, $\{ q_\gamma;\ \gamma \in \mathcal{G}_\text{ad} \} \subset \mathcal{Q}_\text{ad}^\varepsilon$. Let $\gamma \in \mathcal{G}_\text{ad}$. We introduce the sets
$\Gamma_\pm  = \big\{ (\gamma(t) \pm \delta_0, t); t \in [0,T] \big\}$, $\widetilde{\Gamma}_\pm  = \big\{ (\gamma(t) \pm \frac{\delta_0}{2}, t); t \in [0,T] \big\}$ and $Q_T^\varepsilon  = \Omega \times (\varepsilon, T - \varepsilon)$. $\gamma$ being a $M$-Lipschitz curve, we can show that
\begin{equation*}
    d(\widetilde{\Gamma}_\pm, \Gamma_\pm) \geq \frac{\delta_0}{2 \sqrt{M^2 + 1}}.
\end{equation*}
Then, for $\varepsilon < \frac{\delta_0}{2 \sqrt{M^2 + 1}}$, we have $\widetilde{q}_\gamma \cap Q_T^\varepsilon \subset q_\gamma^\varepsilon$, with the observation domain $\widetilde{q}_\gamma$ defined as in~\eqref{eq:q_gmm} with a half-width of $\delta_0/2$. The domain $\widetilde{q}_\gamma \cap Q_T^\varepsilon$ verifies the geometric optics condition because $\delta_0 \leq \gamma \leq 1 - \delta_0$ and $T - 2 \varepsilon \geq 2 (1 - \delta_0)$. Consequently, $q_\gamma^\varepsilon$ also verifies the geometric optics condition and $q_\gamma \in \mathcal{Q}_\text{ad}^\varepsilon$. We conclude the proof by noticing that the constant $\frac{\delta_0}{2 \sqrt{M^2 + 1}}$ is independent of the choice of $\gamma$.
\end{proof}

\section{Optimization of the shape of the control domain}
\label{sect:shp_opt}

In this section, we study the problem of finding the optimal shape and position of the control domain, for a given initial condition $(y_0,y_1) \in \mathbf{V}$.

\subsection{Existence of an optimal domain}

In order to show a well-posedness result, we consider a variant of the optimal problem~\eqref{eq:opt_pb} and replace the characteristic function $\mathbbm{1}_q$ in~\eqref{eq:ctrl_wave} by a more regular function in space. More precisely, we fix $\delta \in (0, \delta_0)$ and, for every $\gamma \in \mathcal{G}_\text{ad}$, we define $\chi_\gamma(x,t) = \chi(x - \gamma(t))$, with $\chi : \mathbb{R} \to [0,1]$ a $C^1$ even function such that
\begin{equation}\label{eq:chi}
    \chi(x) =
    \left\{
    \begin{array}{ll}
        1 & \text{if } x \in (-\delta_0 + \delta, \delta_0 - \delta), \\
        0 & \text{if } x \notin (-\delta_0, \delta_0), \\
        \in (0,1) & \text{otherwise}.
    \end{array}
    \right.
\end{equation}
In the sequel, we will also use the function $\chi_\gamma'$ defined by $\chi_\gamma'(x,t) = \chi'(x - \gamma(t))$. In this new setting, the HUM control now lives in the weighted space
\begin{equation}\nonumber 
    L^2_\chi(q_\gamma) \coloneqq L^2(q_\gamma; \chi_\gamma) = \left\{ v : q_\gamma \to \mathbb{R}; \quad \iint_{q_\gamma} v^2 \chi_\gamma < +\infty \right\}.
\end{equation}
Moreover, we can adapt the uniform observability inequality given in Corollary~\ref{cor:unif_obs_W_gmm}. For $T \geq 2$, there exists a constant $C_\text{obs} > 0$ such that for every $\gamma \in \mathcal{G}_\text{ad}$,
\begin{equation}\label{eq:unif_obs_W_chi}
    \| (\varphi_0, \varphi_1) \|_\mathbf{W}^2 \leq C_\text{obs} \| \varphi \|_{L^2_\chi(q_\gamma)}^2, \quad \forall (\varphi_0, \varphi_1) \in \mathbf{W},
\end{equation}
where $\varphi$ is the solution of~\eqref{eq:adj_wave} associated to the initial condition $(\varphi_0, \varphi_1)$.

Then, our optimization problem reads as follows: for a given initial datum $(y_0,y_1) \in \mathbf{V}$, solve
\begin{equation}\label{eq:min_J}
    \inf_{\gamma \in \mathcal{G}_\text{ad}} J(\gamma) = \| v \|_{L^2_\chi(q_\gamma)}^2 = \iint_{q_\gamma} \varphi^2 \chi_\gamma,
\end{equation}
where $v$ is the control of minimal $L^2_\chi$-norm distributed over $q_\gamma \subset Q_T$, and $\varphi$ is the associated adjoint state such that $v = \varphi |_{q_\gamma}$. This adjoint state can be characterized using the HUM method, it is the solution of~\eqref{eq:adj_wave} associated to the minimum $(\varphi_0, \varphi_1)$ of the conjugate functional
\begin{equation}\label{eq:J*_gmm}
    \mathcal{J}^\star_\gamma(\varphi_0, \varphi_1) = \frac{1}{2} \iint_{q_\gamma} \varphi^2 \chi_\gamma - \langle \varphi_1, y_0 \rangle_{H^{-1}(\Omega), H^1_0(\Omega)} + \langle \varphi_0, y_1\rangle_{L^2(\Omega)}, \quad \forall (\varphi_0, \varphi_1) \in \mathbf{W}.
\end{equation}

To show the well-posedness of~\eqref{eq:min_J}, we follow the steps of~\cite[Theorem~2.1]{Periago09}. We start with a convergence result on the function $\chi_\gamma$.

\begin{lemma}\label{lem:chi_gmm_cvg}
Let $(\gamma_n)_{n \geq 0} \subset \mathcal{G}_\text{ad}$ and $\gamma \in \mathcal{G}_\text{ad}$. If $\gamma_n \to \gamma$ in $L^\infty(0,T)$, then $\chi_{\gamma_n} \to \chi_\gamma$ in $L^\infty(Q_T)$.
\end{lemma}

\begin{proof}
It is a direct consequence of the Taylor's inequality applied to $\chi$. Indeed, this inequality gives
\begin{equation*}
    \| \chi_{\gamma_n} - \chi_\gamma \|_{L^\infty(Q_T)} \leq \| \chi' \|_{L^\infty(\mathbb{R})} \| \gamma_n - \gamma \|_{L^\infty(0,T)}.
\end{equation*}
\end{proof}

We then have that the following continuity result. 

\begin{proposition}\label{prop:J_cont}
The cost $J$ is continuous over $\mathcal{G}_\text{ad}$ for the norm $L^\infty(0,T)$.
\end{proposition}

\begin{proof}
Let $(\gamma_n)_{n \geq 0} \subset \mathcal{G}_\text{ad}$ and $\gamma \in \mathcal{G}_\text{ad}$ such that $\gamma_n \to \gamma$ in $L^\infty(0,T)$ as $n\to \infty$.

For any $n \in \mathbb{N}$, we denote $(\varphi_0^n, \varphi_1^n) \in \mathbf{W}$ the minimum of $\mathcal{J}^\star_{\gamma_n}$, and $\varphi^n$ the corresponding solution of~\eqref{eq:adj_wave}. Using the uniform observability inequality~\eqref{eq:unif_obs_W_chi} and the optimality condition of $\mathcal{J}^\star_{\gamma_n}$, it follows that
\begin{align*}
      \| (\varphi_0^n, \varphi_1^n) \|_\mathbf{W}^2\leq  C_\text{obs} \iint_{q_{\gamma_n}} (\varphi^n)^2 \chi_{\gamma_n}  &=  C_\text{obs} \left( \langle \varphi_1^n, y_0 \rangle_{H^{-1}, H^1_0} - \langle \varphi_0^n, y_1\rangle_{L^2} \right) \\
    & \leq  C_\text{obs} \| (\varphi_0^n, \varphi_1^n) \|_\mathbf{W} \| (y_0, y_1) \|_\mathbf{V}
\end{align*}
leading to the uniform bound $\| (\varphi_0^n, \varphi_1^n) \|_\mathbf{W}\leq C_\text{obs} \| (y_0, y_1) \|_\mathbf{V}$. Consequently, there exist two functions $\varphi_0 \in L^2(\Omega)$ and $\varphi_1 \in H^{-1}(\Omega)$ such that, up to a subsequence, as $n\to \infty$, we have
\begin{equation*}
    \varphi_0^n \rightharpoonup \varphi_0 \text{ weakly in } L^2(\Omega) \quad \text{and} \quad \varphi_1^n \rightharpoonup \varphi_1 \text{ weakly in } H^{-1}(\Omega).
\end{equation*}
From the continuous dependence of the solution of the wave equation with respect to the initial condition, it follows
\begin{equation*}
    \varphi^n \rightharpoonup \varphi \quad \text{weakly in } L^2(0,T; L^2(\Omega)),
\end{equation*}
where $\varphi$ is the solution of~\eqref{eq:adj_wave} associated to $(\varphi_0, \varphi_1)$.

Let $\psi \in L^2(0,T; L^2(\Omega))$. We then have
\begin{equation*}
    \iint_{Q_T} \psi \varphi^n \chi_{\gamma_n} = \iint_{Q_T} \psi \varphi^n \chi_\gamma + \iint_{Q_T} \psi \varphi^n (\chi_{\gamma_n} - \chi_\gamma) \to \iint_{Q_T} \psi \varphi \chi_\gamma.
\end{equation*}
Indeed, we can take the weak limit in the first term because $\psi \chi_\gamma \in L^2(0,T; L^2(\Omega))$. Using Lemma~\ref{lem:chi_gmm_cvg} and the boundedness of $(\varphi^n)_{n \geq 0}$ in $L^2(0,T; L^2(\Omega))$, the second term converges to $0$ because
\begin{equation*}
    \left| \iint_{Q_T} \psi \varphi^n (\chi_{\gamma_n} - \chi_\gamma) \right| \leq \| \psi \|_{L^2(L^2)} \| \varphi^n \|_{L^2(L^2)} \| \chi_{\gamma_n} - \chi_\gamma \|_{L^\infty(Q_T)}.
\end{equation*}
Consequently, we obtain the convergence
\begin{equation*}
    \varphi^n \chi_{\gamma_n} \rightharpoonup \varphi \chi_\gamma \quad \text{weakly in } L^2(0,T; L^2(\Omega)).
\end{equation*}

Let now $(\psi_0, \psi_1) \in \mathbf{W}$ and $\psi$ the corresponding solution of~\eqref{eq:adj_wave}. Taking the weak limit in the optimality condition
\begin{equation*}
    \iint_{q_{\gamma_n}} \psi \varphi^n \chi_{\gamma_n} = \langle \psi_1, y_0 \rangle_{H^{-1}, H^1_0} - \int_\Omega \psi_0 y_1,
\end{equation*}
we find
\begin{equation*}
    \iint_{q_\gamma} \psi \varphi \chi_\gamma = \langle \psi_1, y_0 \rangle_{H^{-1}, H^1_0} - \int_\Omega \psi_0 y_1.
\end{equation*}
This means that $(\varphi_0, \varphi_1)$ is the minimum of $\mathcal{J}^\star_\gamma$. Besides, we remark that this property uniquely characterizes the weak limit of any subsequence of $(\varphi_0^n, \varphi_1^n)$. This implies that the whole sequence $(\varphi_0^n, \varphi_1^n)$ weakly converges. The continuity of $J$ is finally obtain by taking the weak limit in the optimality condition
\begin{equation*}
    \iint_{q_{\gamma_n}} (\varphi^n)^2 \chi_{\gamma_n} = \langle \varphi_1^n, y_0 \rangle_{H^{-1}, H^1_0} - \int_\Omega \varphi_0^n y_1 \to \langle \varphi_1, y_0 \rangle_{H^{-1}, H^1_0} - \int_\Omega \varphi_0 y_1 = \iint_{q_\gamma} \varphi^2 \chi_\gamma.
\end{equation*}
\end{proof}

The continuity of $J$ then allows to show that the extremal problem~\eqref{eq:min_J} is well-posed.

\begin{proposition}\label{prop:min_J_WP}
The cost $J$ reaches its minimum over $\mathcal{G}_\text{ad}$.
\end{proposition}

\begin{proof}
The cost $J$ being bounded by below, there exists a minimizing sequence $(\gamma_n)_{n \geq 0} \subset \mathcal{G}_\text{ad}$. By definition of $\mathcal{G}_\text{ad}$, this sequence is bounded in $W^{1,\infty}(0,T)$. By the generalized Rellich theorem, $W^{1,\infty}(0,T)$ is compactly embedded in $L^\infty(0,T)$. Consequently, there exists a curve $\gamma \in L^\infty(0,T)$ such that, up to a subsequence, $\gamma_n \to \gamma$ in $L^\infty(0,T)$. From the definition of $\mathcal{G}_\text{ad}$, all the curves $\gamma_n$ are $M$-Lipschitzian, with $M$ independent of $n$. So, taking the pointwise limit in the expressions
\begin{equation*}
    \begin{array}{ll}
        |\gamma_n(t) - \gamma_n(s)| \leq M |t - s|, & \forall t,s \in [0,T], \\[2mm]
        \delta_0 \leq \gamma_n(t) \leq 1 - \delta_0, & \forall t \in [0,T],
    \end{array}
\end{equation*}
we notice that $\gamma \in \mathcal{G}_\text{ad}$. Finally, using Proposition~\ref{prop:J_cont}, we obtain $J(\gamma_n) \to J(\gamma) = \inf_{\mathcal{G}_\text{ad}} J$
which means that $\gamma$ is a minimum of $J$ over $\mathcal{G}_\text{ad}$.
\end{proof}

\subsection{First directional derivative of the cost}

We now give the expression of the directional derivative of the cost $J$.

\begin{definition}\label{def:adm_pert}
Let $\gamma, \overline{\gamma} \in W^{1,\infty}(0,T)$, with $\delta_0 \leq \gamma \leq 1 - \delta_0$. The perturbation $\overline{\gamma}$ is said admissible if and only if for any $\eta > 0$ small enough, the perturbed curve $\gamma_\eta = \gamma + \eta \overline{\gamma}$ verifies $\delta_0 \leq \gamma_\eta \leq 1 - \delta_0$.
\end{definition}

\begin{lemma}\label{lem:chi_gmm_deriv}
Let $\chi\in C^2(\mathbb{R})$ and $\gamma, \overline{\gamma} \in W^{1,\infty}(0,T)$, with $\delta_0 \leq \gamma \leq 1 - \delta_0$. For any $\eta > 0$, we define the perturbed curve $\gamma_\eta = \gamma + \eta \overline{\gamma}$. Taking $\eta \to 0$, we then have
\begin{equation*}
    \frac{\chi_{\gamma_\eta} - \chi_\gamma}{\eta} \to -\overline{\gamma} \chi'_\gamma \quad \text{in } L^\infty(Q_T).
\end{equation*}
\end{lemma}

\begin{proof}
It is a direct consequence of the Taylor's inequality applied to $\chi$. Indeed, this inequality gives
\begin{equation*}
    \| \chi_{\gamma_\eta} - \chi_\gamma + \eta \overline{\gamma} \chi'_\gamma \|_{L^\infty(Q_T)} \leq \frac{\eta^2}{2} \| \chi'' \|_{L^\infty(\mathbb{R})} \| \overline{\gamma} \|_{L^\infty(0,T)}^2.
\end{equation*}
\end{proof}

\begin{proposition}\label{prop:dJ}
Let $\chi\in C^2(\mathbb{R})$ and $\gamma, \overline{\gamma} \in W^{1,\infty}(0,T)$, with $\delta_0 \leq \gamma \leq 1 - \delta_0$. For any $\eta > 0$, we define the perturbed curve $\gamma_\eta = \gamma + \eta \overline{\gamma}$. If $\overline{\gamma}$ is an admissible perturbation, then the directional derivative of $J$ at $\gamma$ in the direction $\overline{\gamma}$, denoted by $dJ(\gamma; \overline{\gamma})$, reads
\begin{equation}\label{eq:dJ}
    dJ(\gamma; \overline{\gamma}) = \lim_{\eta \to 0} \frac{J(\gamma_\eta) - J(\gamma)}{\eta} = \int_0^T \overline{\gamma} \int_\Omega \varphi^2 \chi'_\gamma,
\end{equation}
where $\varphi$ is the solution of~\eqref{eq:adj_wave} associated to the minimum $(\varphi_0, \varphi_1)$ of $\mathcal{J}^\star_\gamma$.
\end{proposition}

\begin{proof}
For $\eta > 0$ small enough, we denote $(\varphi_0^\eta, \varphi_1^\eta)$ the minimum of $\mathcal{J}^\star_{\gamma_\eta}$, and $\varphi^\eta$ the corresponding solution of~\eqref{eq:adj_wave}. Likewise, we denote $(\varphi_0, \varphi_1)$ the minimum of $\mathcal{J}^\star_\gamma$, and $\varphi$ the corresponding solution of~\eqref{eq:adj_wave}. Using the optimality conditions of $\mathcal{J}^\star_{\gamma_\eta}$ and $\mathcal{J}^\star_\gamma$, we can write
\begin{align*}
    J(\gamma_\eta) - J(\gamma) & = \iint_{q_{\gamma_\eta}} (\varphi^\eta)^2 \chi_{\gamma_\eta} - \iint_{q_\gamma} \varphi^2 \chi_\gamma \\
    & = \left( \langle \varphi_1^\eta, y_0 \rangle_{H^{-1}, H^1_0} - \int_\Omega \varphi_0^\eta y_1 \right) - \left( \langle \varphi_1, y_0 \rangle_{H^{-1}, H^1_0} - \int_\Omega \varphi_0 y_1 \right) \\
    & = \iint_{q_\gamma} \varphi^\eta \varphi \chi_\gamma - \iint_{q_{\gamma_\eta}} \varphi \varphi^\eta \chi_{\gamma_\eta} = -\iint_{Q_T} \varphi^\eta \varphi (\chi_{\gamma_\eta} - \chi_\gamma).
\end{align*}
Arguying as in the proof of Proposition~\ref{prop:J_cont}, we can show that $\varphi^\eta \rightharpoonup \varphi$ weakly in $L^2(0,T; L^2(\Omega))$. As a result, we have
\begin{align*}
    \frac{J(\gamma_\eta) - J(\gamma)}{\eta} & = \iint_{Q_T} \varphi^\eta \varphi \overline{\gamma} \chi'_\gamma - \iint_{Q_T} \varphi^\eta \varphi \left( \frac{\chi_{\gamma_\eta} - \chi_\gamma}{\eta} + \overline{\gamma} \chi'_\gamma \right) \\
    & \to \iint_{Q_T} \varphi^2 \overline{\gamma} \chi'_\gamma = \int_0^T \overline{\gamma} \int_\Omega \varphi^2 \chi'_\gamma.
\end{align*}
Indeed, we can take the weak limit in the first term because $\varphi \overline{\gamma} \chi'_\gamma \in L^2(0,T; L^2(\Omega))$. Using Lemma~\ref{lem:chi_gmm_deriv} and the boundedness of $(\varphi^\eta)_{\eta > 0}$ in $L^2(0,T; L^2(\Omega))$, the second term converges to $0$ because
\begin{equation*}
    \left| \iint_{Q_T} \varphi^\eta \varphi \left( \frac{\chi_{\gamma_\eta} - \chi_\gamma}{\eta} + \overline{\gamma} \chi'_\gamma \right) \right| \leq \| \varphi^\eta \|_{L^2(L^2)} \| \varphi \|_{L^2(L^2)} \left\| \frac{\chi_{\gamma_\eta} - \chi_\gamma}{\eta} + \overline{\gamma} \chi'_\gamma \right\|_{L^\infty(Q_T)}.
\end{equation*}
\end{proof}

\begin{remark}
We emphasize that the directional derivative does not depend on the solution of an adjoint problem. This is due to the fact that we minimize with respect to the curve $\gamma$ over controls of minimal $L^2(q_{\gamma})$-norm. We refer to the proof of \cite[Theorem 2.3]{Munch08} for more details. 
\end{remark}

\subsection{Regularization and Gradient algorithm}
\label{ssect:grad_algo}

At the practical level, in order to solve the optimal problem~\eqref{eq:min_J} numerically, we need to handle the Lipschitz constraint included in $\mathcal{G}_\text{ad}$. In this respect, we add a regularizing term to the cost $J$ in order to keep the derivative of $\gamma$ uniformly bounded. The optimization problem is now the following one: for $\epsilon > 0$ fixed, solve
\begin{equation}\label{eq:min_Jeps}
    \min_{\substack{\gamma \in W^{1,\infty}(0,T)\\ \delta_0 \leq \gamma \leq 1 - \delta_0}} J_\epsilon(\gamma) = J(\gamma) + \frac{\epsilon}{2} \| \gamma' \|_{L^2(0,T)}^2.
\end{equation}
The regularization parameter $\epsilon$, which can be compared to the Lipschitz constant $M$ in the definition of $\mathcal{G}_\text{ad}$, controls the speed of variation of the curves $\gamma \in W^{1,\infty}(0,T)$.

We fix $\gamma \in W^{1,\infty}(0,T)$ such that $\delta_0 \leq \gamma \leq 1 - \delta_0$. Using Proposition~\ref{prop:dJ}, for any admissible perturbation $\overline{\gamma} \in W^{1,\infty}(0,T)$, a direct calculation provides the expression of the directional derivative of $J_{\epsilon}$
\begin{equation}\label{eq:dJeps}
    dJ_\epsilon(\gamma; \overline{\gamma}) = dJ(\gamma; \overline{\gamma}) + \epsilon \int_0^T \gamma' \overline{\gamma}' = \langle j_\gamma, \overline{\gamma} \rangle_{L^2(0,T)} + \epsilon \langle \gamma', \overline{\gamma}' \rangle_{L^2(0,T)},
\end{equation}
with
\begin{equation}\label{eq:j_gmm}
    j_\gamma(t) = \int_\Omega \varphi^2(x,t) \chi'_\gamma(x,t) \,dx, \quad \forall t \in [0,T].
\end{equation}
In the expression of $j_\gamma$, the function $\varphi$ is the solution of~\eqref{eq:adj_wave} associated to the minimum $(\varphi_0, \varphi_1)$ of $\mathcal{J}^\star_\gamma$. Consequently, a minimizing sequence $(\gamma_n)_{n \in \mathbb{N}}$ for $J_{\epsilon}$ is defined as follows:
\begin{equation}
    \left\{
    \begin{array}{l}
    \gamma_0 \text{ given in } H^1(0,T), \\
    \gamma_{n+1} = P_{[\delta_0, 1 - \delta_0]}(\gamma_n - \rho j^{\epsilon}_{\gamma_n}), \quad \text{for } n \geq 0,
    \end{array}
    \right.
\end{equation}
where $P_{[\delta_0, 1 - \delta_0]}$ is the pointwise projection in the interval $[\delta_0, 1 - \delta_0]$, $\rho>0$ a descent step and $j^{\epsilon}_{\gamma_n}\in H^1(0,T)$ is the solution of the variational formulation
\begin{equation}\label{eq:jreg_var}
    \langle j^{\epsilon}_{\gamma_n}, \widetilde{\gamma} \rangle_{L^2(0,T)} + \epsilon \langle j^{\epsilon\prime}_{\gamma_n}, \widetilde{\gamma}' \rangle_{L^2(0,T)} = \langle j_{\gamma_n}, \widetilde{\gamma} \rangle_{L^2(0,T)} + \epsilon \langle \gamma^\prime_n, \widetilde{\gamma}' \rangle_{L^2(0,T)}, \quad \forall \widetilde{\gamma} \in H^1(0,T),
\end{equation}
so that $dJ_\epsilon(\gamma_n; j^{\epsilon}_{\gamma_n}) = \| j^{\epsilon}_{\gamma_n} \|_{L^2(0,T)}^2 + \epsilon \| j^{\epsilon '}_{\gamma_n} \|_{L^2(0,T)}^2 \geq 0$.

\section{Numerical experiments}
\label{sect:num_simu}

Before to present some numerical experiments, let us briefly mention some aspects of the resolution of the underlying discretized problem.
\vspace{0.2cm}
\par\noindent
$\bullet$ The discretization of the curve $\gamma$ is performed as follows. For any fixed integer $\mathcal{N} > 0$, we denote $\delta t = T/\mathcal{N}$ and define the uniform subdivision $\{t_i\}_{i=0,\cdots, \mathcal{N}}$ of $[0,T]$ such that $t_i= i \delta t$. We then approximate the curve $\gamma$ in the space of dimension $\mathcal{N}+1$
\begin{equation*}
    \mathbbm{P}_1^{\delta t} = \big\{ \gamma \in C([0,T]); \quad \gamma |_{[t_{i-1}, t_i]} \text{ affine},\ \forall i \in \{ 1,\dots,\mathcal{N} \} \big\}.
\end{equation*}
For any $\gamma \in \mathbbm{P}_1^{\delta t}$, $\gamma = \sum_{i=0}^\mathcal{N} \gamma^i L_i^{\delta t}$, with $(\gamma^i)_{0 \leq i \leq \mathcal{N}} \in \Omega^{\mathcal{N}+1}$ where $(L_i^{\delta t})_{0 \leq i \leq \mathcal{N}}$ is the usual Lagrange basis. Consequently, $\gamma$ is defined by the $\mathcal{N}+1$ points $(\gamma^i,t_i)\in \Omega \times [0,T]$. The knowledge of the initial curve $\gamma_0 \in \mathbb{P}_1^{\delta t}$ such that $\delta_0 \leq \gamma_0 \leq 1 - \delta_0$ determines such points and then a triangular mesh of $Q_T$. At each iteration $n\geq 0$, these points are updated along the $x$-axis according to the pointwise time descent direction $j^{\epsilon}_{\gamma_n}\in H^1(0,T)$ (see \eqref{eq:jreg_var}) as follows:
\[
x_i^{n+1}=P_{[\delta_0,1-\delta_0]}\biggl(x_i^n-\rho \, j^{\epsilon}_{\gamma_n}(t_i^n)\biggr), \qquad t_i^{n+1}=t_i^n  \quad \forall i=0,\cdots, \mathcal{N}+1.
\]

We emphasize in particular that a remeshing of $Q_T$ is performed at each iteration $n$ according to the set of points $(x_i^{n},t_i^n)_{i=0,\cdots,\mathcal{N}+1}$.
\vspace{0.2cm}
\par\noindent
$\bullet$ Each iteration of the algorithm requires the numerical approximation of the control of minimal $L^2(q_{\gamma_n})$ for the initial data $(y_0,y_1)$. We use the space-time method described in~\cite[Sections~3-4]{CastroCindeaMunch14} which is very well-adapted to the description of $\gamma$ embedded in a space-time mesh of $Q_T$. The minimization of the conjugate functional $\mathcal{J}^\star_{\gamma_n}$ (see \eqref{eq:J*_gmm}) with respect to $(\varphi_0,\varphi_1)\in \mathbf{V}$ is replaced by the search of the unique saddle-point of the Lagrangian $\mathcal{L}:\mathbf{Z}\times L^2(0,T; H^1_0(\Omega))\to \mathbb{R}$ defined by
\[
\mathcal{L}(\varphi,y)=\frac{1}{2}\Vert \varphi \Vert^2_{L^2(q_{\gamma_n})}-\int_0^T \langle y,L \varphi \rangle_{H_0^1(\Omega),H^{-1}(\Omega)} dt
+ \langle \varphi_t(\cdot,0), y_0 \rangle_{H^{-1}(\Omega),H_0^1(\Omega)} - \langle \varphi(\cdot,0),y_1 \rangle_{L^2(\Omega)},
\]
with $\mathbf{Z}=C^1([0,T];H^{-1}(\Omega))\cap C([0,T];L^2(\Omega))$. The corresponding mixed formulation is solved with a conformal space-time finite element method while a direct method is used to invert the discrete matrix. The interesting feature of the method for which the adaptation of the mesh is very simple to handle with, is that only a small part of the matrix - corresponding to the term $\Vert \varphi\Vert^2_{L^2(q_{\gamma_n})}$ - is modified from two consecutive iterations $n$ and $n+1$.

\subsection{Numerical illustrations}
\label{ssect:num_simu}

We discuss several experiments performed with FreeFEM (see~\cite{Hecht2012}) for various initial data and control domains. We notably use an UMFPACK type solver. 

We fix $\delta_0 = 0.15$ and $\delta = \delta_0/4$. Moreover, according to~\eqref{eq:chi}, we define the function $\chi\in C^2(\mathbb{R})$ in $[\delta_0 - \delta, \delta_0]$ as the unique polynomial of degree $5$ such that $\chi(\delta_0 - \delta) = 1$, $\chi(\delta_0)=\chi'(\delta_0 - \delta) = \chi'(\delta_0) =\chi''(\delta_0 - \delta) = \chi''(\delta_0) = 0$ and vanishing on $\mathbb{R}\setminus [\delta_0-\delta,\delta_0]$. 

Concerning the stopping criterion for the descent algorithm, we observed that the usual one based on the relative quantity $|J_\epsilon^n - J_\epsilon^{n-1}|/J_\epsilon^0$ is inefficient because too noisy. 
This is due to the  uncertainty on the numerical computation of the adjoint state $\varphi^n$ and the perturbation $j^{\epsilon}_{\gamma_n}$. Consequently, in order to better capture the variations of the sequence $(J_\epsilon^n)_{n \in \mathbb{N}}$, we replace $J_\epsilon^n$ and $J_\epsilon^{n-1}$ by the right and left $p$-point average respectively leading to the stopping criterion
\begin{equation}\label{eq:stop_crit}
    \Delta J_\epsilon^n \coloneqq \left| \frac{\frac{1}{p} \sum_{i=n}^{n+p-1} J_\epsilon^i - \frac{1}{p} \sum_{i=n-p}^{n-1} J_\epsilon^i}{J_\epsilon^0} \right| < \eta, \quad \text{for } p \in \mathbb{N}^* \text{ fixed}.
\end{equation}
In the sequel, we fix $p = 10$ and $\eta = 10^{-3}$. 

Furthermore, in order to measure the gain obtained by using non-cylindrical control domains rather than cylindrical ones, we define a performance index associated to each optimal curves $\gamma_\text{opt}$;  identifying any constant curve $\gamma \equiv x_0$ with its value $x_0 \in [\delta_0, 1 - \delta_0]$, we compute the minimal cost $\min_{x_0} J_\epsilon(x_0)$ for cylindrical domains. We then define the performance index of $\gamma_\text{opt}$ by
\begin{equation}\label{perf_ind}
    \Pi(\gamma_\text{opt}) = 100 \left( 1 - \frac{J_\epsilon(\gamma_\text{opt})}{\min_{x_0} J_\epsilon(x_0)} \right).
\end{equation}

In the sequel, in practice, the minimum of $J_\epsilon$ with respect to $x_0$ is searched among $13$ distincts values equi-distributed between $0.2$ and $0.8$.

\vspace{0.2cm}
\par\noindent
$\bullet$ We first consider the regular initial datum $(y_0,y_1)$ given by
\begin{equation}\label{eq:EX1}\tag{\bf EX1}
    y_0(x) = \sin(2 \pi x), \quad y_1(x) = 0, \qquad \text{for } x \in (0,1).
\end{equation}
and  $T = 2$,  $\epsilon = 10^{-2}$, $\rho=10^{-4}$.  We initialize the descent algorithm with the following three initial curves:
\begin{equation}\label{eq:EX1_gmm0}
    \gamma_0^1(t) = \frac{2}{5}, \quad \gamma_0^2(t) = \frac{3}{5}, \quad \gamma_0^3(t) = \frac{1}{2} + \frac{1}{10} \cos \left( \pi \frac{t}{T} \right), \qquad \text{for } t \in (0,T).
\end{equation}
The corresponding initial and optimal domains are depicted in Figure~\ref{fig:simu_EX1_gmm} together with typical space-time meshes. The numbers of iterations until convergence, the values of the functional $J_\varepsilon$ evaluated at the optimal curve $\gamma_\text{opt}$ and the performance indices of $\gamma_\text{opt}$ are listed in Table~\ref{tab:simu_EX1}.

\begin{table}[ht!]
\centering
\begin{tabular}{c|ccc}
Initial curve & $\gamma_0^1$ & $\gamma_0^2$ & $\gamma_0^3$ \\
\hline
Number of iterations & $33$ & $33$ & $84$ \\
$J_\epsilon(\gamma_\text{opt})$ & $47.09$ & $47.09$ & $47.93$ \\
$\Pi(\gamma_\text{opt})$ & $-0.32$\% & $-0.32$\% & $-2.11$\%
\end{tabular}
\caption{{\bf(EX1)} - Number of iterations, optimal value of the functional $J_\varepsilon$ and performance index, for the initial curves $(\gamma_0^i)_{i \in \{ 1,2,3 \}}$ given by~\eqref{eq:EX1_gmm0}.}
\label{tab:simu_EX1}
\end{table}

In Figure~\ref{fig:simu_EX1_gmm}, we observe that the optimal domain computed by the algorithm depends on the initial domain chosen. This indicates that our functional $J_\varepsilon$ does have several local minima. Moreover, one can show that, among the cylindrical domains, there are two optimal values, $x_0 = 1/4$ and $x_0 = 3/4$, leading to $J_{\varepsilon}(x_0)\approx 46.94$. These values correspond to the extrema of the function $\sin(2 \pi x)$ in $[0,1]$. The simulations associated with the initial curves $\gamma_0^1$ and $\gamma_0^2$ are in agreement with this result. On the other hand, the worst cylindrical domain corresponds to $x_0 = 1/2$ (see Figure~\ref{fig:simu_EX12_JCyl}-Left).

\begin{figure}[ht!]
\centering
\begin{tabular}{ccc}
\includegraphics[width=0.3\textwidth]{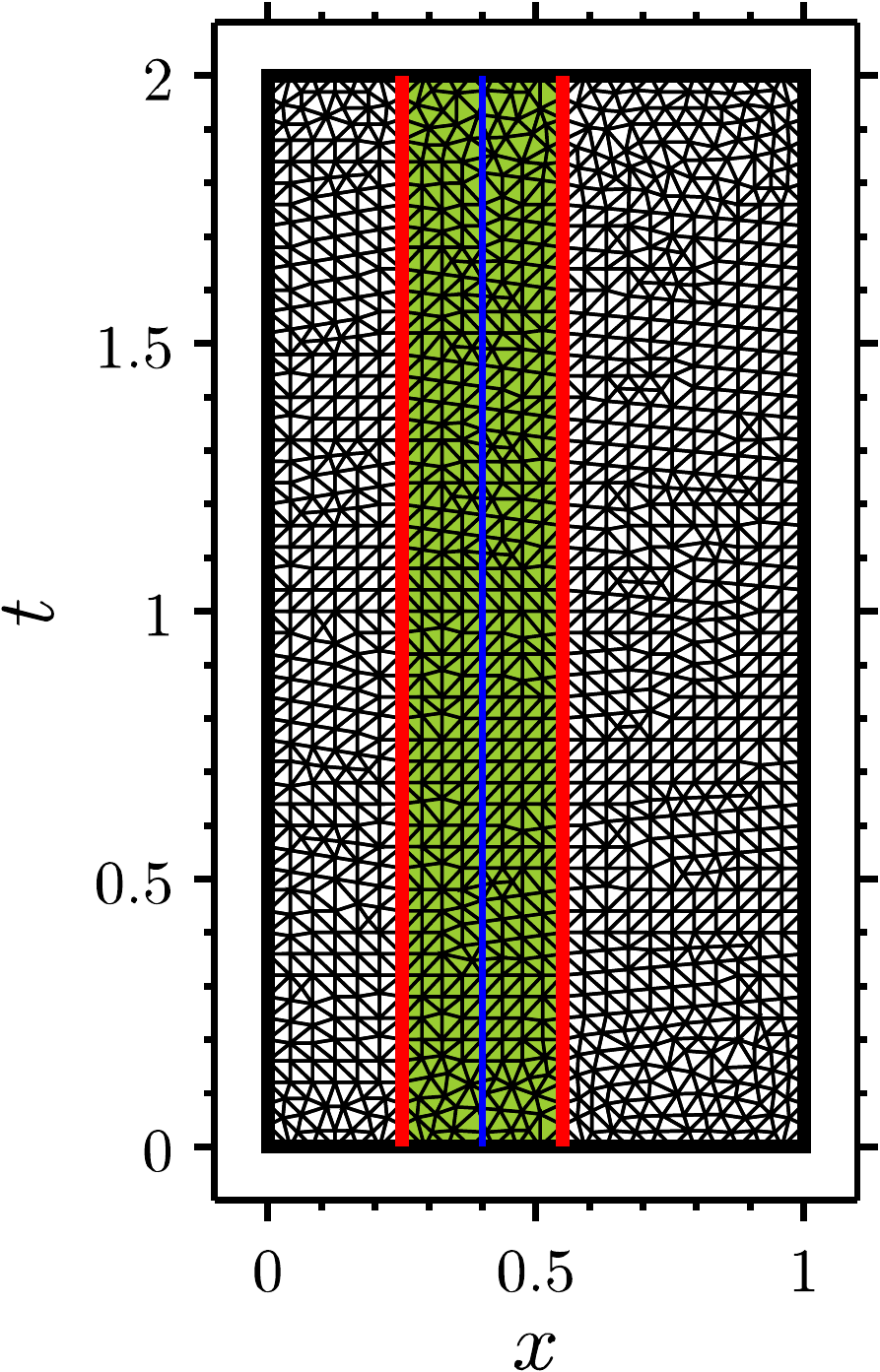} &
\includegraphics[width=0.3\textwidth]{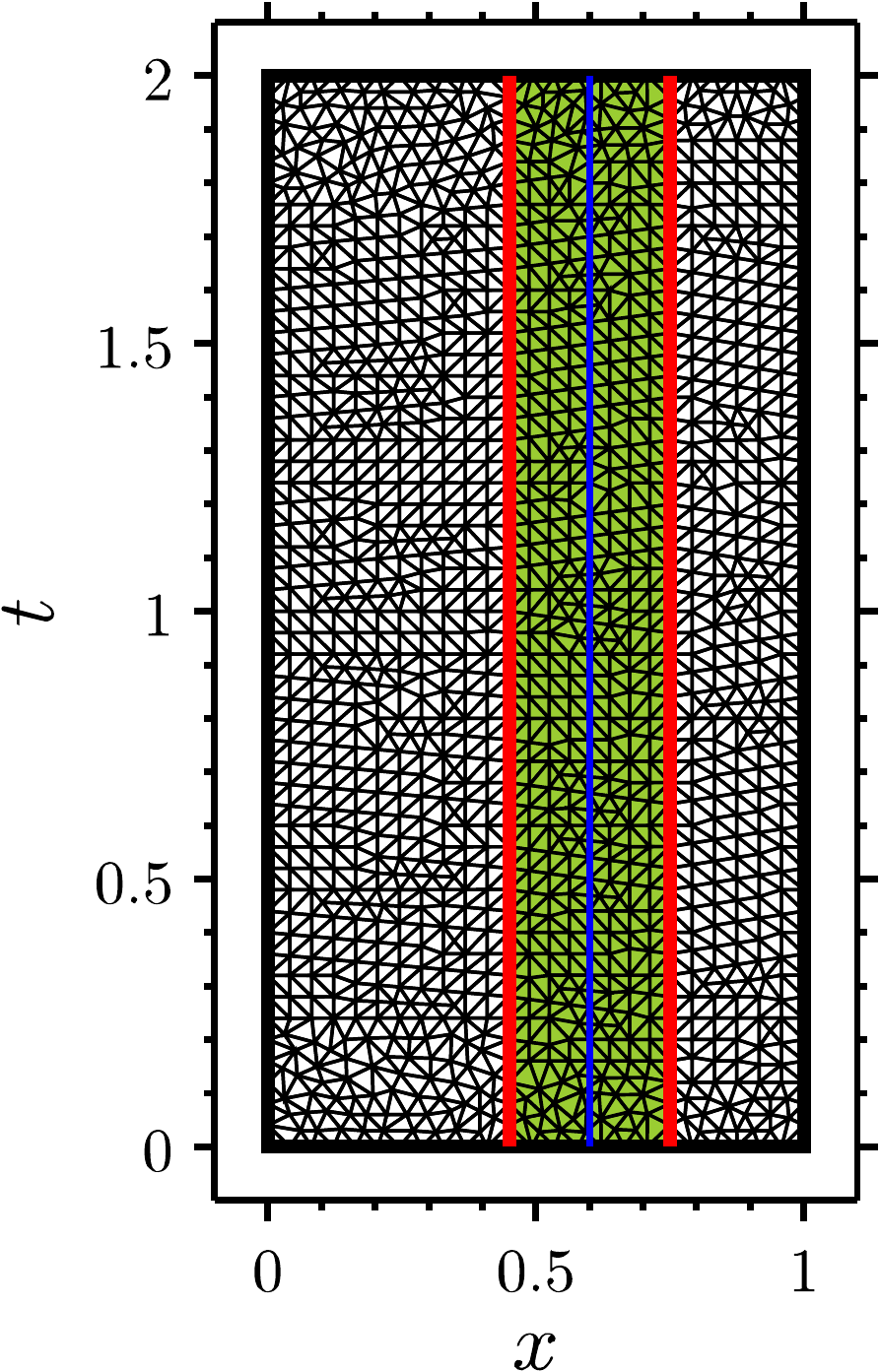} &
\includegraphics[width=0.3\textwidth]{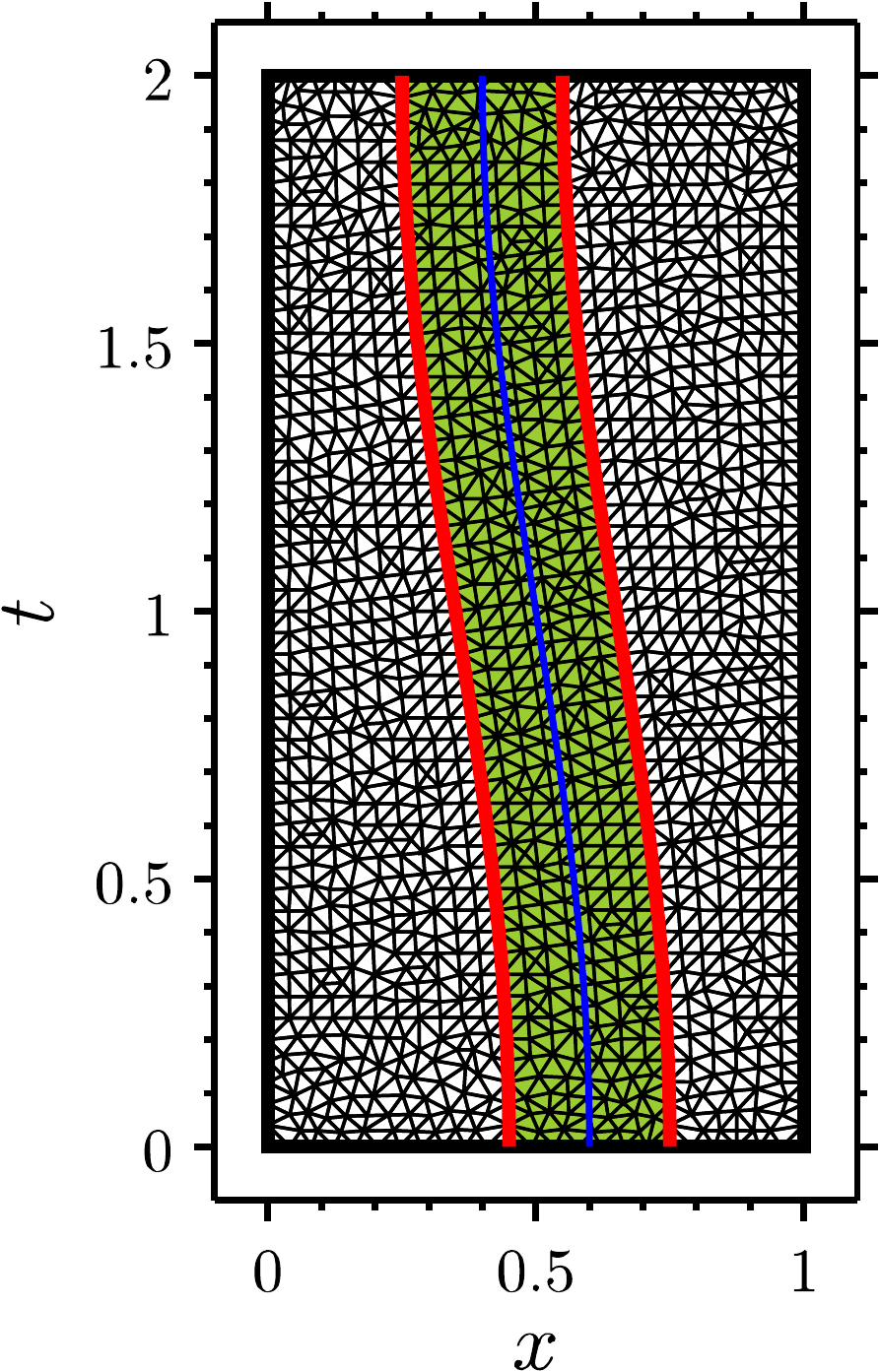}
\tabularnewline
\includegraphics[width=0.3\textwidth]{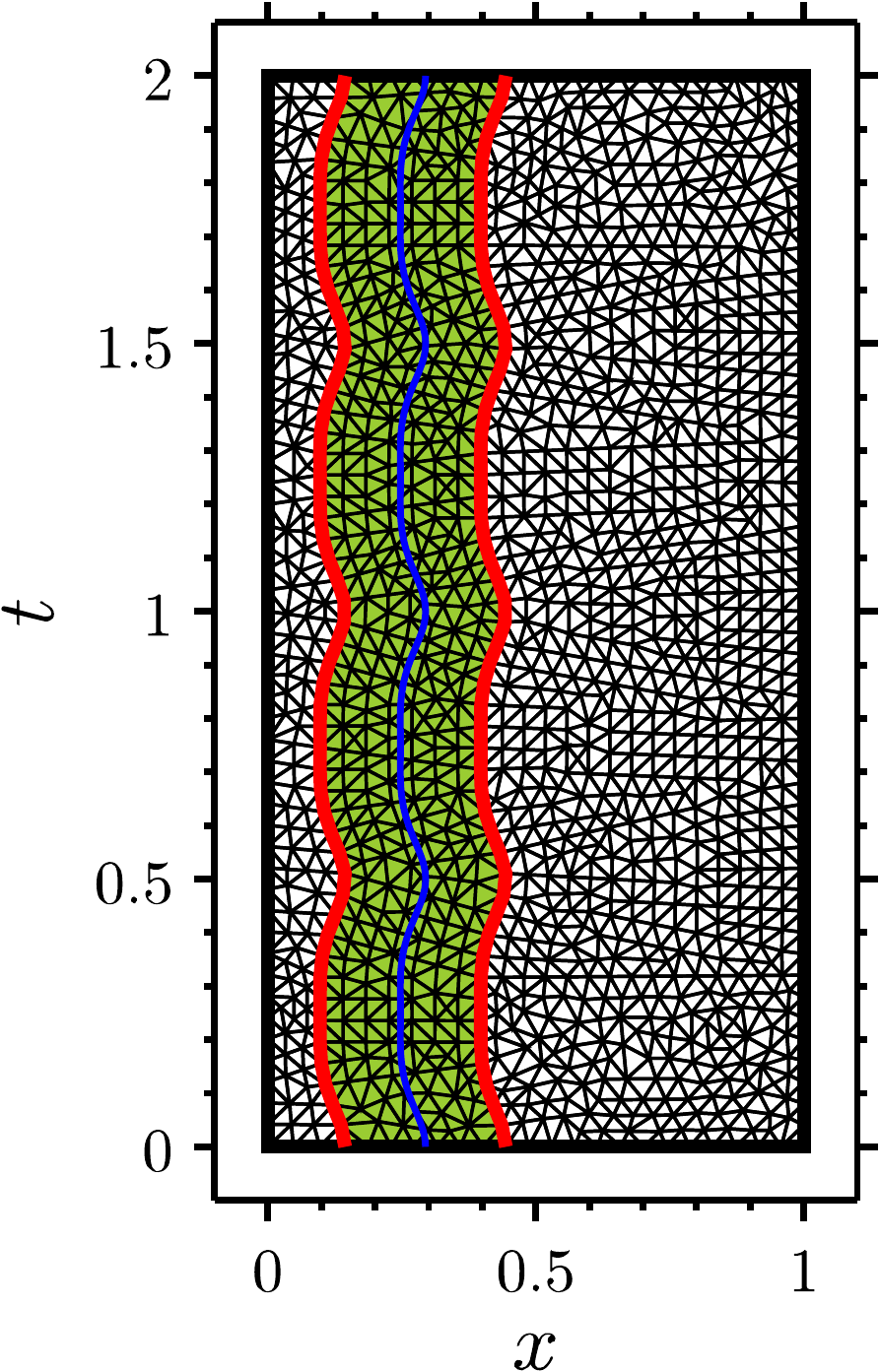} &
\includegraphics[width=0.3\textwidth]{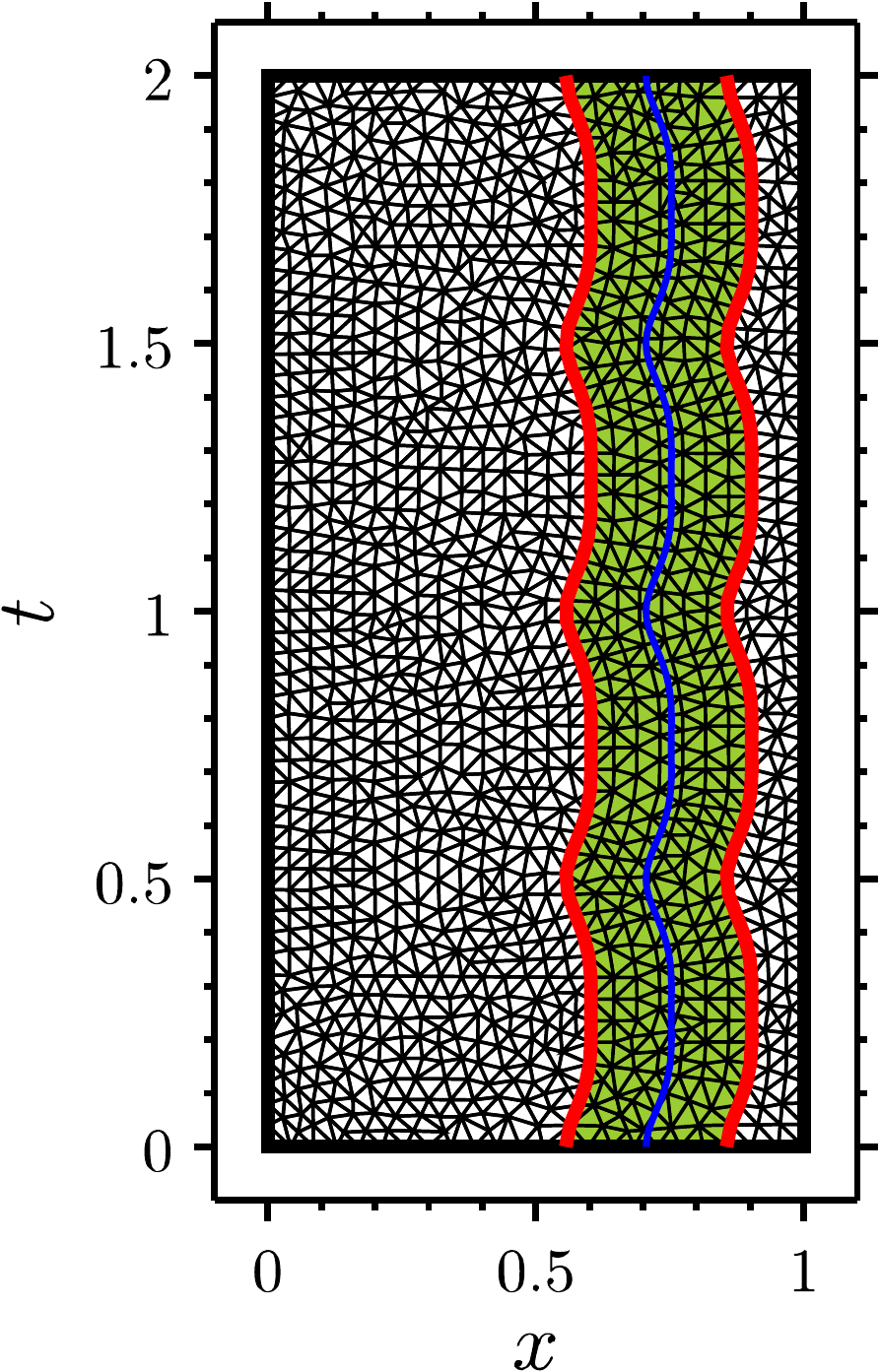} &
\includegraphics[width=0.3\textwidth]{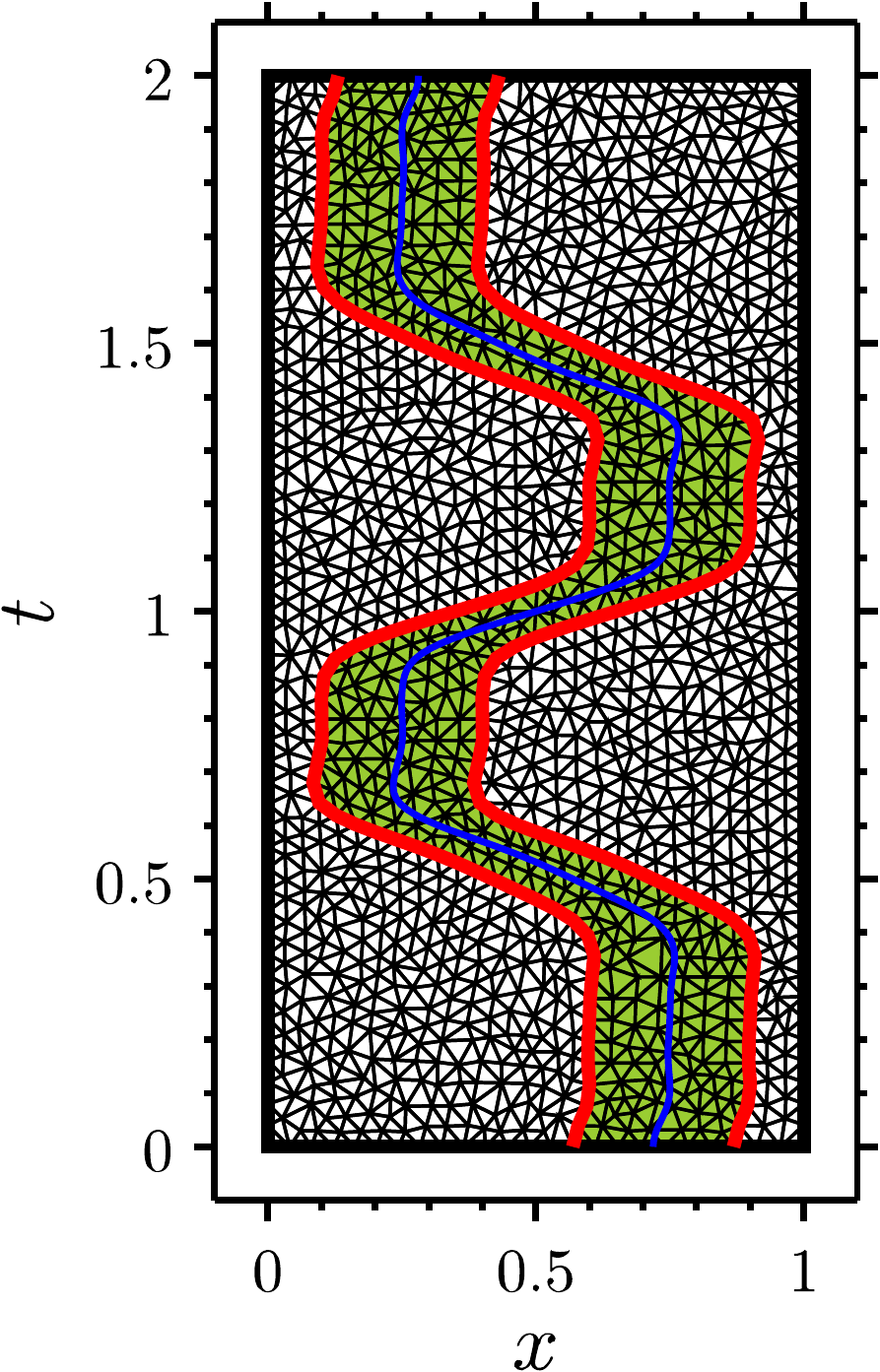}
\end{tabular}
\caption{{\bf (EX1)} - Initial (top) and optimal (bottom) control domains for the initial curves $(\gamma_0^i)_{i \in \{ 1,2,3 \}}$ given by~\eqref{eq:EX1_gmm0} (from left to right).}
\label{fig:simu_EX1_gmm}
\end{figure}

Eventually, the adjoint states $\varphi$ (from which we obtain the control  $v = \varphi |_{q_\gamma^n}$) computed for the optimal domains in Figure~\ref{fig:simu_EX1_gmm}-Bottom, are displayed in Figure~\ref{fig:simu_EX1_phi}.

\begin{figure}[ht!]
\centering
\begin{tabular}{ccc}
\includegraphics[width=0.3\textwidth]{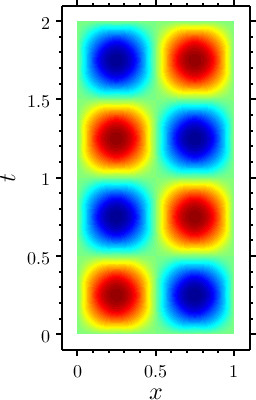} &
\includegraphics[width=0.3\textwidth]{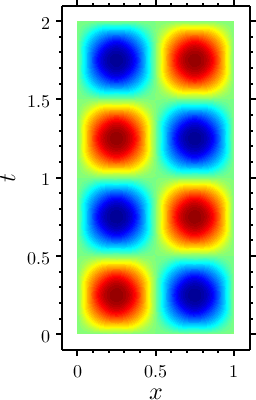} &
\includegraphics[width=0.3\textwidth]{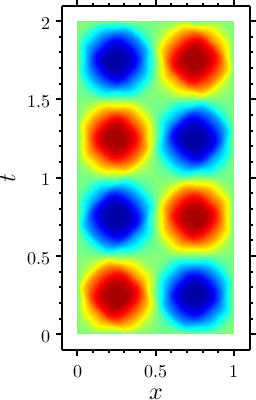}
\end{tabular}
\caption{{\bf (EX1)} -  Isovalues of the adjoint states $\varphi$ computed for the optimal domains obtained for the initial curves $(\gamma_0^i)_{i \in \{ 1,2,3 \}}$ given by~\eqref{eq:EX1_gmm0} (from left to right).}
\label{fig:simu_EX1_phi}
\end{figure}

\begin{figure}[ht!]
\centering
\begin{tabular}{cc}
\includegraphics[width=0.45\textwidth]{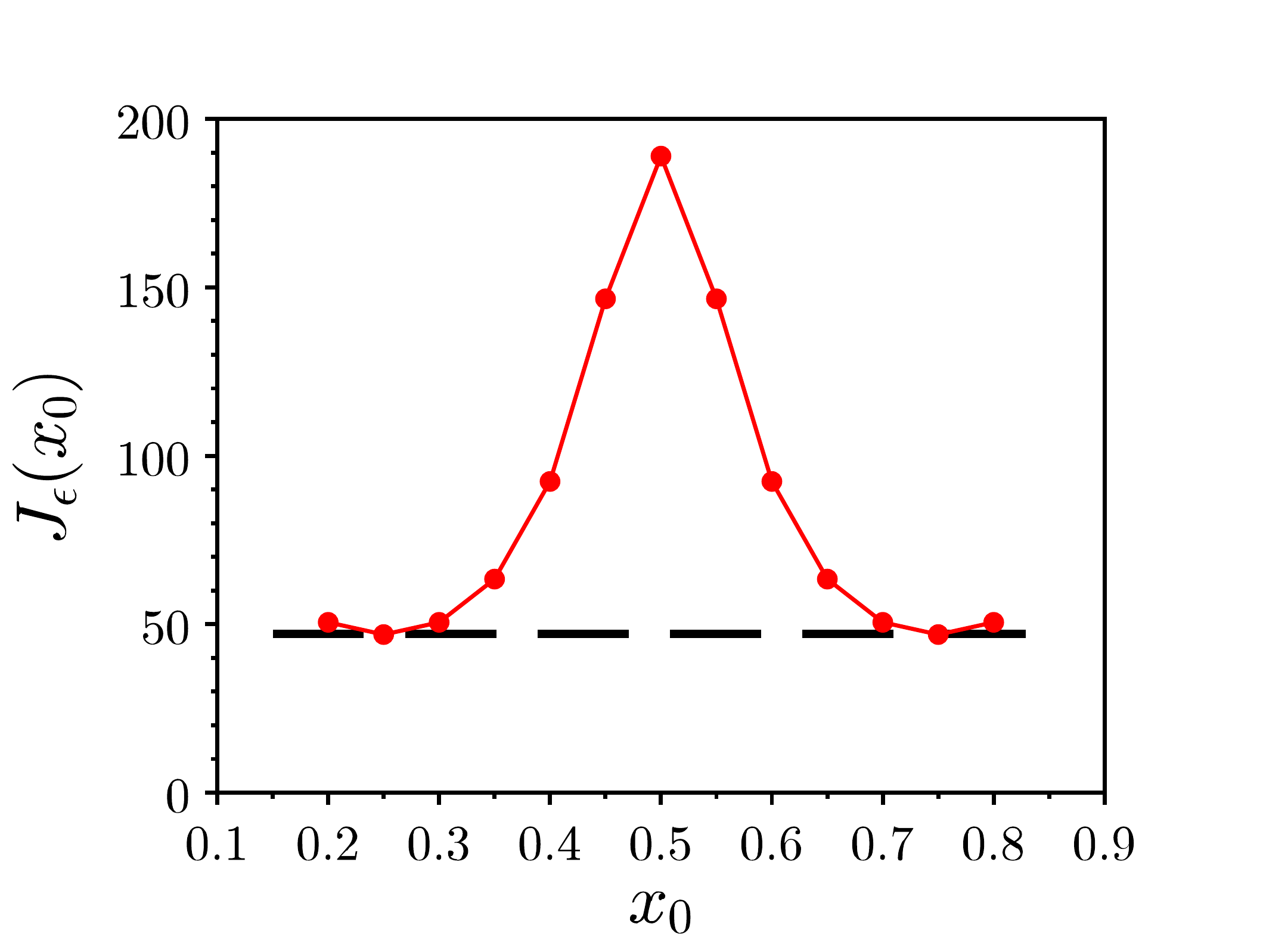} & \includegraphics[width=0.45\textwidth]{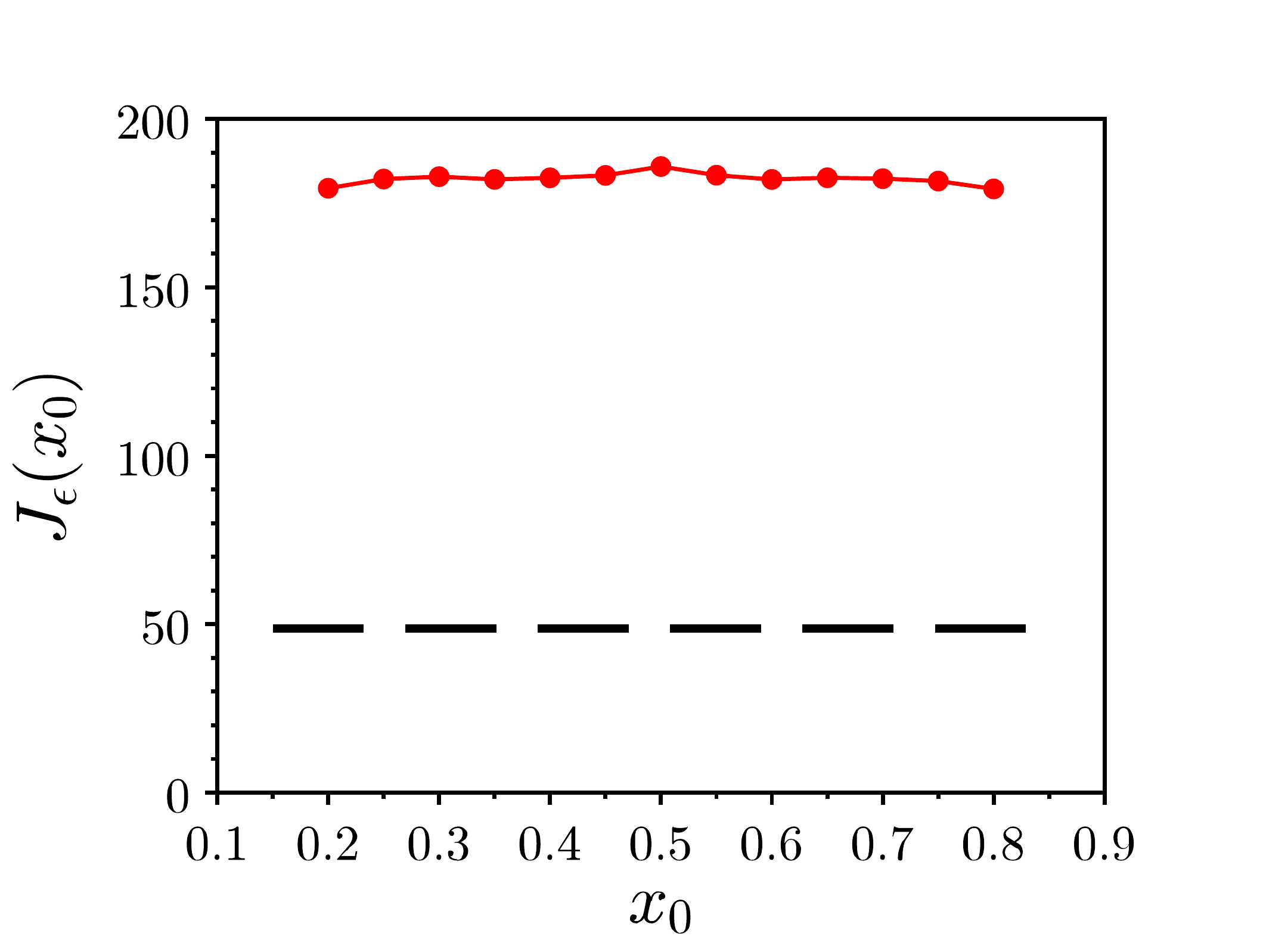}
\end{tabular}
\caption{Values of $J_\epsilon$ for constant curves $\gamma \equiv x_0$ ({\color{red} \textbullet}), for the initial data~\eqref{eq:EX1} (left) and~\eqref{eq:EX2} (right). The dashed line ({\color{black} \textbf{- -}}) represents the value of $J_\epsilon(\gamma_\text{opt})$, for the initial curves $\gamma_0 \equiv 2/5$ (left) and $\gamma_0 \equiv 1/2$ (right).}
\label{fig:simu_EX12_JCyl}
\end{figure}

\vspace{0.2cm}
\par\noindent
$\bullet$ We now consider the initial datum $(y_0,y_1)$ given by
\begin{equation}\label{eq:EX2}\tag{\bf EX2}
    y_0(x) = (10 x - 4)^2 (10 x - 6)^2 \mathbbm{1}_{[0.4, 0.6]}(x), \quad y_1(x) = y_0'(x), \qquad \text{for } x \in (0,1).
\end{equation}
This initial condition, plotted in Figure~\ref{fig:simu_EX2_y0}, generates a travelling wave, as can be seen in Figure~\ref{fig:simu_EX2_gmm}.3.

\begin{figure}[ht!]
\centering
\begin{tabular}{cc}
\includegraphics[width=0.45\textwidth]{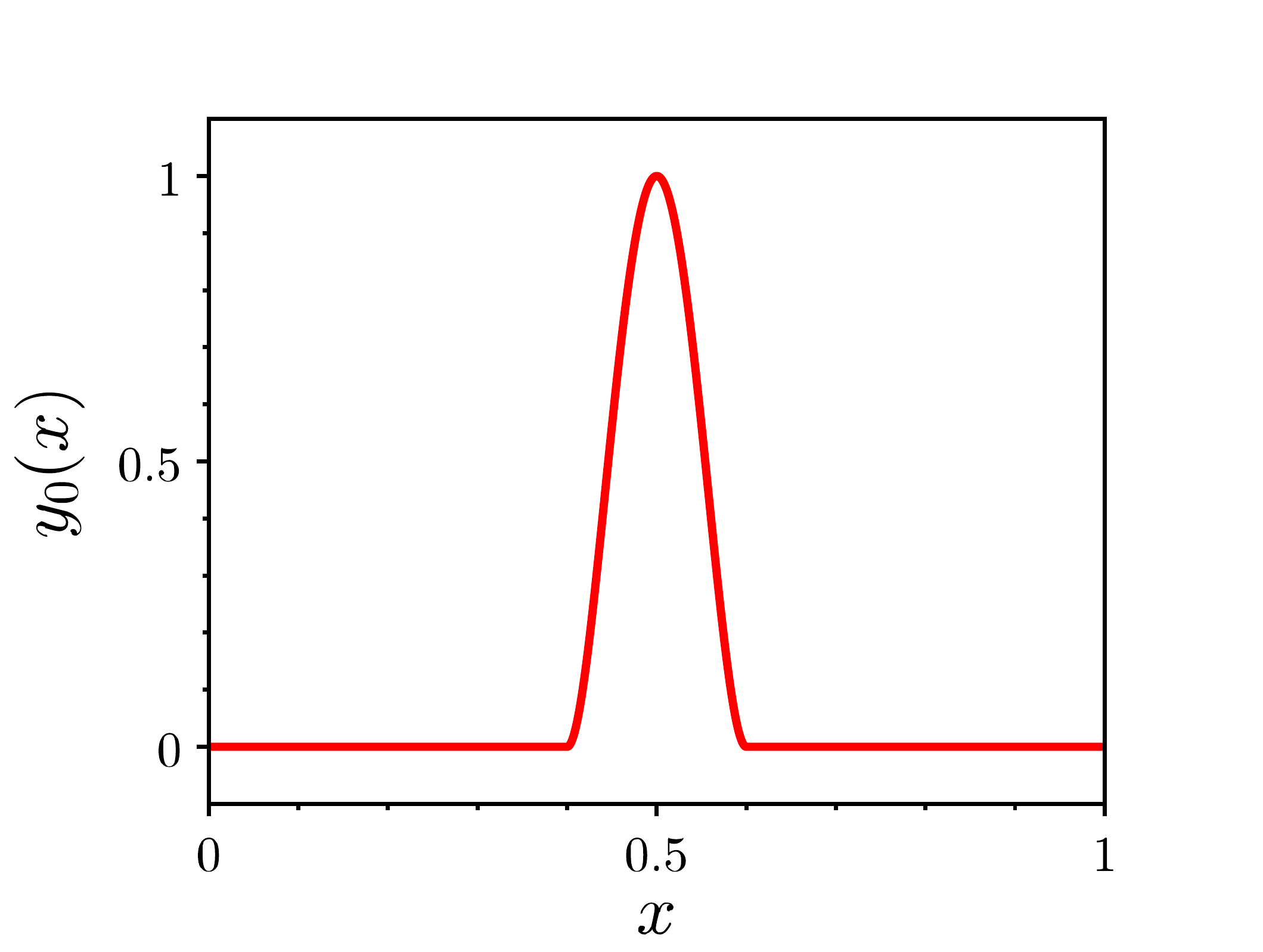} &
\includegraphics[width=0.45\textwidth]{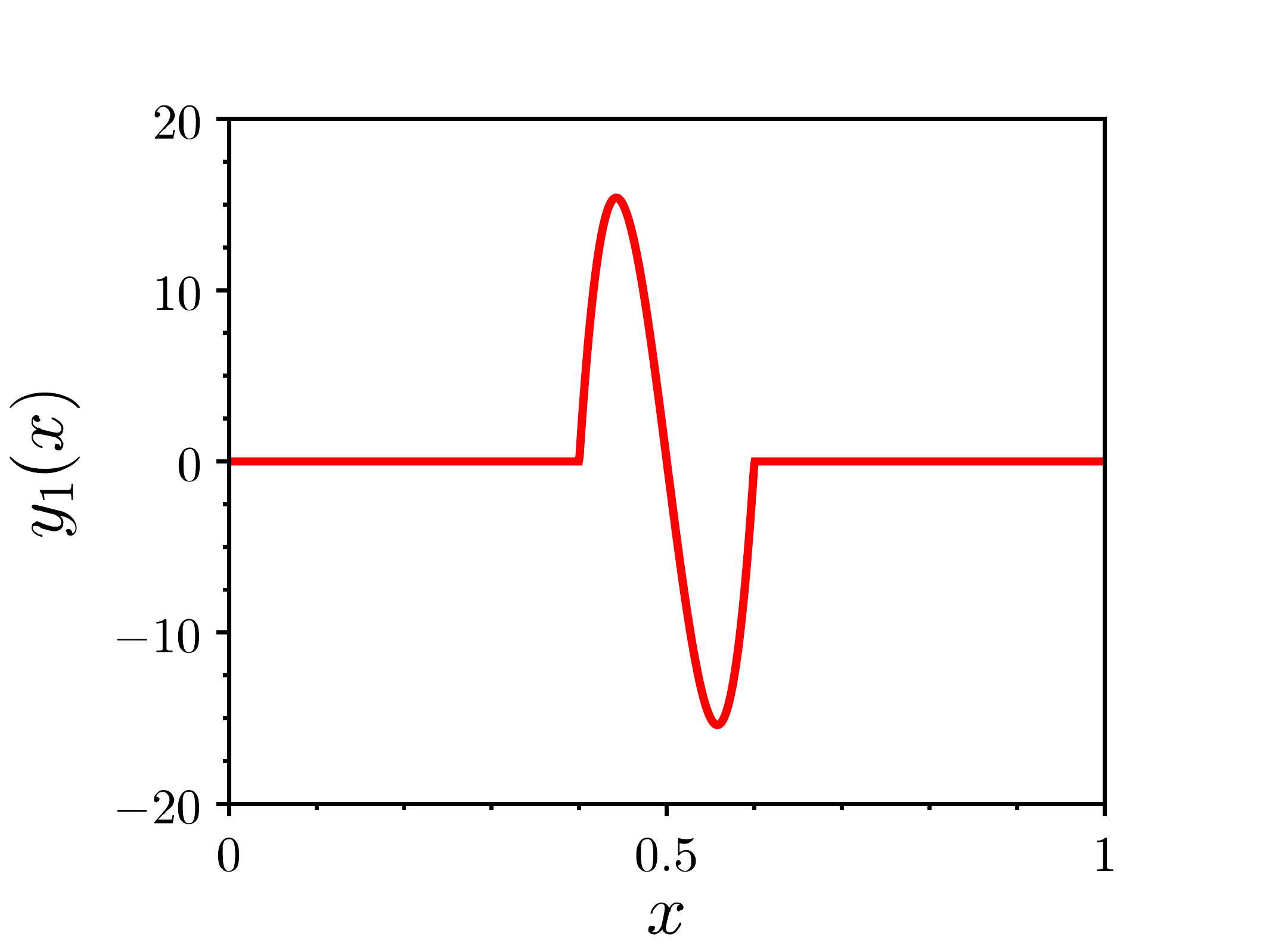}
\end{tabular}
\caption{Initial datum $(y_0,y_1)$ defined in~\eqref{eq:EX2}.}
\label{fig:simu_EX2_y0}
\end{figure}

For $T = 2$, $\epsilon = 10^{-2}$ and $\rho=10^{-4}$, we initialize the descent algorithm with the curve $\gamma_0 \equiv 1/2$. The convergence is reached after $68$ iterations and the optimal cost is $J_\epsilon(\gamma_\text{opt}) \approx 48.70$. Moreover, the minimal cost for cylindrical domains is $\min_{x_0} J_\epsilon(x_0) \approx 179.22$ leading to a performance index $\Pi(\gamma_\text{opt}) \approx 72.83$\%. The non-cylindrical setup is in that case much more efficient that the cylindrical one. It is due to the fact that the domains we consider can follow very closely the propagation of the travelling wave. This can be noticed in Figure~\ref{fig:simu_EX2_gmm}, where we display the optimal control domain, the corresponding adjoint state $\varphi$, the uncontrolled and controlled solutions over the optimal domain.

\begin{figure}[ht!]
\hspace{-0.5cm}
\centering
\begin{tabular}{cccc}
\includegraphics[width=0.23\textwidth]{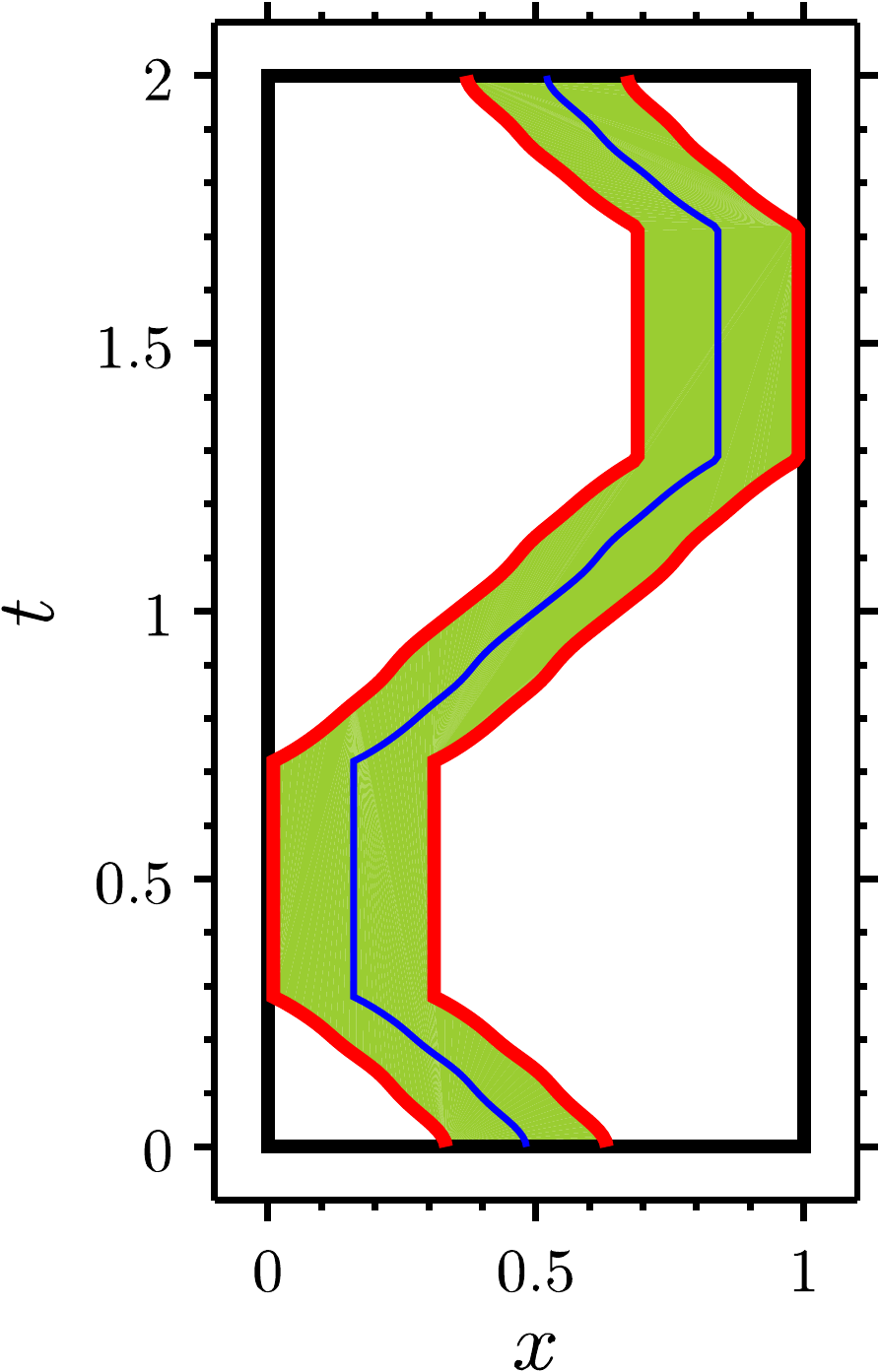} &
\includegraphics[width=0.23\textwidth]{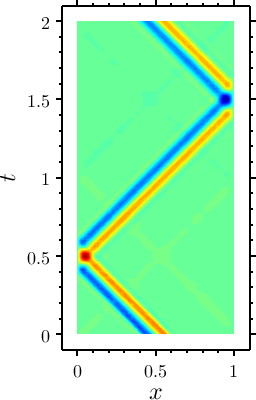} &
\includegraphics[width=0.23\textwidth]{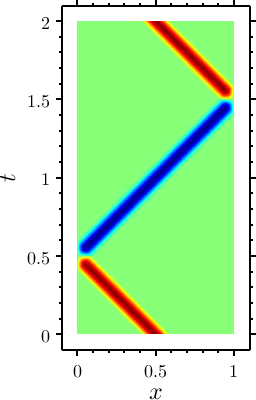} &
\includegraphics[width=0.23\textwidth]{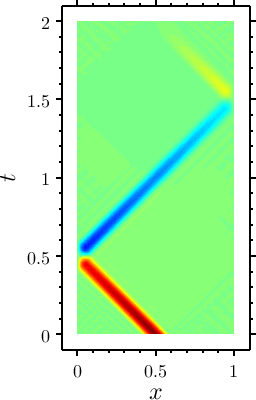}
\end{tabular}
\caption{{\bf (EX2)} -  From left to right, optimal control domain, isovalues of the corresponding adjoint state $\varphi$, isovalues of the uncontrolled and controlled wave over the optimal domain, for the initial curve $\gamma_0 \equiv 1/2$.}
\label{fig:simu_EX2_gmm}
\end{figure}

The evolution of the cost $J_\epsilon^n$ and the derivative $dJ_\epsilon^n$ with respect to $n$ are displayed in Figure~\ref{fig:simu_EX2_J}. Figure~\ref{fig:simu_EX12_JCyl}-Right depicts the values of the functional $J_\epsilon$ for the constant curves $\gamma \equiv x_0$ used to determine the best cylindrical domain and highlights the low variation of the cost with respect to the position of such domains.

\begin{figure}[ht!]
\centering
\begin{tabular}{cc}
\includegraphics[width=0.48\textwidth]{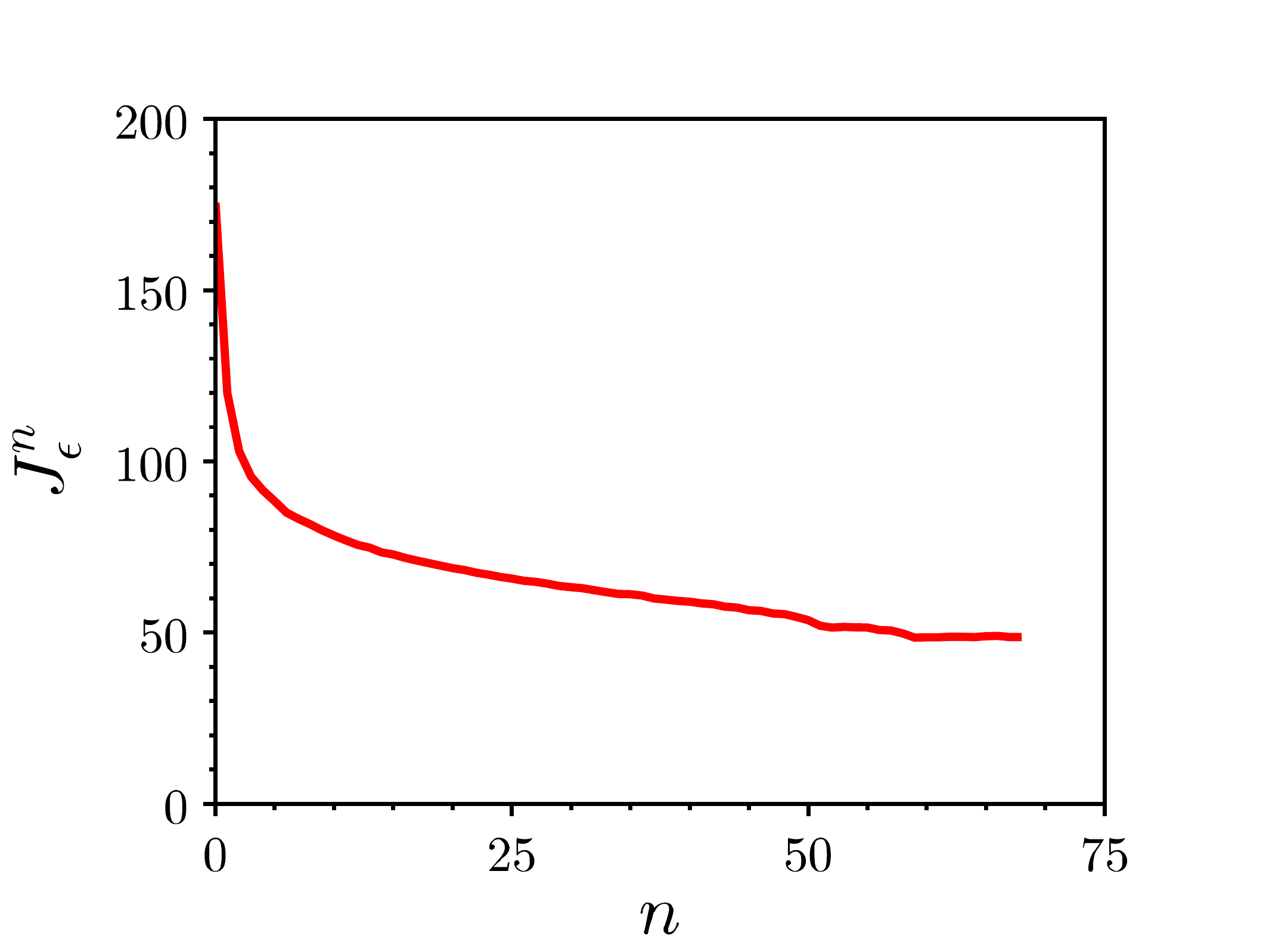} &
\includegraphics[width=0.48\textwidth]{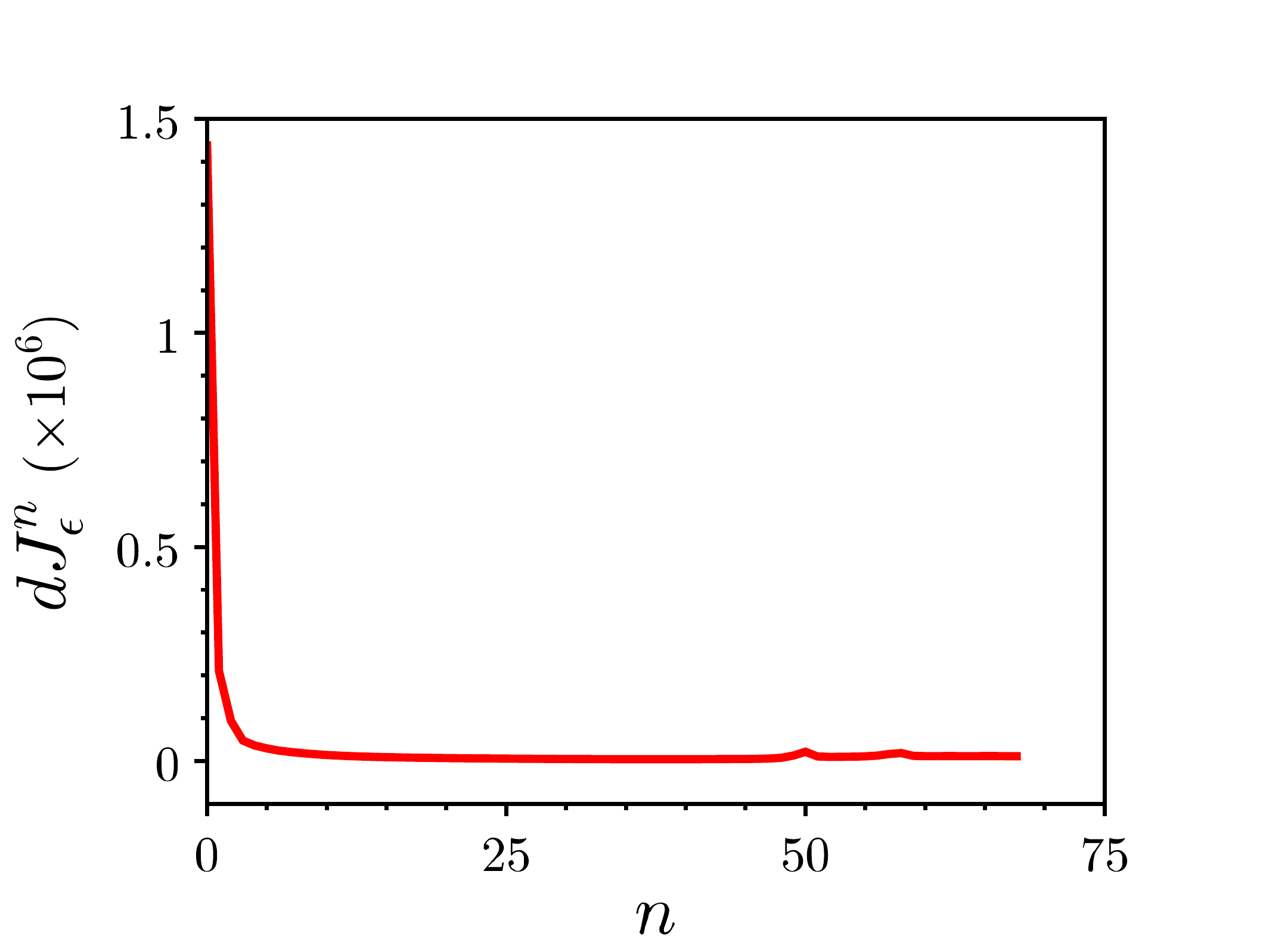}
\end{tabular}
\caption{{\bf (EX2)} -  Evolution of the cost $J_\epsilon^n$ (left) and the derivative $dJ_\epsilon^n$ (right) for the initial curve $\gamma_0 \equiv 1/2$.}
\label{fig:simu_EX2_J}
\end{figure}

\vspace{0.2cm}
\par\noindent
$\bullet$
We now consider the initial datum $(y_0,y_1)$ given by
\begin{equation}\label{eq:EX3}\tag{\bf EX3}
    y_0(x) = (10 x - 4)^2 (10 x - 6)^2 \mathbbm{1}_{[0.4, 0.6]}(x), \quad y_1(x) = 0, \qquad \text{for } x \in (0,1).
\end{equation}
This initial condition generates two travelling waves going in opposite directions, as can be seen in Figure~\ref{fig:simu_EX3_gmm}.3. For $T = 2$, $\epsilon = 10^{-2}$ and $\rho=10^{-4}$, we initialize the algorithm with the initial curve $\gamma_0 \equiv 1/2$. The convergence is observed after $111$ iterations leading to $J_\epsilon(\gamma_\text{opt}) \approx 41.02$. Moreover, the minimal cost for cylindrical domains is $\min_{x_0} J_\epsilon(x_0) \approx 85.08$, so that the performance index is $\Pi(\gamma_\text{opt}) \approx 51.79$\%. Once again, our non-cylindrical setup is much more efficient than the cylindrical one. It is still due to the fact that the domains we consider can follow the propagation of the travelling waves, one after the other. This can be noticed in Figure~\ref{fig:simu_EX3_gmm}, where we display the optimal control domain, the corresponding adjoint state $\varphi$, the uncontrolled and the controlled wave over the optimal domain.

\begin{figure}[ht!]
\hspace{-0.5cm}
\centering
\begin{tabular}{cccc}
\includegraphics[width=0.23\textwidth]{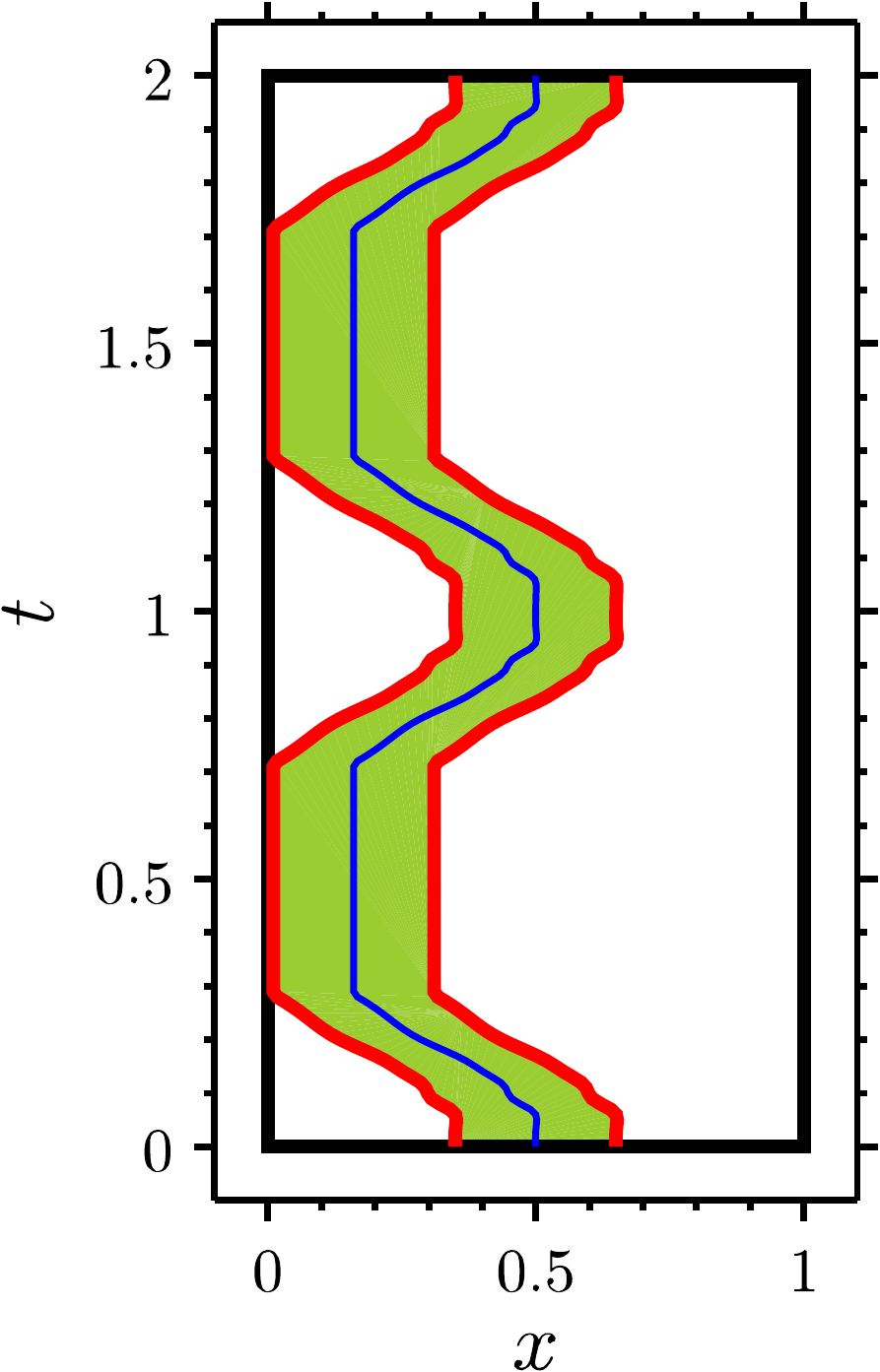} &
\includegraphics[width=0.23\textwidth]{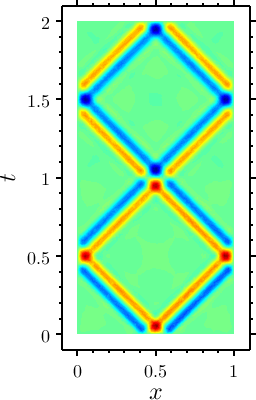} &
\includegraphics[width=0.23\textwidth]{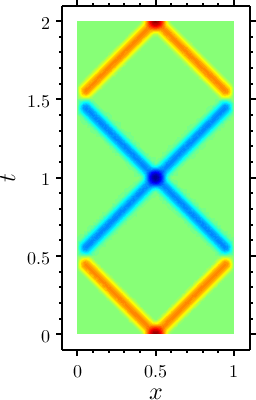} &
\includegraphics[width=0.23\textwidth]{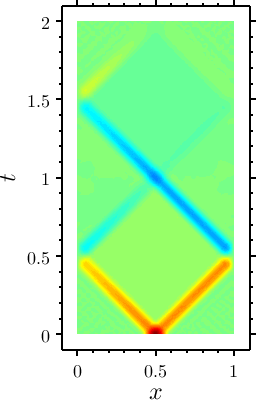}
\end{tabular}
\caption{{\bf (EX3)} -  From left to right, optimal control domain, isovalues of the corresponding adjoint state $\varphi$, isovalues of the uncontrolled and controlled wave over the optimal domain, for $T = 2$, for the initial curve $\gamma_0 \equiv 1/2$.}
\label{fig:simu_EX3_gmm}
\end{figure}

In order to show the influence of the controllability time on the optimal domain, for $\epsilon = 10^{-2}$ and $\gamma_0 \equiv 1/2$, we use the descent algorithm with $T = 1$ and $\rho=2.5\times 10^{-5}$, initialized with the curve $\gamma_0=1/2$. Remark that the corresponding domain do satisfies the geometric optic condition. The convergence is observed after $213$ iterations and the optimal cost is $J_\epsilon(\gamma_\text{opt}) \approx 94.78$. Moreover, the minimal cost for cylindrical domains is $\min_{x_0} J_\epsilon(x_0) \approx 183.98$, so that the performance index is $\Pi(\gamma_\text{opt}) \approx 48.48$\%. Observe that the cylindrical domains associated with $x_0\notin (0.25,0.75)$ do not verify the geometric optics condition. This highlights the necessity to use non-cylindrical domains. Compared to the simulation for $T = 2$, the optimal cost increases by a factor around $2.3$. Figure~\ref{fig:simu_EX3_T} displays the optimal control domain, the corresponding adjoint state $\varphi$, the uncontrolled and controlled wave over the optimal domain. We remark that the projection of the optimal domain on the $x$-axis covers the whole domain $\Omega$, in contrast with the domain associated with $T=2$. 

\begin{figure}[ht!]
\hspace{-0.5cm}
\centering
\begin{tabular}{cccc}
\includegraphics[width=0.23\textwidth]{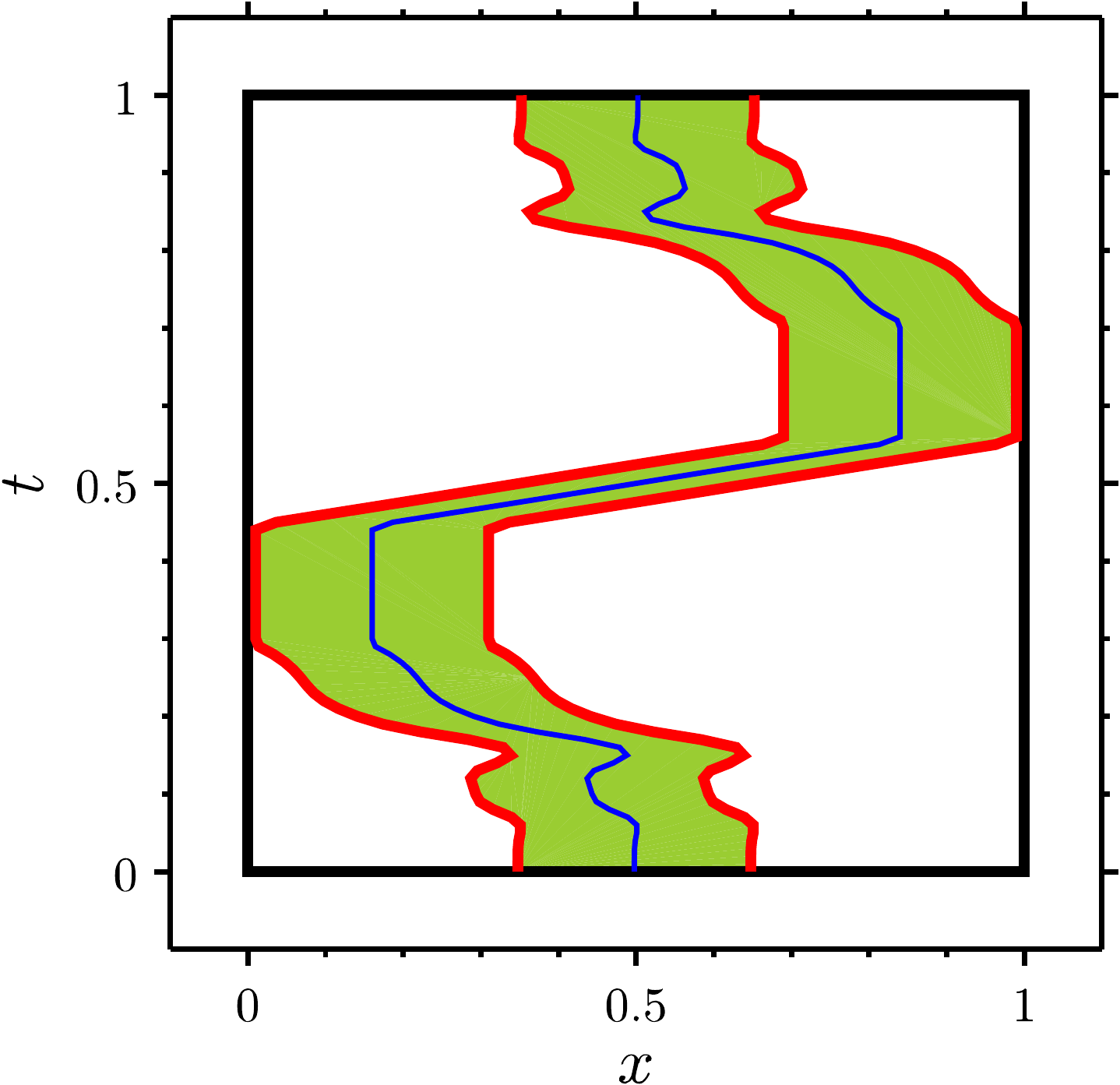} &
\includegraphics[width=0.23\textwidth]{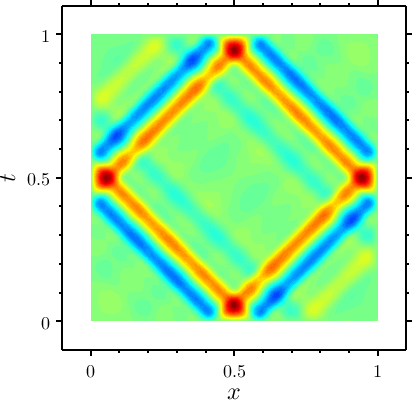} &
\includegraphics[width=0.23\textwidth]{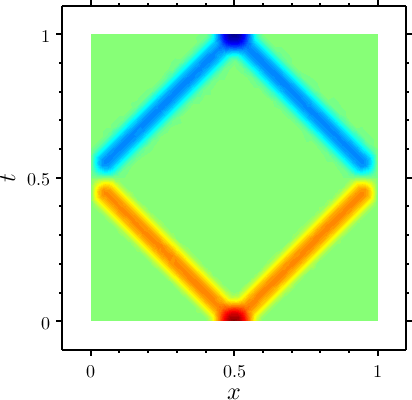} &
\includegraphics[width=0.23\textwidth]{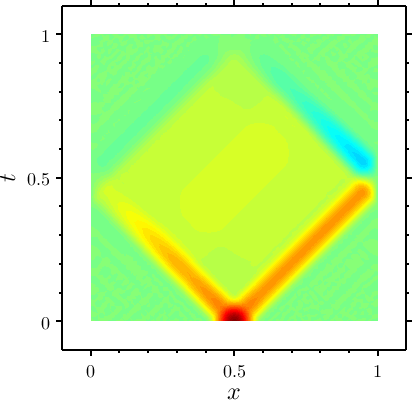}
\end{tabular}
\caption{{\bf (EX3)} -  From left to right, optimal control domain, isovalues of the corresponding adjoint state $\varphi$, isovalues of the uncontrolled and controlled wave over the optimal domain, for $T = 1$, for the initial curve $\gamma_0 \equiv 1/2$.}
\label{fig:simu_EX3_T}
\end{figure}

\vspace{0.2cm}
\par\noindent
$\bullet$ Eventually, in order to highlight the influence of the regularization parameter $\epsilon$ on the optimal domain, we now consider the initial datum $(y_0,y_1)$ given by
\begin{equation}\label{eq:EX4}\tag{\bf EX4}
    y_0(x) =
    \left\{
    \begin{array}{cl}
    3 x & \text{if } 0 \leq x \leq 1/3, \\
    3 (1 - 2 x) & \text{if } 1/3 \leq x \leq 2/3, \\
    -3 (1 - x) & \text{if } 2/3 \leq x \leq 1,
    \end{array}
    \right.
    \quad y_1(x) = 0, \qquad \text{for } x \in (0,1).
\end{equation}
For $T = 2$ and $\rho=10^{-5}$, we initialize the descent algorithm with the curve $\gamma \equiv 1/2$ and consider $\epsilon = 10^{-2}$ and $\epsilon = 0$. The numbers of iterations until convergence, the values of the functional $J_\epsilon$ evaluated at the optimal curve $\gamma_\text{opt}$ and the performance indices of $\gamma_\text{opt}$ are listed in Table~\ref{tab:simu_EX4}. For the initial datum~\eqref{eq:EX4}, the minimal cost for cylindrical domains is $\min_{x_0} J_\epsilon(x_0) \approx 47.71$.

\begin{table}[ht!]
\centering
\begin{tabular}{c|cc}
$\epsilon$ & $0$ & $10^{-2}$ \\
\hline
Number of iterations & $247$ & $389$ \\
$J_\epsilon(\gamma_\text{opt})$ & $60.35$ & $43.23$ \\
$\Pi(\gamma_\text{opt})$ & $-26.51$\% & $9.38$\% 
\end{tabular}
\caption{{\bf (EX4)} -  Number of iterations, optimal value of the functional $J_\epsilon$ and performance index, for $\epsilon \in \{ 0, 10^{-2} \}$, for the initial curve $\gamma_0 \equiv 1/2$.}
\label{tab:simu_EX4}
\end{table}

In Figure~\ref{fig:simu_EX4_gmm}, we clearly see the regularizing effect of $\epsilon$ and the need of regularization in this case, as the optimal domain obtained when $\varepsilon = 0$ is very oscillating.

\begin{figure}[ht!]
\hspace{-0.5cm}
\centering
\begin{tabular}{cccc}
\includegraphics[width=0.23\textwidth]{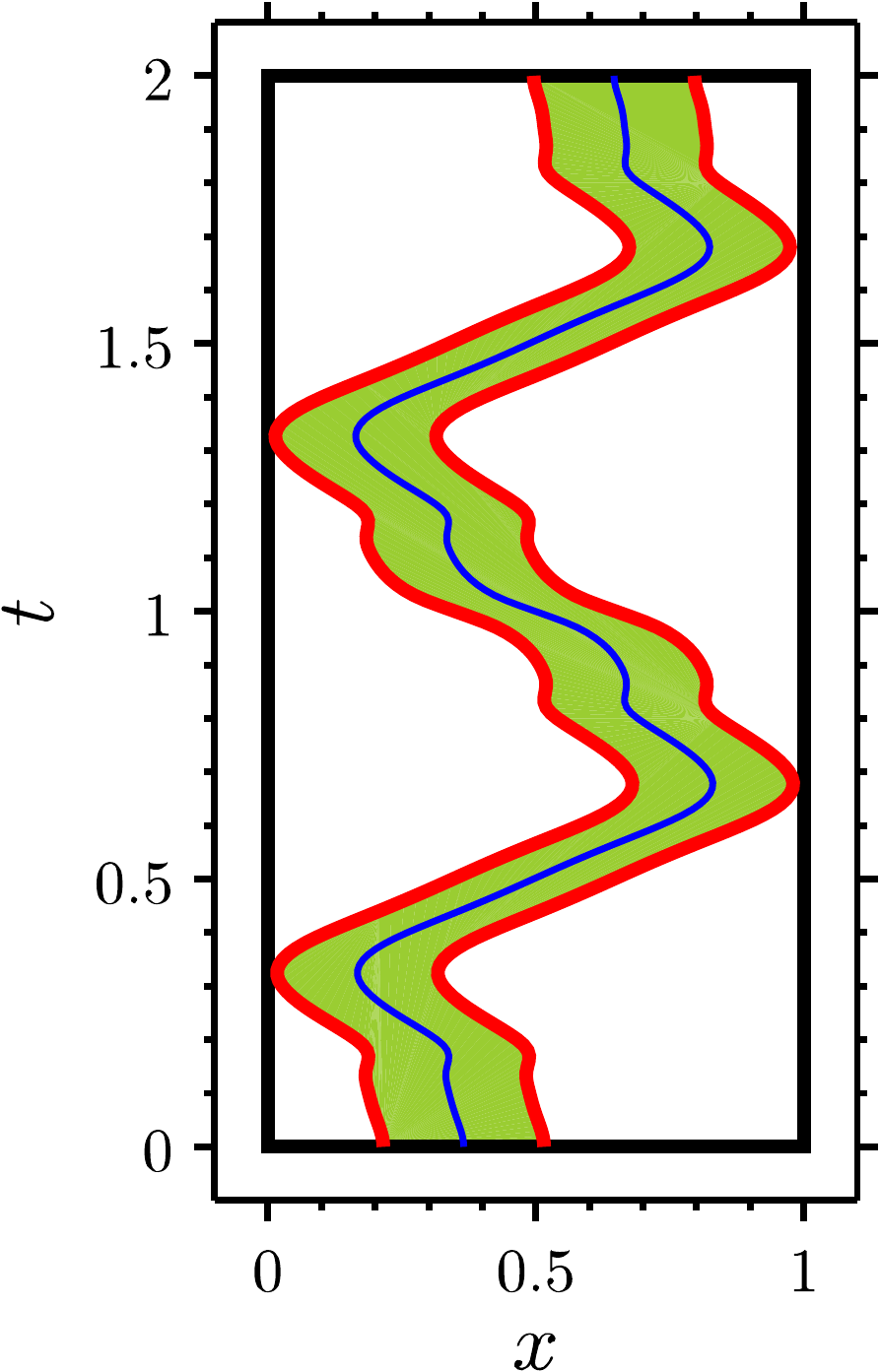} &
\includegraphics[width=0.23\textwidth]{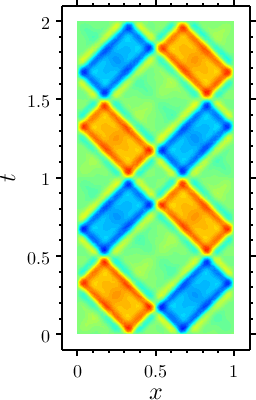} &
\includegraphics[width=0.23\textwidth]{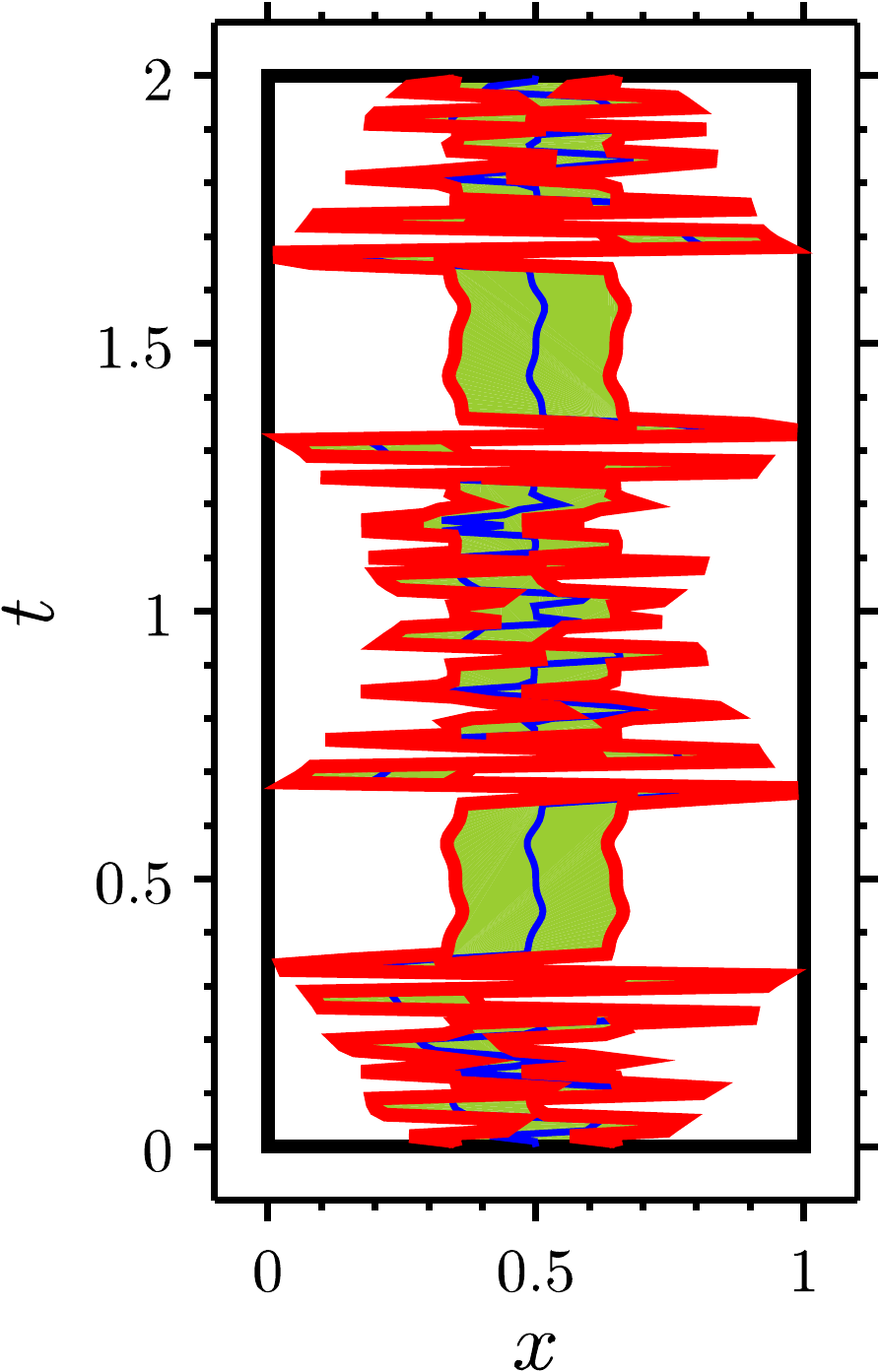} &
\includegraphics[width=0.23\textwidth]{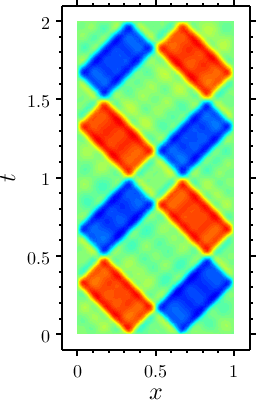}
\end{tabular}
\caption{{\bf (EX4)} - Optimal control domain and isovalues of the adjoint state for $\epsilon = 10^{-2}$ (left), $\epsilon = 0$ (right), for the initial curve $\gamma_0 \equiv 1/2$.}
\label{fig:simu_EX4_gmm}
\end{figure}

\subsection{Iterative approximation of the observability constant}
\label{sect:Cobs_num}

In this last part, we formally describe and use an algorithm allowing to approximate the observability constant appearing in~\eqref{eq:obs_W}, associated to any domain $q\subset Q_T$. The algorithm is based on the following characterization:
\begin{equation}\label{eq:Cobs_char}
    C_\text{obs}(q) = \sup_{\mathbf{y}_0 \in \mathbf{V}} \frac{\langle R \Lambda_q \mathbf{y}_0, \mathbf{y}_0 \rangle_\mathbf{V}}{\| \mathbf{y}_0 \|_\mathbf{V}^2}.
\end{equation}
where $\Lambda_q$ and $R$ are respectively the control operator associated to the domain $q$ and the duality operator between the space $\mathbf{W}$ and $\mathbf{V}$:
\begin{equation}\label{eq:Cobs_op}
    \Lambda_q :
    \left\{
    \begin{array}{ccc}
        \mathbf{V} & \to & \mathbf{W} \\
        \mathbf{y}_0 & \mapsto & \widehat{\boldsymbol{\varphi}}_0
    \end{array}
    \right.,
    \qquad R :
    \left\{
    \begin{array}{ccc}
        \mathbf{W} & \to & \mathbf{V} \\
        (\varphi_0, \varphi_1) & \mapsto & ((-\partial^2_x)^{-1} \varphi_1, -\varphi_0)
    \end{array}
    \right..
\end{equation}
In the definition of $\Lambda_q$, $\widehat{\boldsymbol{\varphi}}_0 \in \mathbf{W}$ is the minimum of the functional $\mathcal{J}^\star$ (cf.~\eqref{eq:J*}) associated to $\mathbf{y}_0 \in \mathbf{V}$. The characterization~\eqref{eq:Cobs_char} can be obtained by following the steps of~\cite[Section~2]{Munch18} and~\cite[Remark~2.98]{Coron07}. The main consequence of this characterization is that $C_\text{obs}(q)$ can be viewed as the largest eigenvalue of the operator $R \Lambda_q$ in $\mathbf{V}$. Consequently, we can formally adapt the power iteration method to our infinite-dimensional setting. The algorithm reads as follows. Let $\mathbf{y}_0^0 \in \mathbf{V}$ be given such that $\| \mathbf{y}_0^0 \|_\mathbf{V} = 1$. For $n \geq 0$, using the space-time finite element method described in~\cite[Section~3-4]{CastroCindeaMunch14}, we compute $\widehat{\boldsymbol{\varphi}}_0^n = \Lambda_q \mathbf{y}_0^n$ then set $\mathbf{z}_0^n = R \widehat{\boldsymbol{\varphi}}_0^n$ and $\mathbf{y}_0^{n+1} = \mathbf{z}_0^n/\| \mathbf{z}_0^n \|_\mathbf{V}$. We finally have $C_\text{obs}(q) = \lim_{n \to \infty} \| \mathbf{z}_0^n \|_\mathbf{V}$ while $\mathbf{y}_0^n$ converges in $\mathbf{V}$ to the most expensive initial datum to control.
For the control domain of Figure~\ref{fig:RN_q_ex}, this algorithm (after an appropriate space-time), initialized with $\mathbf{y}_0^0=K (x(1-x),0)$ - $K$ such that $\| \mathbf{y}_0^0 \|_\mathbf{V} = 1$-, produces the following sequence $\{\| \mathbf{z}_0^n \|_\mathbf{V}\}_{n\geq 0}=\{2.6895, 3.829, 3.981, 3.994, 3.997, \cdots\}$ converging toward the value $4$, in agreement with the result of Section~\ref{ssect:unif_obs_ex} based on a graph argument. The most expensive initial datum to control is displayed in Figure~\ref{fig:Cobs_PI}. Remark that the initial datum solution of \eqref{eq:Cobs_char} is not unique. 

\begin{figure}[ht!]
\centering
\begin{tabular}{cc}
\includegraphics[width=0.45\textwidth]{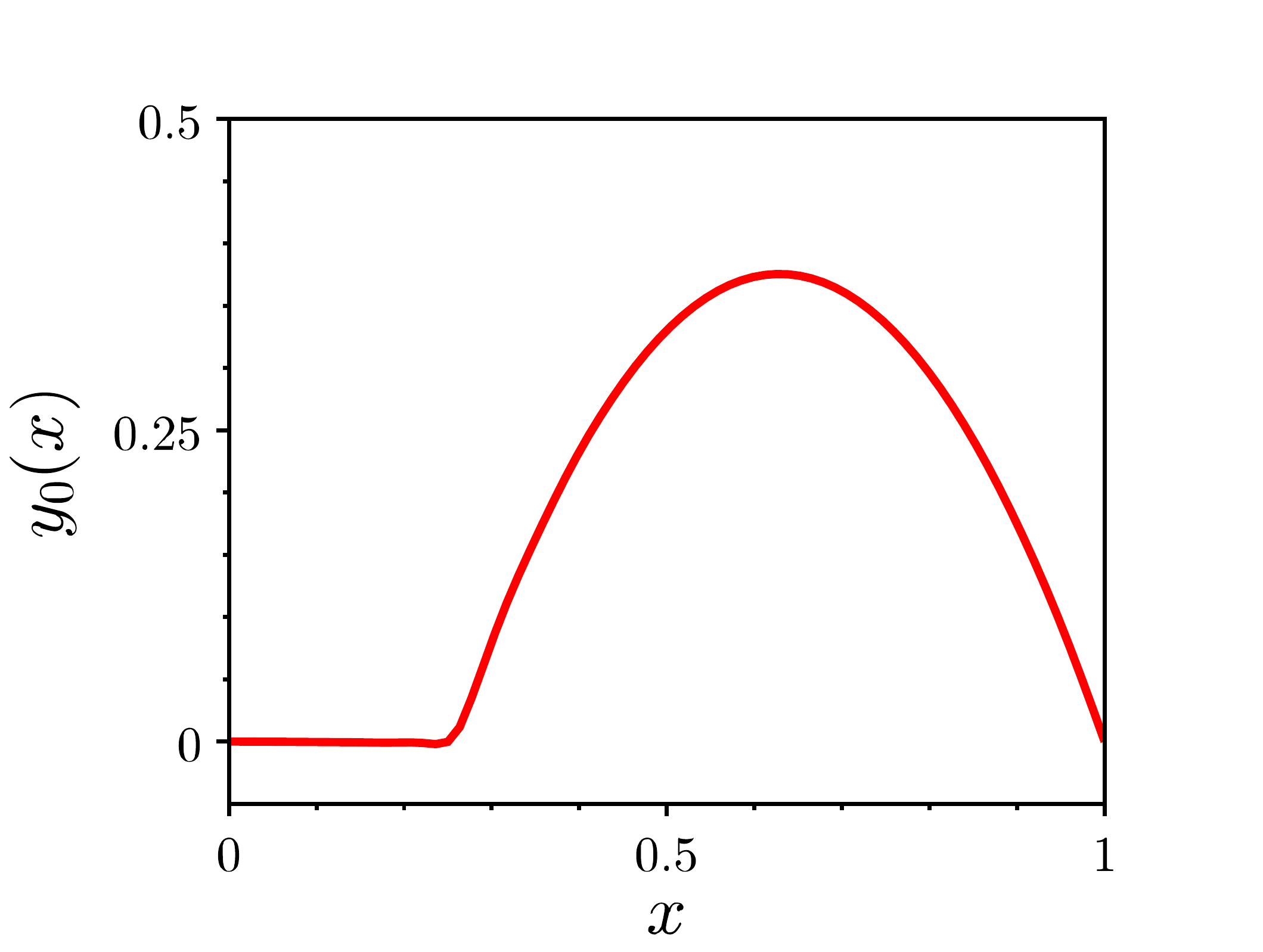} &
\includegraphics[width=0.45\textwidth]{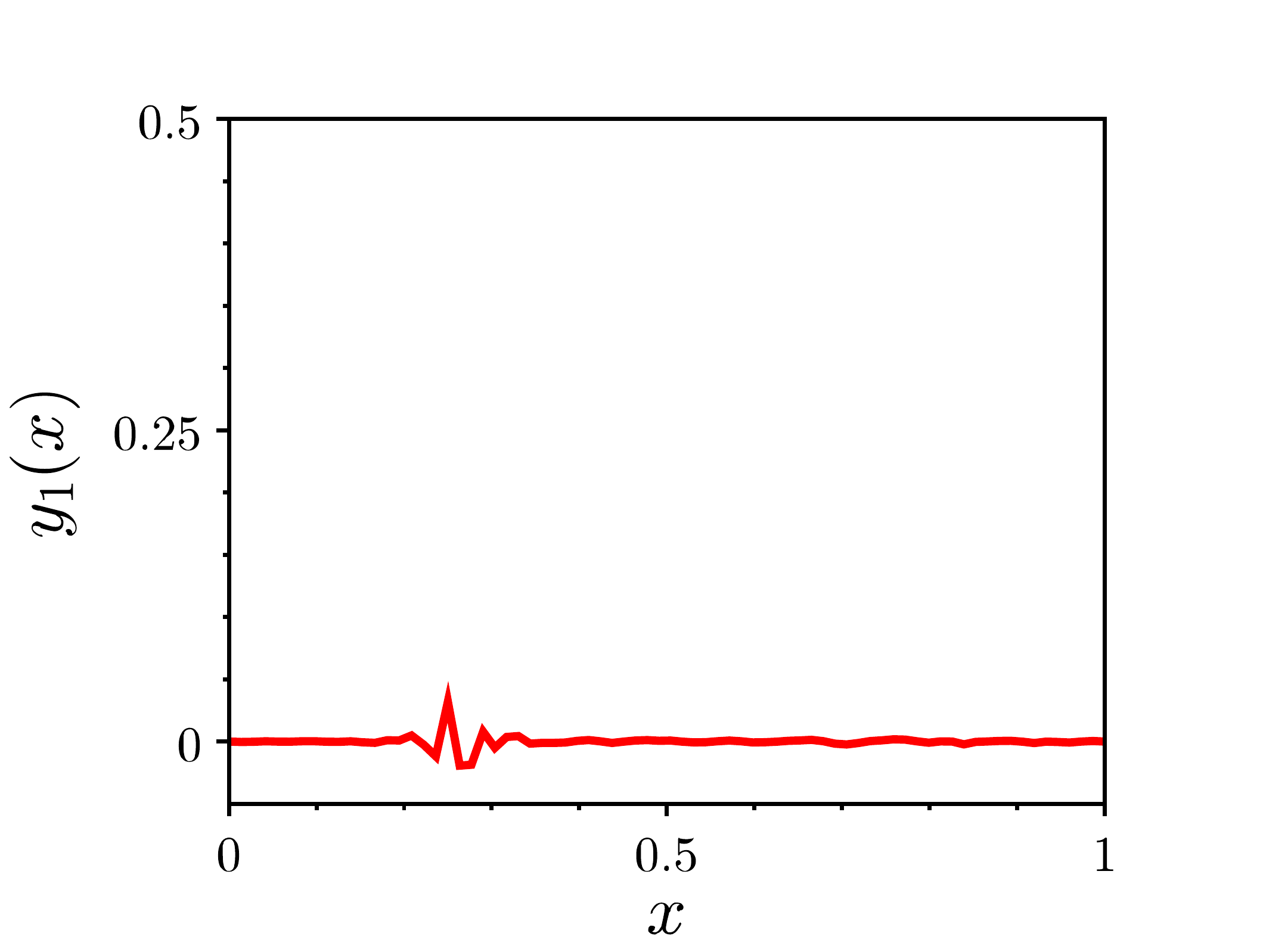}
\end{tabular}
\caption{Most expensive initial data $y_0$ (left) and $y_1$ (right) to be controlled.}
\label{fig:Cobs_PI}
\end{figure}

\section{Conclusion and perspectives}
\label{sect:conclusion}

Making use of the d'Alembert formulae for the solutions of the one dimensional wave equation, we have shown a uniform observability inequality with respect to the class of non cylindrical domains satisfying the geometric optics condition. The proof based on arguments from graph theory allows notably to relate the value of the observability constant to the spectrum of the Laplacian matrix, defined in term of the graph of any domain $q \subset Q_T$. The uniform observability property then allows to consider and analyze the problem of the control's optimal support associated to fixed initial conditions. For simplicity, the optimization is made over connected domains defined by regular curves. As expected, the optimal domains (approximated within a space-time finite element method) are closely related to the travelling waves generated by the initial conditions.

This work may be extended in several directions. First, the characterization of the observability constant in term of a computable eigenvalue problem in Section~\ref{sect:Cobs_num} may allow to consider the optimization of such constant with respect to the domain of observation, \emph{i.e.} $\inf_{q\in \mathcal{Q}_\text{ad}^\varepsilon} C_\text{obs}(q)$. Moreover, from an approximation point of view, we may also consider more general domains (than connected ones) and use, for instance, a level set method to describe the geometry (as done in~\cite{Munch08}). Eventually, this work may be adapted to the case of controls supported on single curves of $Q_T$, using the uniform observability property given in~\cite{Castro13}.

The extension of this work to the $N$-dimensional case studied in~\cite{Lebeau, Shao} is also a challenge. 

\bibliographystyle{siam}


\end{document}